\documentclass[reqno, 11pt]{amsart}

\usepackage{amsmath}
\usepackage{amsthm}
\usepackage{amsfonts}
\usepackage{amssymb}
\usepackage{mathrsfs}
\usepackage{graphicx}

\newcommand{\mc}{\mathcal}

\renewcommand{\Re}{\mathrm{Re}\,}
\renewcommand{\Im}{\mathrm{Im}\,}

\newcommand{\N}{\mathbb{N}}
\newcommand{\R}{\mathbb{R}}
\newcommand{\C}{\mathbb{C}}
\newcommand{\Z}{\mathbb{Z}}

\newcommand{\supp}{\mathrm{supp}\,}

\newcommand{\ad}{\mathrm{ad}}

\newcommand \be{\begin{equation}}
\newcommand \ee{\end{equation}}

\newtheorem{lemma}{Lemma}[section]
\newtheorem{proposition}[lemma]{Proposition}
\newtheorem{theorem}[lemma]{Theorem}
\newtheorem{corollary}[lemma]{Corollary}
\theoremstyle{remark}
\newtheorem{remark}[lemma]{Remark}

\theoremstyle{definition}
\newtheorem{definition}[lemma]{Definition}
\newtheorem{hypothesis}{Hypothesis}

\numberwithin{equation}{section}

\author{Roland Donninger}
\address{\'Ecole Polytechnique F\'ed\'erale de Lausanne, 
Department of Mathematics, Station 8, CH-1015 Lausanne, Switzerland}
\email{roland.donninger@epfl.ch}

\author{Joachim Krieger}
\address{\'Ecole Polytechnique F\'ed\'erale de Lausanne, 
Department of Mathematics, Station 8, CH-1015 Lausanne, Switzerland}
\email{joachim.krieger@epfl.ch}

\title[Vector field method]{A vector field method on the distorted Fourier side and decay for wave equations with potentials}

\begin{document}
\begin{abstract}We study the Cauchy problem for the one-dimensional wave equation
\[ \partial_t^2 u(t,x)-\partial_x^2 u(t,x)+V(x)u(t,x)=0. \]
The potential $V$ is assumed to be smooth with asymptotic behavior 
\[ V(x)\sim -\tfrac14 |x|^{-2}\mbox{ as } |x|\to \infty. \]
We derive dispersive estimates, energy estimates, and estimates involving the scaling
vector field $t\partial_t+x\partial_x$,
where the latter are obtained by employing a vector field method on the ``distorted''
Fourier side.
In addition, we prove local energy decay estimates.
Our results have immediate applications in the context of geometric evolution problems.
The theory developed in this paper is fundamental for the proof of the co-dimension 1 stability of the catenoid
under the vanishing mean curvature flow in Minkowski space, see \cite{DonKriSzeWon13}.
\end{abstract}

\maketitle

\tableofcontents

\section{Introduction}
\noindent There is a growing interest in decay estimates for wave equations with potentials.
Such equations arise, for instance, as linearizations about nontrivial static solutions 
of nonlinear wave equations or in the context of geometric evolution problems.
Consequently, they are of great interest to mathematical physics, general relativity,
and geometry.
In this paper we study the Cauchy problem
\begin{equation}
\label{eq:main}
\left \{ \begin{array}{l}
\partial_t^2 u(t,x)-\partial_x^2 u(t,x)+V(x)u(t,x)=0,\quad t>0 \\
u(0,x)=f(x),\quad \partial_t u(t,x)|_{t=0}=g(x) \end{array} \right .
\end{equation}
where $f, g$ are prescribed functions (the initial data) and for simplicity
we restrict ourselves to the half-line case $x\geq 0$ with the Neumann condition
$\partial_x u(t,x)|_{x=0}=0$ for all $t$.
Needless to say that our methods carry over to the Dirichlet and full-line
cases as well.
Throughout we make the following assumptions on the potential.
\begin{hypothesis}
\label{hyp:A}\hfill
\begin{itemize}
\item $V\in C^\infty([0,\infty))$.
\item We have the asymptotics
\[ V(x)=-\tfrac14 x^{-2}[1+O(x^{-\alpha})],\quad x\geq 1 \] 
where $\alpha>0$ and
the $O$-term is of symbol type, i.e., $\partial_x^k O(x^{-\alpha})=O(x^{-\alpha-k})$
for all $x\geq 1$ and each $k\in \N_0$.
\item The potential $V$ does not admit bound states, i.e., there do not exist nontrivial
$f\in L^2((0,\infty))\cap C^\infty([0,\infty))$ with $f'(0)=0$ and 
\[ -f''+V f=E f \] for some $E<0$.
\end{itemize}
\end{hypothesis}
We remark that the last condition on the nonexistence of bound states is not a real restriction
since one may always orthogonally project the evolution to the continuous spectral subspace.
Our results then hold for the projected evolution.
Furthermore, the smoothness assumption on the potential may be relaxed considerably provided one modifies
the statements accordingly, but we do not elaborate on this.

It is well-known that inverse square decay of the potential is in some sense
critical for the spectral and scattering theory.
This is of course closely related to the fact that the angular momentum barrier $\frac{\ell(\ell+1)}{r^2}$
has exactly this type of decay as $r\to\infty$.
As a consequence, the decay of the associated wave evolution depends on the \emph{coefficient}
in front of the $x^{-2}$-term in the asymptotics of $V(x)$, cf.~\cite{SchSofSta10, SchSofSta10a, DonSchSof11}.
This is in stark contrast to faster decaying potentials where the decay of the associated wave flow
depends only on the dominant power in the asymptotic expansion of $V$ whereas the coefficient
is completely irrelevant \cite{DonSch10, CosHua12}.
In addition, the particular coefficient $-\frac14$ we are considering is also critical
in the sense that it leads to delicate logarithmic corrections in the asymptotics
of the spectral measure.
This is related to the fact that the asymptotics $-\frac14 x^{-2}$
are typical for 
two-dimensional problems.
To see this, rewrite the radial two-dimensional wave equation
\[ \partial_t^2 v(t,r)-\partial_r^2 v(t,r)-\tfrac{1}{r}\partial_r v(t,r)=0 \]
in terms of the variable $u(t,r):=r^\frac12 v(t,r)$.
This yields
\[ \partial_t^2 u(t,r)-\partial_r^2 u(t,r)-\tfrac{1}{4r^2}u(t,r)=0. \]
Of course, the free case is not compatible with Hypothesis \ref{hyp:A} due to the singularity
of $-\frac{1}{4r^2}$ at the origin $r=0$.
However, the study of wave equations on a large class of two-dimensional manifolds 
leads to potentials as in Hypothesis \ref{hyp:A}, see \cite{SchSofSta10}.
In fact, the main application we have in mind is the hyperbolic vanishing mean curvature flow as
studied in \cite{KriLin12}.
In this context, a wave equation with a potential satisfying Hypothesis \ref{hyp:A} (and, in fact,
all of the more restrictive assumptions we impose below) arises
by linearizing the evolution equation about the catenoid, a minimal surface which
is an explicit static solution of the system.
Consequently, the present paper develops the linear theory which is crucial for the proof
of the co-dimension 1 stability of the catenoid \cite{DonKriSzeWon13}.

\subsection{Main results}

In this paper we prove the following estimates for the solution of the Cauchy problem Eq.~\eqref{eq:main}.
Throughout we restrict ourselves to the half-line case $x\geq 0$ with a Neumann condition at the origin and
we write $\R_+:=[0,\infty)$ as well as $\langle x\rangle:=\sqrt{1+|x|^2}$.

\begin{theorem}[Basic dispersive estimate]
\label{thm:basicdisp}
Assume that Hypothesis \ref{hyp:A} holds. Then there exists a constant $C>0$ such that the solution $u$ of Eq.~\eqref{eq:main}
satisfies the estimate
\begin{align*} 
\left \|\langle \cdot \rangle^{-\frac12}u(t,\cdot)\right \|_{L^\infty(\R_+)}
\leq C\langle t\rangle^{-\frac12}\Big (&\left \|\langle \cdot \rangle^\frac12 f'\right \|_{L^1(\R_+)}
+\left \|\langle \cdot \rangle^\frac12 f\right \|_{L^1(\R_+)} \\
&+\left \|\langle \cdot \rangle^\frac12 g\right \|_{L^1(\R_+)} \Big )
\end{align*}
for all $t \geq 0$ and all $f,g \in \mc S(\R_+)$.
\end{theorem}

Theorem \ref{thm:basicdisp} is the analogue of the well-known dispersive estimate
for the $2$-dimensional free wave equation.
We remark that this estimate can also be extracted from \cite{SchSofSta10}.
In addition, we have the following version of improved decay.

\begin{theorem}[Accelerated decay]
\label{thm:accdec}
Under the assumptions of Theorem \ref{thm:basicdisp} we have the bound
\begin{align*} 
&\left \|\langle \cdot \rangle^{-1}u(t,\cdot)\right \|_{L^\infty(\R_+)} \\
&\quad \leq C\langle t\rangle^{-1}\Big (\left \|\langle \cdot \rangle f'\right \|_{L^1(\R_+)}
+\left \|\langle \cdot \rangle f\right \|_{L^1(\R_+)}+\left \|\langle \cdot \rangle g\right \|_{L^1(\R_+)} \Big )
\end{align*}
for all $t \geq 0$ and all $f,g \in \mc S(\R_+)$.
\end{theorem}

The next result concerns energy bounds.
Clearly, the wave equation \eqref{eq:main} has the conserved energy
\begin{equation}
\label{eq:energy} 
\tfrac12 \int_0^\infty \left (u_t(t,x)^2+u_x(t,x)^2+V(x)u(t,x)^2 \right ) dx. 
\end{equation}
For the following we write $Af:=-f''+V f$.
By Hypothesis \ref{hyp:A}, the operator $A$ is self-adjoint on $L^2(\R_+)$ and nonnegative.
In particular, $\sqrt A$ is defined by means of the functional calculus for $A$.
The conserved energy can be written as
\[ \tfrac12 \left (\|\partial_t u(t,\cdot)\|_{L^2(\R_+)}^2+\|\sqrt A u(t,\cdot)\|_{L^2(\R_+)}^2 \right ). \]
Similar to the free case, we obtain higher energy bounds.
These, however, require slightly stronger assumptions on the potential.

\begin{hypothesis}
\label{hyp:B}
In addition to Hypothesis \ref{hyp:A} we assume that
\[ V^{(2j+1)}(0)=0 \]
for all $j\in \N_0$, i.e., $V$ extends to an even function $V\in C^\infty(\R)$.
\end{hypothesis}

\begin{theorem}[Energy bounds]
\label{thm:energy}
Assume that Hypothesis \ref{hyp:B} holds.
Then there exist constants $C_{k,\ell}>0$ such that the solution $u$ of Eq.~\eqref{eq:main}
satisfies
\begin{align*} 
\left \|\partial_t^\ell \sqrt A u(t,\cdot)\right \|_{H^k(\R_+)}&\leq C_{k,\ell}\left [
\|\sqrt A f\|_{H^{k+\ell}(\R_+)}+\|g\|_{H^{k+\ell}(\R_+)} \right ] \\
\left \|\partial_t^\ell \partial_t u(t,\cdot)\right \|_{H^k(\R_+)}&\leq C_{k,\ell}\left [
\|\sqrt Af\|_{H^{k+\ell}(\R_+)}+\|g\|_{H^{k+\ell}(\R_+)} \right ] 
\end{align*}
for all $t\geq 0$, $k,\ell \in \N_0$, and all even $f,g \in \mc S(\R)$.
\end{theorem}

\begin{remark}
In view of nonlinear applications it is desirable to replace the nonlocal operator $\sqrt A$
by the ordinary derivative $\nabla$.
This can indeed be done at the expense of requiring an $L^1$-bound on the function $g$ and
one obtains the estimate
\[ \|\nabla u(t,\cdot)\|_{H^k(\R_+)}\leq C_{k} \left (
\|f\|_{H^{1+k}(\R_+)}+\|g\|_{H^{k}(\R_+)}+\|g\|_{L^1(\R_+)} \right ), \]
see Lemma \ref{lem:energyfree}.
\end{remark}

The final set of estimates involves the scaling vector field $S=t\partial_t+x\partial_x$.
Since the pioneering works of Morawetz and Klainerman, vector field methods have become a standard
device in the study of nonlinear wave equations.
A vector field is useful if it has a ``nice'' commutator with the wave operator $\Box=-\partial_t^2+\partial_x^2$.
In general, for the free wave equation one obtains a sufficiently large set of suitable vector fields by considering
the generators of the symmetries (and conformal symmetries) of Minkowski space \cite{Kla85}.
One particular example of this kind is the scaling vector field $S$ which obeys the 
commutator relation $[S,\Box]=-2\Box$.
In the presence of a potential, the situation becomes more complicated.
It should be remarked, however, that for special potentials (which arise, for instance, in the study of
wave equations on black hole backgrounds) one may still be able
to obtain suitable vector fields by geometric considerations, see e.g.~\cite{AndBlu09, 
DafRod09, DafRod10, Luk10, Luk12}.
For general potentials the situation is less clear.
In the following we restrict ourselves to the scaling vector field which, in conjunction with
local energy decay, has proven very useful for the study of nonlinear problems.
To illustrate the difficulties one encounters, let us naively hit the equation $-\Box u+V u=0$ with $S$
which yields
\[ -\Box Su+VSu=U u \]
where $U(x)=-2V(x)-xV'(x)$.
By Duhamel's formula, $Su$ is given by\footnote{As usual, the operators $\cos(t\sqrt A)$ and $\frac{\sin(t\sqrt A)}{\sqrt A}$
are defined in terms of the functional calculus for the self-adjoint Schr\"odinger operator
$Af=-f''+Vf$, see below.}
\begin{align*}
 Su(t,\cdot)=&\cos(t\sqrt A)Su(0,\cdot)+\frac{\sin(t\sqrt A)}{\sqrt A}\partial_t Su(t,\cdot)|_{t=0} \\
&+\int_0^t \frac{\sin((t-s)\sqrt A)}{\sqrt A}(Uu(s,\cdot))ds
\end{align*}
where $Af=-f''+Vf$.
However, the time decay of $u$ is not strong enough to render the integral
convergent as $t\to\infty$.
This shows that such a naive approach cannot work.
Consequently, we employ a more subtle method where we use a representation of 
$S$ on the distorted Fourier side.
This allows us to conclude suitable bounds involving the vector field $S$ 
in the
presence of a general potential satisfying Hypothesis \ref{hyp:B}.
To be more precise, we have to require an additional nonresonance condition on $V$.

\begin{definition}
The potential $V$ is called \emph{resonant}\footnote{As we will prove (see Lemma \ref{lem:en0} below), 
there exists a fundamental system $\{f_0,f_1\}$ for the equation $-f''+V f=0$ with asymptotic behavior
\[ f_0(x)\sim x^\frac12,\quad f_1(x)\sim x^\frac12 \log x \]
as $x\to \infty$. Consequently, the resonant case is exceptional and the nonresonant
case is generic.} if there exists a nontrivial function $f$ which
satisfies $-f''+Vf=0$, $f'(0)=0$, and $f(x)=O(x^\frac12)$ as $x\to \infty$.
Otherwise, it is called \emph{nonresonant}.
\end{definition}

\begin{theorem}[Vector field bounds]
\label{thm:vf}
Suppose Hypothesis \ref{hyp:B} holds true and the potential $V$ is nonresonant.
Furthermore, set $Df(x):=xf'(x)$.
Then there exist constants $C_{k,\ell,m}>0$ such that the solution $u$ of Eq.~\eqref{eq:main}
satisfies the bounds
\begin{align*} 
\left \|\partial_t^\ell \sqrt A S^m u(t,\cdot)\right \|_{H^k(\R_+)}&\leq C_{k,\ell,m}
\sum_{j=0}^m \Big [
\|\sqrt A D^j f\|_{H^{k+\ell}(\R_+)}+\|D^j g\|_{H^{k+\ell}(\R_+)} \Big ] \\
\left \|\partial_t^\ell \partial_t S^mu(t,\cdot)\right \|_{H^k(\R_+)}&\leq C_{k,\ell,m}\sum_{j=0}^m \Big [
\|\sqrt A D^j f\|_{H^{k+\ell}(\R_+)}+\|D^j g\|_{H^{k+\ell}(\R_+)} \Big ]
\end{align*}
for all $t\geq 0$, $k,\ell,m \in \N_0$, and all $f,g \in \mc S(\R)$ even.
\end{theorem}

\begin{remark}
As before with the energy bounds, it is possible to replace $\sqrt A$ by $\nabla$, see
Lemmas \ref{lem:vfn} and \ref{lem:vfodhigh}.
\end{remark}

\begin{remark}
We emphasize that the nonresonance condition on the potential seems necessary here.
The point is that upon interchanging the order of $\sqrt A$ with $S$ one encounters a commutator
$[\sqrt A,E]$ where $E$ is a certain nonlocal error operator which measures the deviation
from the free case $V=0$, see Lemma \ref{lem:commAB}.
In order to prove suitable mapping properties for this commutator, we use an additional
smoothing property of the spectral measure at small frequencies which is only available
in the nonresonant case. This problem does not occur for the simpler energy estimates
which is why Theorem \ref{thm:energy} does not require this additional
condition.
\end{remark}

Finally, we also prove the following bounds (local energy decay) involving the scaling
vector field.
We use the notation
\[ \|u(t,x)\|_{L^p_t(I) H^k_x(J)}:=\left (\int_I \|u(t,\cdot)\|_{H^k(J)}^p dt \right )^{1/p} \]
for intervals $I,J \subset \R$.
We emphasize that the nonresonance condition is not needed for the following estimates.

\begin{theorem}[Local energy decay]
\label{thm:locen}
Suppose $V$ satisfies Hypothesis \ref{hyp:B} and set $Df(x):=xf'(x)$.
Then there exist constants $C_{k,\ell,m}>0$ such that 
\begin{align*}
\left \|\langle x \rangle^{-\frac12-}\partial_t^\ell S^m e^{\pm it\sqrt A}f(x)\right \|_{L^2_t(\R_+) H_x^{k}(\R_+)}
&\leq C_{k,\ell,m}\sum_{j=0}^m \left \|D^j f\right \|_{H^{k+\ell}(\R_+)} \\
\left \|\langle x \rangle^{-1}
S^m \frac{\sin(t\sqrt A)}{\sqrt A}g(x)\right \|_{L^2_t(\R_+) L^2_x(\R_+)}
&\leq C_{k,m}\sum_{j=0}^m \|\langle \cdot \rangle^{\frac12+}D^j g\|_{L^2(\R_+)} \\
\left \|\langle x \rangle^{-1}
S^m \frac{\sin(t\sqrt A)}{\sqrt A}g(x)\right \|_{L^2_t(\R_+) H^{1+k}_x(\R_+)}
&\leq C_{k,m}\sum_{j=0}^m \|\langle \cdot \rangle^{\frac12+}D^j g\|_{H^{k}(\R_+)} 
\end{align*}
for all $k,\ell,m\in \N_0$ and all $f,g\in \mc S(\R)$ even.
\end{theorem}

Via Duhamel's principle, corresponding bounds hold for the inhomogeneous problem
\begin{equation}
\label{eq:maininhom}
 \left \{ \begin{array}{l}\partial_t^2 u(t,x)-\partial_x^2 u(t,x)+V(x)u(t,x)=f(t,x),\quad t\geq 0 \\
u(0,\cdot)=\partial_t u(t,\cdot)|_{t=0}=0 \end{array} \right . 
\end{equation}
where $f: [0,\infty)\times \R_+ \to \R$ is prescribed.
The analogue of Theorem \ref{thm:vf} for the inhomogeneous problem
is stated explicitly in Lemmas \ref{lem:vfinhom} and \ref{lem:vfinhomfree} below and for
the inhomogeneous version of the local energy decay we refer to Lemma \ref{lem:loceninhnabla}
and Remark \ref{rem:loceninh}.

In a final section we also consider the case when the initial data are in divergence form, i.e., if
$u_t(t,x)|_{t=0}=\tilde g'(x)$ for some function $\tilde g$,
where slight improvements can be obtained, see Lemmas \ref{lem:div}, \ref{lem:divvf},
\ref{lem:divinhx}, \ref{lem:divinhs} below.

\subsection{Related work}

While the problem of wave decay in the presence of potentials has a long history
in the physics literature, mathematically rigorous treatments are surprisingly recent
and there are still many open questions.
A large amount of work is devoted to the study of wave evolution on curved spacetimes
that arise in general relativity, e.g.~Schwarzschild and Kerr black holes.
This is currently a very active area of research which is motivated by the black hole stability problem.
Decay estimates for waves on black hole spacetimes are obtained, for instance, in
\cite{AndBlu09, Are12, DafRod05, DafRod09, DafRod10, DonSchSof11, DonSchSof12, Dya13, Luk10, Luk12,
FinKamSmoYau06, FinSmo09, Sch10, Tat09, TatToh11, MetTatToh12} and for
Strichartz-type estimates in this context we refer to \cite{MarMetTatToh10, Toh12}.
General one-dimensional wave equations with polynomially decaying potentials are studied
in \cite{DonSch10, CosHua12} and the semiclassical regime is considered in \cite{CosSchStaTan08, CosDonSchTan12}.
We also mention the recent \cite{AndBluNic13} which deals with a complex potential.
For other recent work on decay of solutions of wave equations with potentials see
e.g.~\cite{GeoVis03, GeoVis03a, Cuc00, CarVod12, BurPlaStaTah03, PlaStaTah03, BurPlaStaTah04, Bea94, AncFan08, Anc08,
AncFanVegVis10, CueVod06, BeaStr93, AncPie05, Pie07, Mou09} and references therein.
The first part of our paper is strongly influenced by the works \cite{SchSofSta10, SchSofSta10a}
on wave decay on manifolds with conical ends.
Needless to say, there are parallel developments for Schr\"odinger equations and other systems.
We refer the reader to the survey article \cite{Sch07}.

\subsection{Method of proof}

Our approach is based on the Fourier representation of the solution of Eq.~\eqref{eq:main}.
In other words, we use the spectral transformation (the ``distorted Fourier transform'') associated to the self-adjoint 
Schr\"odinger operator $Af:=-f''+Vf$ to write the solution as an oscillatory integral 
in the Fourier variable $\xi$.
This leads to the representation
\begin{equation}
\label{eq:introFourier}
 u(t,x)=2\int_0^\infty \phi(x,\xi^2) [\xi\cos(t\xi)\mc F f(\xi)+\sin(t\xi)\mc Fg(\xi)]\rho(\xi^2)d\xi. 
 \end{equation}
Here, $\rho(\lambda)d\lambda$ is the spectral measure associated to $A$, $\phi$ is the unique function
satisfying 
\[ -\phi''(x,\lambda)+V(x)\phi(x,\lambda)=\lambda \phi(x,\lambda),\quad \phi(0,\lambda)=-1, \phi'(0,\lambda)=0, \]
for $\lambda>0$ and
\[ \mc F f(\xi):=\int_0^\infty \phi(x,\xi^2)f(x)dx \]
is the distorted Fourier transform of $f$.
In analogy with the free Fourier transform we find it convenient to use the Fourier variable
$\xi$ instead of the spectral parameter $\lambda=\xi^2$.
In order to construct $\rho$ and $\phi$, we apply Weyl-Titchmarsh theory to the operator $A$, see e.g.~\cite{DunSch88,
Wei03, GesZin06, Tes09}.
Since the spectral problem for $A$ is not exactly solvable\footnote{This is standard terminology
in physics although it is a little misleading. What one means by ``exactly solvable'' is the
existence of a fundamental system to the equation $Af=\lambda f$ in terms of special functions
that are known sufficiently well such that one can obtain \emph{global} information on the
solutions of $Af=\lambda f$.},
we have to resort to a perturbative procedure.
In this respect we closely follow the work by Schlag, Soffer, and Staubach \cite{SchSofSta10} although
we obtain more detailed information under more general assumptions.
Experience shows that the important information on decay of the wave evolution is encoded
in the asymptotic properties of $\rho(\lambda)$ as $\lambda \to 0+$.
More precisely, the decisive quantity is $\phi(x,\lambda)\phi(y,\lambda)\rho(\lambda)$.
In order to extract the necessary information, we
distinguish the regimes $0<x\lambda\leq 1$ and $x\lambda \geq 1$.
In the former case we write the spectral problem as
$-f''+V f=\lambda f$ and construct a fundamental system for this equation by perturbing around $\lambda=0$.
In the latter case we write
\[ -f''(x)-\tfrac{1}{4x^2}f(x)-\lambda f(x)=O(x^{-2-\alpha})f(x) \]
and perturb off of Hankel functions.
The errors and \emph{all their derivatives} are quantitatively controlled by Volterra iterations.
The construction is set up in such a way that there remains a window where 
we can glue together both approximations which leads to a global representation
of the Jost function\footnote{
The Jost function is the unique function satisfying 
\[ -f_+''(x,\xi)+V(x)f_+(x,\xi)=\xi^2 f_+(x,\xi),\quad \xi>0 \]
and $f_+(x,\xi)\sim e^{ix\xi}$ as $x\to\infty$.}
$f_+(\cdot,\xi)$ for small $\xi>0$.
From this information we obtain a precise description of $\rho(\lambda)$ and $\phi(x,\lambda)$ for small
$\lambda>0$.
In particular, we show that $\phi(1,\lambda)^2\rho(\lambda)\simeq 1$ as $\lambda\to 0+$ (regardless
whether $V$ is resonant or not!).
The regime $\lambda>0$ large is much easier and can be treated as a perturbation
of the free case $V=0$.
Then we use the representation Eq.~\eqref{eq:introFourier}
and nonstationary-phase-type arguments to obtain the dispersive estimates
stated in Theorems \ref{thm:basicdisp}, \ref{thm:accdec}.
The energy bounds from Theorem \ref{thm:energy} follow by relating the standard
Sobolev spaces $H^k(\R_+)$ to the ``distorted'' Sobolev spaces generated by the operator $\sqrt A$.

The main part of the paper is concerned with the vector field bounds.
Here we employ a novel approach where we develop a vector field method on the distorted Fourier
side. 
This idea is based on the observation that for the free equation ($V=0$) the vector field
$t\partial_t+x\partial_x$ has a simple representation on the Fourier side, namely $t\partial_t-\xi\partial_\xi-1$.
This suggests to write $t\partial_t+x\partial_x$ as $t\partial_t-\xi\partial_\xi-1+B$
on the \emph{distorted} Fourier side
where $B$ is a nonlocal operator.
Heuristically, one expects $B$ to be ``well-behaved'' since $-\partial_x^2+V$ may be viewed
as a perturbation of $-\partial_x^2$.
We quantify this expectation by proving various bounds for $B$ on suitable weighted $L^2$-spaces.
The operator $B$ decomposes into delta-like contributions and a singular integral operator
of Hilbert-transform-type.
We analyze $B$ in detail and obtain the necessary pointwise bounds on the kernel of $B$ that allow us
to eventually conclude the estimates stated in Theorems \ref{thm:vf} and \ref{thm:locen}.
This also requires bounds on the commutators and iterated commutators 
of $B$ with $\xi\partial_\xi$ and the evolution operators.
We remark that the analysis of the operator $B$ bears some similarities with the study
of the ``$\mc K$-operators'' in \cite{KriSchTat08, KriSchTat09, DonKri13}
although the potentials considered there display stronger decay.
Finally, we derive similar bounds for the inhomogeneous problem Eq.~\eqref{eq:maininhom}
by applying Duhamel's formula.

\subsection{Further discussion}
One should contrast the spectral behavior of $A$ to the free one-dimensional case, i.e., $V=0$ with a Neumann condition at $x=0$, 
where one has $\phi(1,\lambda)^2 \rho(\lambda)\simeq \lambda^{-\frac12}$.
Thus, the spectral measure for $A$ is \emph{more regular} than in the free one-dimensional
case which leads to the decay of the wave evolution stated in Theorems \ref{thm:basicdisp} and \ref{thm:accdec}.
We further remark that in the free two-dimensional case, i.e., 
$V(x)=-\frac{1}{4x^2}$, one obtains $\phi(1,\lambda)^2 \rho(\lambda)\simeq 1$ as $\lambda\to 0+$.
This shows that the spectral behavior of the operator $A$ is essentially the same as in
the free two-dimensional case.
We refer to Section \ref{sec:free} for a more detailed discussion on this.

To conclude this introduction, we would like to emphasize that our techniques are by no means confined to
potentials with the particular asymptotics stated in Hypothesis \ref{hyp:A} or potentials with inverse
square decay. Since we work with an explicit representation of the solution, 
it is in principle possible to consider \emph{any} kind of potential as long as
the necessary spectral theoretic assumptions are satisfied. 
In this sense, our method provides a \emph{universal} approach to the study of one-dimensional (or radial) wave evolution
in the presence of a potential.

\subsection{Notations and conventions}
The natural numbers $\{1,2,3,\dots\}$ are denoted by $\N$ and we set $\N_0:=\{0\}\cup \N$.
Furthermore, we abbreviate $\R_+:=[0,\infty)$ and $\langle x\rangle:=\sqrt{1+|x|^2}$
(the ``japanese bracket'').
We write $a\simeq b$ if there exists a constant $C>0$ such that $a\leq C b$.
Similarly, we use $a\gtrsim b$ and $a\simeq b$ means $b\lesssim a\lesssim b$.
In general, the letter $C$ (possibly with indices to indicate dependencies) stands
for a positive constant that might change its actual value at each occurrence.
We write $f(x)\sim g(x)$ for $x\to a$ 
if $\lim_{x\to a}\frac{f(x)}{g(x)}=1$.

The symbol $O(f(x))$ is used to denote a generic real-valued function that satisfies
$|O(f(x))|\lesssim |f(x)|$ in a domain of $x$ that is either specified explicitly or
follows from the context.
If the function attains complex values as well, we indicate this by a subscript $\C$, e.g.~$O_\C(x)$.
We write $O(f(x)g(y))$, etc. if the function depends on more variables.
We say that $O(x^\gamma)$, $\gamma\in \R$, behaves like a symbol if $|\partial_x^k O(x^\gamma)|\leq C_k |x|^{\gamma-k}$
for all $k\in \N_0$.

We use the standard Sobolev spaces $H^k(I)$ where $I\subset \R$ is an interval.
We also write $L^2_w(I)$ for the weighted $L^2$-space with norm $\|f\|_{L^2_w(I)}=\|w^\frac12 f\|_{L^2(I)}$
where $w$ is a weight function.
Especially when dealing with functions that depend on more variables 
it is convenient to state the variables explicitly and for this purpose we occasionally
use the notation
$\|f(x)\|_{H^k_x(I)}:=\|f\|_{H^k(I)}$.
As usual, $\mc S(I)$ denotes the Schwartz space, i.e., $f\in \mc S(I)$ if $f\in C^\infty(I)$
and $\|\langle \cdot\rangle^k \nabla^\ell f\|_{L^\infty(I)}\leq C_{k,\ell}$ for all $k,\ell\in \N_0$.
Since we are mostly dealing with functions on $\R$, the nabla operator $\nabla$ is just
the ordinary derivative but sometimes it is notationally more convenient to write $\nabla f$
instead of $f'$.

Throughout the paper, we assume that Hypothesis \ref{hyp:A} holds and the symbols $V$ and $\alpha$
are reserved for this purpose. Furthermore, when we speak of the solution
of the Cauchy problem \eqref{eq:main} we always mean the solution on the half-space $x\geq 0$
with a Neumann condition at the origin.

\section{Weyl-Titchmarsh theory for $ A$}

In this section we obtain representations for the spectral measure $\rho$ and the
function $\phi$ which yields the distorted Fourier basis.
We begin by recalling that the operator $Af=-f''+V f$ is self-adjoint on $L^2(\R_+)$
with domain
\[ \mc D(A)=\{f\in L^2(\R_+): f,f'\in \mathrm{AC}_\mathrm{loc}(\R_+), f'(0)=0, f''\in L^2(\R_+)\}. \]
Furthermore, we have $\sigma(A)=\sigma_\mathrm{ac}(A)=[0,\infty)$, see \cite{DunSch88, Wei03, Tes09}
for these standard facts and the general background on Weyl-Titchmarsh theory.

\subsection{Zero energy solutions}

We start by constructing a suitable fundamental system $\{\phi_0,\theta_0\}$ for the equation
$Af=0$.

\begin{lemma}
\label{lem:en0}
There exists a (smooth) fundamental system $\{\phi_0,\theta_0\}$ of $Af=0$ with 
$\phi_0(0)=-\theta_0'(0)=-1$ and $\phi_0'(0)=\theta_0(0)=0$.
Furthermore, we have
\begin{align*}
\phi_0(x)&=a_1 x^\frac12 \log x+a_2 x^\frac12 +O(x^{\frac12-\alpha+}) \\
\theta_0(x)&=b_1 x^\frac12 \log x+b_2 x^\frac12+ O(x^{\frac12-\alpha+})
\end{align*}
for $x\geq 2$ where $a_j,b_j \in \R$ and 
$a_1b_2-a_2b_1=1$.
Finally, the $O$-terms behave like symbols under differentiation.
\end{lemma}

\begin{proof}
First of all, it is clear the $\phi_0$ and $\theta_0$ exist on $[0,c]$ for any fixed 
$c>0$ since $V \in C^\infty([0,\infty))$.
In order to study the asymptotics, we write the equation $Af=0$ as
\[ f''(x)+\tfrac{1}{4x^2}f(x)=O(x^{-2-\alpha})f(x). \]
Note that $x^\frac12$ and $x^\frac12 \log x$ are solutions to $f''(x)+\frac{1}{4x^2}f(x)=0$.
First, we construct a solution $f_0$ which behaves like $f_0(x)\sim x^\frac12$ as $x\to\infty$.
The variation of constants formula suggests to consider the integral equation
\begin{align}
\label{eq:voltth0} f_0(x)&=x^\frac12
-x^\frac12 \int_x^\infty y^\frac12 \log y O(y^{-2-\alpha})
f_0(y)dy \nonumber \\
&\quad +x^\frac12 \log x \int_x^\infty y^\frac12 O(y^{-2-\alpha})f_0(y)dy.
\end{align}
Setting $\tilde f_0(x):=x^{-\frac12}f_0(x)$ we rewrite Eq.~\eqref{eq:voltth0}
as 
\[ \tilde f_0(x)=1+\int_x^\infty K(x,y)\tilde f_0(y)dy \]
with a kernel $K$ satisfying
\[ |K(x,y)|\lesssim y^{-1-\alpha}\log y \]
for all $2\leq x\leq y$. Consequently, we infer
\[ \int_2^\infty \sup_{x\in (2,y)}|K(x,y)|dy \lesssim 1 \]
and the standard result on Volterra equations (see e.g.~\cite{DeiTru79} or \cite{SchSofSta10}, Lemma 2.4) 
implies the existence of $f_0$
satisfying Eq.~\eqref{eq:voltth0} for $x\geq 2$.
Furthermore, by construction, $f_0$ satisfies
\[ f_0(x)=x^\frac12[1 + O(x^{-\alpha+})] \]
for $x\geq 2$ with an $O$-term that behaves like a symbol.

Next, by the standard reduction formula, a second solution $f_1$ is given by
\[ f_1(x)=f_0(x)\int_c^x f_0(y)^{-2}dy \]
provided $c$ is chosen so large that $f_0(y)>0$ for all $y\geq c$.
We have the asymptotics $f_0(y)^{-2}=y^{-1}[1+O(y^{-\alpha+})]$ for $y\geq c$ and thus,
\[ f_1(x)=x^\frac12 \log x+\tilde b_2 x^\frac12+O(x^{\frac12-\alpha+}) \]
for $x\geq c$ where $\tilde b_2 \in \R$ is some constant.
Since $f_0$ and $f_1$ are linearly independent, there exist constants $c_0,c_1 \in \R$
such that $\phi_0=c_0 f_0+c_1 f_1$ and this yields the stated asymptotics
for $\phi_0$.
Similarly, there exist constants $\tilde c_0, \tilde c_1 \in \R$ such that
$\theta_0=\tilde c_0 f_0+\tilde c_1 f_1$ and we obtain the claimed asymptotics for $\theta_0$.
From the boundary conditions at $0$ we infer $W(\theta_0,\phi_0)=1$ and evaluating this
Wronskian as $x\to\infty$ yields $a_1b_2-a_2b_1=1$.
\end{proof}

\subsection{Perturbative solutions for small energies}
Next, we construct solutions to $Af=\lambda f$ where $\lambda>0$ is small.

\begin{lemma}
\label{lem:ensm}
There exists a (smooth) fundamental system $\{\phi(\cdot,\lambda),\theta(\cdot,\lambda)\}$
for the equation $Af=\lambda f$ which satisfies
$\phi(0,\lambda)=-\theta'(0,\lambda)=-1$ and $\phi'(0,\lambda)=\theta(0,\lambda)=0$
for all $\lambda>0$.
In addition, we have the asymptotics
\begin{align*}
\phi(x,\lambda)&=\phi_0(x)+O(x^2\langle x\rangle^{\frac12+}\lambda) \\
\theta(x,\lambda)&=\theta_0(x)+O(x^2\langle x\rangle^{\frac12+}\lambda)
\end{align*}
for $x\in [0,\lambda^{-\frac12+}]$ and $0<\lambda \leq 1$
where the $O$-terms behave like symbols under differentiation with respect to $x$ and $\lambda$.
\end{lemma}

\begin{proof}
In view of the variation of constants formula and Lemma \ref{lem:en0} our goal is to construct
a function $\phi(\cdot,\lambda)$ which satisfies
\begin{align} 
\label{eq:voltphi}
\phi(x,\lambda)&=\phi_0(x)-\lambda \phi_0(x)\int_0^x \theta_0(y)\phi(y,\lambda)dy \nonumber \\
&\quad+\lambda \theta_0(x)\int_0^x \phi_0(y)\phi(y,\lambda)dy.
\end{align}
We set $\tilde \phi(x,\lambda):=\frac{\phi(x,\lambda)}{\langle x\rangle^\frac12 \langle 
\log \langle x \rangle \rangle}$ and infer the Volterra equation
\[ \tilde \phi(x,\lambda)=\frac{\phi_0(x)}{\langle x\rangle^\frac12 \langle \log \langle x
\rangle \rangle}+\int_0^x K(x,y,\lambda)\tilde \phi(y,\lambda)dy \]
with
\begin{align*} |K(x,y,\lambda)|&=\lambda \frac{\langle y\rangle^\frac12 \langle \log\langle y\rangle\rangle}
{\langle x\rangle^\frac12 \langle \log\langle x\rangle\rangle}\left |\phi_0(x)\theta_0(y)
-\phi_0(y)\theta_0(x) \right | \\
&\lesssim \lambda \langle y\rangle^{1+}
\end{align*}
for all $0\leq y\leq x$.
We obtain
\[ \int_0^{\lambda^{-1/2+}}\sup_{x\geq y}|K(x,y,\lambda)|dy\lesssim 1 \]
and the standard result on Volterra equations implies the existence of a solution $\phi(\cdot,\lambda)$ of
Eq.~\eqref{eq:voltphi} on $[0,\lambda^{-\frac12+}]$.
Furthermore, we have the bound
\[ \phi(x,\lambda)=\phi_0(x)+O(x^2 \langle x\rangle^{\frac12+}\lambda) \]
for $x\in [0,\lambda^{-\frac12+}]$ as claimed.
The construction for $\theta(\cdot,\lambda)$ is identical.
\end{proof}

\subsection{The Jost function at small energies}

As a next step in the analysis we construct the outgoing Jost function of the operator
$A$, i.e., the function $f_+(\cdot,\xi)$ satisfying $Af_+(\cdot,\xi)=\xi^2 f_+(\cdot,\xi)$, 
$\xi>0$, and $f_+(x,\xi)\sim e^{i\xi x}$ as $x\to\infty$. 
In the following, $J_0$, $Y_0$, and $H^\pm_0=J_0\pm iY_0$ denote the standard Bessel and Hankel functions of order $0$
as defined in \cite{Olv97, DLMF}.

\begin{lemma}
\label{lem:Jostsm}
Fix $\epsilon>0$.
Then the outgoing Jost function $f_+(\cdot,\xi)$ of the operator $A$ is of the form
\[ f_+(x,\xi)=\sqrt{\pi/2}e^{i\pi/4}(x\xi)^\frac12 [J_0(x\xi)+iY_0(x\xi)]
+e^{ix\xi}O_\C(x^{-\alpha-\epsilon}\xi^{-\epsilon}\langle x\xi\rangle^{-1+\epsilon}) \]
for all $x\geq \xi^{-\epsilon/\alpha}$ and $0<\xi\leq 1$.
Furthermore, the $O_\C$-term behaves like a symbol under differentiation with respect 
to $x$ and $\xi$.
\end{lemma}

\begin{proof}
We rewrite the spectral equation $Af=\xi^2 f$ as 
\[ f''(x)+\tfrac{1}{4x^2}f(x)+\xi^2 f(x)=x^{-2}O(\langle x\rangle^{-\alpha})f(x). \]
Rescaling by $f(x)=g(\xi x)$ yields
\[ g''(z)+\tfrac{1}{4z^2}g(z)+g(z)
=z^{-2}O(\langle \tfrac{z}{\xi}\rangle^{-\alpha})g(z) \]
where $z=\xi x$.
A fundamental system for the left-hand side is provided by the Hankel functions
$h_\pm(z):=\sqrt{\pi/2}e^{\pm i\pi/4}\sqrt{z}H^\pm_0(z)$ with asymptotic behavior
$h_\pm(z)\sim e^{\pm iz}$ as $z\to\infty$.
In terms of Bessel functions we have
\begin{align*} 
h_\pm(z)&=\sqrt{\pi/2}e^{\pm i\pi/4}\sqrt{z}[J_0(z)\pm iY_0(z)]. 
 \end{align*}
Note that $|h_\pm(z)|\lesssim 1$ for $z\geq 1$ and, since $|J_0(z)|\lesssim 1$, $|Y_0(z)|\lesssim \langle \log z\rangle$
for $z\in (0,1]$, we infer the bound $|h_\pm(z)|\lesssim z^{\frac12-}$ for $z\in (0,1]$.
We consider the Volterra equation
\begin{equation}
\label{eq:volg+} 
\frac{g_+(z,\xi)}{e^{iz}}=\frac{h_+(z)}{e^{iz}}+\int_z^\infty K(z,y,\xi)
\frac{g_+(y,\xi)}{e^{iy}}dy 
\end{equation}
with $K(z,y,\xi)=ie^{-i(z-y)}[h_+(z)h_-(y)-h_+(y)h_-(z)]y^{-2}O(\langle \frac{y}{\xi}
\rangle^{-\alpha})$.
The kernel $K$ satisfies the bounds
\begin{align*}
|K(z,y,\xi)|&\lesssim y^{-2}\langle \tfrac{y}{\xi}\rangle^{-\alpha}\lesssim \xi^\alpha y^{-2-\alpha}, & & y\geq 1, z\geq 0 \\
|K(z,y,\xi)|&\lesssim y^{\frac12-}z^{\frac12-}y^{-2}\langle \tfrac{y}{\xi}\rangle^{-\alpha}\lesssim \xi^\alpha y^{-1-\alpha-\epsilon}
& & \xi\leq z\leq y\leq 1
\end{align*}
which may be combined to yield $|K(z,y,\xi)|\lesssim \xi^\alpha y^{-1-\alpha-\epsilon}\langle y\rangle^{-1+\epsilon}$
for all $\xi\leq z\leq y$.
This implies
\[ \int_z^\infty |K(z,y,\xi)|dy\lesssim \xi^\alpha z^{-\alpha-\epsilon}\langle z\rangle^{-1+\epsilon} \]
for $z\geq \xi$ and thus,
\[ \int_{\xi^{1-\epsilon/\alpha}}^\infty \sup_{z\in (\xi^{1-\epsilon/\alpha},y)}|K(z,y,\xi)|dy\lesssim \xi^{\epsilon^2/\alpha}\lesssim 1. \]
As a consequence, Eq.~\eqref{eq:volg+} has a solution on $z\geq \xi^{1-\epsilon/\alpha}$ satisfying
\[ g_+(z,\xi)=h_+(z)+e^{iz}O_\C(z^{-\alpha-\epsilon}\langle z\rangle^{-1+\epsilon}\xi^\alpha). \]
The claimed symbol behavior of the $O_\C$-term deserves a comment.
At first glance, this might seem odd since the kernel $K$ is oscillatory after all.
However, by a simple change of variables one may rewrite the Volterra equation \eqref{eq:volg+} 
as
\[ \tilde g_+(z,\xi)=\frac{h_+(z)}{e^{iz}}+\int_0^\infty K(z,y+z,\xi)
\tilde g_+(y+z,\xi)dy \]
for $\tilde g_+(z,\xi)=e^{-iz}g_+(z,\xi)$.
Furthermore, by the asymptotics of $h_\pm$ we infer
\begin{align*} 
K(z,y,\xi)&=e^{-i(z-y)}[e^{i(z-y)}-e^{-i(z-y)}]O_\C(z^0 y^0 \xi^0) \\
&=[1-e^{-2i(z-y)}]O_\C(z^0 y^0 \xi^0)
\end{align*}
with an $O_\C$-term of symbol type and thus,
\[ K(z,y+z,\xi)=[1-e^{-2iy}]O_\C(z^0 y^0 \xi^0). \]
In this formulation it is evident 
that derivatives with respect to $z$ and $\xi$ 
hit only terms with symbol behavior.
In summary, we obtain the representation of the Jost solution
\[ f_+(x,\xi)=g_+(x\xi,\xi)=h_+(x\xi)+e^{ix\xi}O_\C(x^{-\alpha-\epsilon}\xi^{-\epsilon}\langle x\xi\rangle^{-1+\epsilon}). \]
with symbol behavior of the $O_\C$-term.
\end{proof}

\subsection{The Jost function at large energies}
Next, we derive a suitable representation of the Jost function $f_+(\cdot,\xi)$ in the case
$\xi\geq \frac12$.
This is much easier since the problem can be treated as a perturbation of the free case $V=0$.

\begin{lemma}
\label{lem:Jostlg}
The outgoing Jost function of the operator $A$ is of the form
\[ f_+(x,\xi)=e^{ix\xi}[1+O_\C(\langle x\rangle^{-1}\xi^{-1})] \]
for all $x\geq 0$ and $\xi\geq \frac12$ where the $O_\C$-term behaves like a symbol under
differentiation with respect to $x$ and $\xi$.
\end{lemma}

\begin{proof}
We rewrite $Af=\xi^2 f$ as
\[ f''(x)+\xi^2 f(x)=O(\langle x\rangle^{-2})f(x) \]
and the variation of constants formula yields the Volterra equation
\[ e^{-i\xi x}f_+(x,\xi)=1+\int_x^\infty K(x,y,\xi)e^{-i\xi y}f_+(y,\xi)dy \]
with
\begin{align*} 
K(x,y,\xi)&=\tfrac{1}{2i\xi}e^{-i\xi (x-y)}(e^{i\xi x}e^{-i\xi y}-e^{i\xi y}e^{-i\xi x})
O(\langle y\rangle^{-2})  \\
&=\tfrac{1}{2i\xi}(1-e^{-2i\xi(x-y)})O(\langle y\rangle^{-2}).
\end{align*}
We have
$|K(x,y,\xi)|\lesssim \xi^{-1}\langle y\rangle^{-2}$
and thus,
\[ \int_0^\infty \sup_{x> 0}|K(x,y,\xi)|dy\lesssim \xi^{-1} \]
which yields the existence of $f_+(\cdot,\xi)$ on $[0,\infty)$ with the stated bounds.
\end{proof}

\subsection{The Wronskians}

We compute the Wronskians of the Jost function and the solutions $\phi$, $\theta$
from Lemma \ref{lem:ensm} for small energies.

\begin{lemma}
\label{lem:Wsm}
We have
\begin{align*} 
\sqrt{2/\pi}e^{-i\pi/4}W(f_+(\cdot,\xi),\phi(\cdot,\xi^2))&=\xi^\frac12
\big [a_1+\tfrac{2i a_1}{\pi} \log\xi+\tfrac{2i}{\pi}(\gamma_0 a_1-a_2) \\
&\quad + O_\C(\xi^{\frac{3\alpha}{4+2\alpha}-}) \big ] \\
\sqrt{2/\pi}e^{-i\pi/4}W(f_+(\cdot,\xi),\theta(\cdot,\xi^2))&=\xi^\frac12
\big [b_1+\tfrac{2i b_1}{\pi} \log\xi+\tfrac{2i}{\pi}(\gamma_0 b_1-b_2) \\
&\quad + O_\C(\xi^{\frac{3\alpha}{4+2\alpha}-}) \big ] 
\end{align*}
for $0<\xi\leq \frac12$ where $a_j$, $b_j$, $j\in \{1,2\}$, are from Lemma \ref{lem:en0} 
and $\gamma_0=\gamma-\log 2$ with $\gamma$ the Euler-Mascheroni constant.
\end{lemma}

\begin{proof}
From Lemmas \ref{lem:en0} and \ref{lem:ensm} we have
\begin{align*}
\phi(x,\xi^2)&=a_1 x^\frac12 \log x+a_2x^\frac12+O(x^{\frac12-\alpha+})+O(x^{\frac52+}\xi^2) \\
\phi'(x,\xi^2)&=\tfrac{a_1}{2}x^{-\frac12}\log x+(a_1+\tfrac{a_2}{2})x^{-\frac12}
+O(x^{-\frac12-\alpha+})+O(x^{\frac32+}\xi^2)
\end{align*}
for $x\in [2,\xi^{-1+}]$.
Furthermore, from Lemma \ref{lem:Jostsm} we infer
\begin{align*} 
\Re [\sqrt{2/\pi}e^{-i\pi/4} f_+(x,\xi)]&=(x\xi)^\frac12 J_0(x\xi)+O(x^{-\alpha-\epsilon}
\xi^{-\epsilon}) \\
&=(x\xi)^\frac12+O((x\xi)^\frac52)+O(x^{-\alpha-\epsilon}
\xi^{-\epsilon}) \\
\Re [\sqrt{2/\pi}e^{-i\pi/4} f_+'(x,\xi)]&=\tfrac12 \xi(x\xi)^{-\frac12}+O(x^\frac32 \xi^\frac52)
+O(x^{-1-\alpha-\epsilon} \xi^{-\epsilon})
\end{align*}
for $x \in [\xi^{-\epsilon/\alpha},\xi^{-1}]$.
We evaluate the Wronskians at $x=\xi^{-\frac{2}{2+\alpha}}$.
From the above we infer
\begin{align*}
\phi(\xi^{-\frac{2}{2+\alpha}},\xi^2)&=-\tfrac{2a_1}{2+\alpha}\xi^{-\frac{1}{2+\alpha}}\log \xi
+a_2 \xi^{-\frac{1}{2+\alpha}}+O(\xi^{\frac{2\alpha-1}{2+\alpha}-}) \\
\phi'(\xi^{-\frac{2}{2+\alpha}},\xi^2)&=-\tfrac{a_1}{2+\alpha}\xi^{\frac{1}{2+\alpha}}\log \xi
+(a_1+\tfrac{a_2}{2})\xi^{\frac{1}{2+\alpha}}+O(\xi^{\frac{2\alpha+1}{2+\alpha}-})
\end{align*}
as well as
\begin{align*}
\Re[\sqrt{2/\pi}e^{-i\pi/4} f_+(\xi^{-\frac{2}{2+\alpha}},\xi)]&=\xi^\frac{\alpha}{4+2\alpha}
+O(\xi^{\frac{2\alpha}{2+\alpha}-}) \\
\Re [\sqrt{2/\pi}e^{-i\pi/4}f_+'(\xi^{-\frac{2}{2+\alpha}},\xi)]&=
\tfrac12 \xi^{\frac{4+\alpha}{4+2\alpha}}+O(\xi^{\frac{2+2\alpha}{2+\alpha}-}) .
\end{align*}
This yields
\begin{align*}
\Re W(\sqrt{2/\pi}e^{-i\pi/4}f_+(\cdot,\xi),\phi(\cdot,\xi^2))&=-\tfrac{a_1}{2+\alpha}
\xi^\frac12 \log \xi+(a_1+\tfrac{a_2}{2})\xi^\frac12 \\
&\quad -\left [ -\tfrac{a_1}{2+\alpha}\xi^\frac12 \log \xi+\tfrac{a_2}{2}\xi^\frac12 \right ] \\
&\quad +O(\xi^{\frac{1+2\alpha}{2+\alpha}-}) \\
&=\xi^\frac12[a_1+O(\xi^{\frac{3\alpha}{4+2\alpha}-})].
\end{align*}

Similarly, for the imaginary part we obtain
\begin{align*}
\Im [\sqrt{2/\pi}e^{-i\pi/4}f_+(x,\xi)]&=(x\xi)^\frac12 Y_0(x\xi)+O(x^{-\alpha-\epsilon}
\xi^{-\epsilon}) \\
&=\tfrac{2}{\pi}(x\xi)^\frac12 [\log(x \xi)+\gamma_0]+O((x\xi)^{\frac52-}) \\
&\quad +O(x^{-\alpha-\epsilon}
\xi^{-\epsilon}) \\
\Im [\sqrt{2/\pi}e^{-i\pi/4}f_+'(x,\xi)]&=\tfrac{1}{\pi} \xi(x\xi)^{-\frac12}\log(x\xi)
+\tfrac{2+\gamma_0}{\pi}\xi (x\xi)^{-\frac12}+O(x^{\frac32-} \xi^{\frac52-}) \\
&\quad +O(x^{-1-\alpha-\epsilon} \xi^{-\epsilon})
\end{align*}
for $x \in [\xi^{-\epsilon/\alpha},\xi^{-1}]$ where $\gamma_0=\gamma-\log 2$ and
$\gamma$ is Euler's constant.
Consequently, at $x=\xi^{-\frac{2}{2+\alpha}}$ we infer
\begin{align*}
\Im [\sqrt{2/\pi}e^{-i\pi/4}f_+(\xi^{-\frac{2}{2+\alpha}},\xi)]&=\tfrac{2}{\pi}\tfrac{\alpha}{2+\alpha}
\xi^\frac{\alpha}{4+2\alpha}\log \xi
+\tfrac{2\gamma_0}{\pi}\xi^{\frac{\alpha}{4+2\alpha}}+O(\xi^{\frac{2\alpha}{2+\alpha}-}) \\
\Im [\sqrt{2/\pi}e^{-i\pi/4}f_+'(\xi^{-\frac{2}{2+\alpha}},\xi)]&=
\tfrac{1}{\pi}\tfrac{\alpha}{2+\alpha}\xi^{\frac{4+\alpha}{4+2\alpha}}\log \xi
+\tfrac{2+\gamma_0}{\pi}\xi^{\frac{4+\alpha}{4+2\alpha}}+O(\xi^{\frac{2+2\alpha}{2+\alpha}-}) 
\end{align*}
and thus,
\begin{align*}
\Im W&(\sqrt{2/\pi}e^{-i\pi/4}f_+(\cdot,\xi),\phi(\cdot,\xi^2))= \\
&=-\tfrac{2}{\pi}
\tfrac{\alpha a_1}{(2+\alpha)^2}\xi^\frac12 \log^2 \xi
+\tfrac{2}{\pi}\tfrac{\alpha(a_1+\frac{a_2}{2})-\gamma_0 a_1}{2+\alpha}\xi^\frac12 \log \xi 
+\tfrac{2\gamma_0}{\pi}(a_1+\tfrac{a_2}{2})\xi^\frac12 \\
&\quad -\Big [ -\tfrac{2}{\pi}\tfrac{\alpha a_1}{(2+\alpha)^2}\xi^\frac12 \log^2 \xi 
+\tfrac{1}{\pi}\tfrac{\alpha a_2-2 a_1(2+\gamma_0)}{2+\alpha}\xi^\frac12 \log \xi 
+\tfrac{2+\gamma_0}{\pi}a_2 \xi^\frac12 \Big ] \\
&\quad +O(\xi^{\frac{1+2\alpha}{2+\alpha}-}) \\
&=\tfrac{2}{\pi}a_1 \xi^\frac12 \log \xi
+\tfrac{2(\gamma_0 a_1-a_2)}{\pi}\xi^\frac12 +O(\xi^{\frac{1+2\alpha}{2+\alpha}-}).
\end{align*}
\end{proof}

We also provide expressions for the Wronskians at large energies. This is considerably easier
since Lemma \ref{lem:Jostlg} allows us to evaluate the Jost function at $x=0$.

\begin{lemma}
\label{lem:Wlg}
We have
\begin{align*} 
W(f_+(\cdot,\xi), \phi(\cdot,\xi^2))&=i\xi[1+O_\C(\xi^{-1})] \\
W(f_+(\cdot,\xi), \theta(\cdot,\xi^2))&=1+O_\C(\xi^{-1})
\end{align*}
for all $\xi\geq \frac12$ where the $O_\C$-terms are of symbol type.
\end{lemma}

\begin{proof}
According to Lemma \ref{lem:Jostlg} we have
\begin{align*} 
f_+(x,\xi)&=e^{i\xi x}[1+O_\C(\langle x\rangle^{-1}\xi^{-1})] \\
f_+'(x,\xi)&=i\xi e^{i\xi x}[1+O_\C(\langle x\rangle^{-1}\xi^{-1})].
\end{align*}
Since $\phi(0,\xi^2)=-\theta'(0,\xi^2)=-1$ and $\phi'(0,\xi^2)=\theta(0,\xi^2)=0$ for all
$\xi\geq 0$ by Lemma \ref{lem:ensm}, we obtain by evaluation at $x=0$,
\begin{align*}
W(f_+(\cdot,\xi),\phi(\cdot,\xi^2))&=i\xi[1+O_\C(\xi^{-1})] \\
W(f_+(\cdot,\xi),\theta(\cdot,\xi^2))&=1+O_\C(\xi^{-1})
\end{align*}
for all $\xi\geq \frac12$ as claimed.
\end{proof}

\subsection{Computation of the spectral measure}

We have collected all the information necessary to compute the spectral measure.

\begin{lemma}
\label{lem:rho}
The spectral measure $\mu$ associated to the operator $ A$ is purely absolutely continuous
and given by $d\mu=\rho(\lambda)d\lambda$
where $\lambda$ is the Lebesgue measure on $\R$ and
\begin{align*} 
\rho(\lambda)&=0  & &\lambda< 0 \\
\rho(\lambda)&=\frac{1}{\frac{\pi^2}{2}a_1^2+2[\frac{a_1}{2}\log\lambda+\gamma_0 a_1-a_2]^2}
[1+O(\lambda^{\frac{3\alpha}{8+4\alpha}-})] & &0\leq \lambda< \tfrac14 \\
\rho(\lambda)&=\tfrac{1}{\pi}\lambda^{-\frac12}[1+O(\lambda^{-\frac12})] & & \lambda \geq \tfrac14
\end{align*}
where $a_1, a_2$ are from Lemma \ref{lem:en0} and $\gamma_0=\gamma-\log 2$
with $\gamma$ the Euler-Mascheroni constant.
Furthermore, the $O$-terms are of symbol type.
\end{lemma}

\begin{proof}
The fact that $\rho(\lambda)=0$ for $\lambda<0$ follows from $\sigma( A)=[0,\infty)$ and
thus, it suffices to consider $\lambda\geq 0$.
We set $\xi=\lambda^\frac12$ and consider $0\leq \xi<\frac12$.
The Weyl-Titchmarsh $m$-function is given by
\[ m(\xi^2)=\frac{W(\theta(\cdot,\xi^2),f_+(\cdot,\xi))}{W(f_+(\cdot,\xi),\phi(\cdot,\xi^2))}
=\frac{B_1(\xi)+iB_2(\xi)}{A_1(\xi)+iA_2(\xi)} \]
where
\begin{align*}
A_1(\xi)&=a_1+O(\xi^{\frac{3\alpha}{4+2\alpha}-}) \\
A_2(\xi)&=\tfrac{2a_1}{\pi}\log \xi+\tfrac{2}{\pi}(\gamma_0 a_1-a_2)+O(\xi^{\frac{3\alpha}{4+2\alpha}-}) \\
B_1(\xi)&=-b_1+O(\xi^{\frac{3\alpha}{4+2\alpha}-}) \\
B_2(\xi)&=-\tfrac{2b_1}{\pi}\log \xi-\tfrac{2}{\pi}(\gamma_0 b_1-b_2)+O(\xi^{\frac{3\alpha}{4+2\alpha}-}),
\end{align*}
see Lemma \ref{lem:Wsm}.
The spectral function is computed as
\[ \rho(\xi^2)=\tfrac{1}{\pi}\Im m(\xi^2)=\tfrac{1}{\pi}\frac{A_1(\xi) B_2(\xi)
-A_2(\xi) B_1(\xi)}{A_1(\xi)^2+A_2(\xi)^2} \]
and we have
\begin{align*} 
A_1(\xi)^2+A_2(\xi)^2&=a_1^2+\tfrac{4}{\pi^2}[a_1 \log \xi+\gamma_0 a_1-a_2]^2
+O(\xi^{\frac{3\alpha}{4+2\alpha}-}) \\
A_1(\xi)B_2(\xi)-A_2(\xi)B_1(\xi)&=\tfrac{2}{\pi}(a_1b_2-a_2b_1)
+O(\xi^{\frac{3\alpha}{4+2\alpha}-})
\end{align*}
which yields the stated form for $0\leq \lambda< \frac14$.
For $\lambda\geq \frac14$ we use Lemma \ref{lem:Wlg} and infer
\[ m(\xi^2)=\frac{-1}{i\xi}[1+O_\C(\xi^{-1})]=i\xi^{-1}[1+O_\C(\xi^{-1})] \]
and thus, $\rho(\xi^2)=\frac{1}{\pi}\Im m(\xi^2)=\tfrac{1}{\pi}\xi^{-1}[1+O(\xi^{-1})]$
for all $\xi\geq \frac12$.
\end{proof}

\subsection{Global representations for $\phi(\cdot,\lambda)$}
Based on the above we obtain the following representation of the function $\phi(\cdot,\lambda)$.

\begin{lemma}
\label{lem:phi}
Fix $\epsilon>0$.
For the function $\phi$ from Lemma \ref{lem:ensm} we have
\begin{align*}
\phi(x,\xi^2)&=\phi_0(x)+O(x^2 \langle x \rangle^{\frac12+}\xi^2) 
& &x\in [0,\xi^{-1+}] \\
\phi(x,\xi^2)&= x^\frac12[a_1\log x+a_2+O(x^0 \xi^{\frac{3\alpha}{4+2\alpha}-})]
[1+O((x\xi)^{2-})] \\
&\quad +\xi^{-\frac12}[1+a_1 \log \xi]O(x^{-\alpha-\epsilon}\xi^{-\epsilon})
&& x\in [\xi^{-\epsilon/\alpha},\xi^{-1}] \\
\phi(x,\xi^2)&=a(\xi)e^{i\xi x}[1+O_\C((x\xi)^{-1})+O_\C(x^{-1-\alpha}\xi^{-1})] \\
&\quad +\mbox{compl.~conj.} && x\geq \xi^{-1}
\end{align*}
for all $0<\xi<\frac12$
where $|a(\xi)|\lesssim \xi^{-\frac12}(|a_1\log \xi|+1)$.
Furthermore, in the case $\xi\geq \frac12$ we have
\begin{align*}
\phi(x,\xi^2)&=[-\tfrac12+O_\C(\xi^{-1})]e^{i\xi x}[1+O_\C(\langle x\rangle^{-1}\xi^{-1})]
+\mbox{compl.~conj.}
\end{align*}
for all $x\geq 0$.
Finally, all $O$- and $O_\C$-terms are of symbol type.
\end{lemma}

\begin{proof}
The first assertion is a consequence of Lemma \ref{lem:ensm}.
From the asymptotics $f_+(x,\xi)\sim e^{i\xi x}$ as $x\to \infty$ we infer
$W(f_+(\cdot,\xi),\overline{f_+(\cdot,\xi)})=-2i\xi$ which shows that 
$\{f_+(\cdot,\xi),\overline{f_+(\cdot,\xi)}\}$ is a fundamental system for the equation
$Af=\xi^2 f$ provided $\xi>0$.
Consequently, there exist constants $a(\xi)$, $b(\xi)$ such that
\begin{equation}
\label{eq:phif+}
 \phi(x,\xi^2)=a(\xi)f_+(x,\xi)+b(\xi)\overline{f_+(x,\xi)}. 
\end{equation}
Since $\phi(\cdot,\xi^2)$ is real-valued for $\xi\geq 0$, we must have $b(\xi)=\overline{a(\xi)}$.
From Eq.~\eqref{eq:phif+} we infer
\[ W(f_+(\cdot,\xi),\phi(\cdot,\xi^2))=\overline{a(\xi)}W(f_+(\cdot,\xi),\overline{f_+(\cdot,\xi)})
=-2i\xi \overline{a(\xi)} \]
and Lemma \ref{lem:Wsm} yields
\[ \overline{a(\xi)}=\tfrac{i}{2}\sqrt{\pi/2}e^{i\pi/4}\xi^{-\frac12}
\big [a_1+\tfrac{2i a_1}{\pi} \log\xi+\tfrac{2i}{\pi}(\gamma_0 a_1-a_2) \\
 + O_\C(\xi^{\frac{3\alpha}{4+2\alpha}-}) \big ] \]
 provided $0<\xi<\frac12$.
From Lemma \ref{lem:Jostsm} and the Bessel asymptotics we obtain
 \begin{align*}
\overline{f_+(x,\xi)}&=\sqrt{\pi/2}e^{-i\pi/4}(x\xi)^\frac12[1-i\tfrac{2}{\pi}\log (x\xi)
-i\tfrac{2\gamma_0}{\pi}+O_\C((x\xi)^{2-})] \\
&\quad +O_\C(x^{-\alpha-\epsilon}\xi^{-\epsilon})
 \end{align*}
for all $x\in [\xi^{-\epsilon/\alpha},\xi^{-1}]$
and this yields
 \begin{align*}
\phi(x,\xi^2)&=2\Re[\overline{a(\xi)f_+(x,\xi)}] \\
&=x^\frac12[a_1\log x+a_2+O(x^0 \xi^{\frac{3\alpha}{4+2\alpha}-})][1+O((x\xi)^{2-})] \\
&\quad +[1+a_1\log \xi]O(x^{-\alpha-\epsilon}\xi^{-\frac12-\epsilon})
 \end{align*}
as claimed.
In the case $x\geq \xi^{-1}$ we invoke Lemma \ref{lem:Jostsm} and the Hankel asymptotics
to conclude
\[ \overline{f_+(x,\xi)}=e^{-i\xi x}[1+O_\C((x\xi)^{-1})
+O_\C(x^{-1-\alpha} \xi^{-1})] \]
and we infer the stated expression.

In the case $\xi\geq \frac12$ we obtain from Lemma \ref{lem:Wlg} that
$a(\xi)=-\tfrac12+O_\C(\xi^{-1})$ and the stated expression follows from Lemma \ref{lem:Jostlg}.
\end{proof}

A simple but useful consequence of the representations in Lemma \ref{lem:phi}
is the following estimate for the nonoscillatory regime.

\begin{corollary}
\label{cor:phi}
We have the bound
\[ |\phi(x,\xi^2)|\lesssim \langle x\rangle^\frac12 (|a_1\log \xi|+1) \]
for all $x \in[0,\xi^{-1}]$ and $0<\xi\leq \frac12$ where $a_1$ is the constant
from Lemma \ref{lem:en0}.
\end{corollary}

\begin{proof}
Let $\epsilon \in (0,1)$. If $x\in [0,\xi^{-1+\epsilon/2}]$ we use 
$\langle x\rangle^{2+\epsilon}\xi^2\leq \xi^{\epsilon^2/2}\lesssim 1$ to conclude
the stated bound from Lemmas \ref{lem:en0} and \ref{lem:phi}.
If $x\in [\xi^{-1+\epsilon/2},\xi^{-1}]$ we have $\xi^{-\frac12}\lesssim \langle x\rangle^{\frac12+\epsilon}$
and, since 
\[ \langle x\rangle^\epsilon O(x^{-\alpha-\epsilon}\xi^{-\epsilon})
=O(\xi^{\alpha-\epsilon \alpha/2-\epsilon-\epsilon^2/2}), \]
we infer the stated bound from Lemma \ref{lem:phi} provided $\epsilon$ is sufficiently small.
\end{proof}

\subsection{The distorted Fourier transform}

Now we define the distorted Fourier transform by
\[ \mc F f(\xi):=\int_0^\infty \phi(x,\xi^2)f(x)dx, \quad f\in \mc S(\R_+). \]

\begin{lemma}
\label{lem:Fourier}
The operator $\mc F: \mc S(\R_+)\to L^2_{\tilde \rho}(\R_+)$ extends to a unitary map
$\mc F: L^2(\R_+)\to L^2_{\tilde \rho}(\R_+)$ where $\tilde \rho(\xi):=2\xi \rho(\xi^2)$.
Furthermore, for all $f\in \mc D(A)$ we have
\[ \mc F  A f(\xi)=\xi^2 \mc F f(\xi). \]
Finally, the inverse $\mc F^{-1}$ is given by
\[ \mc F^{-1}f(x)=\int_0^\infty \phi(x,\xi^2)f(\xi)\tilde \rho(\xi)d\xi \]
for $f\in \mc S(\R_+)$.
\end{lemma}

\begin{proof}
This is a consequence of Weyl-Titchmarsh theory, see \cite{DunSch88, Wei03, GesZin06, Tes09}.
\end{proof}

As in the case of the standard Fourier transform, we will use the term ``Plancherel's theorem'' to refer to the fact that 
$\|\mc F f\|_{L^2_{\tilde \rho}(\R_+)}=\|f\|_{L^2(\R_+)}$.

\section{Dispersive bounds}
\label{sec:disp}

In this section we prove dispersive estimates for the solution of the 
initial value problem
\begin{equation}
\label{eq:init}
 \left \{ \begin{array}{l}
\partial_t^2 u(t, \cdot)+ Au(t, \cdot)=0, \quad t>0 \\
u(0,\cdot)=f,\quad \partial_t u(t,\cdot)|_{t=0}=g \end{array} \right . 
\end{equation}
where $f,g \in \mc S(\R_+)$. As usual, by standard approximation arguments
one may allow for more general $f,g$ as well. 

\subsection{Fundamental dispersive estimate}

We start with the fundamental dispersive estimate.
First, we consider the sine evolution.
The following simple result will be useful in the sequel.

\begin{lemma}
\label{lem:osc}
Let $f \in C^2(0,\infty)$ with $f(\xi)=0$ for all $\xi\geq a$ and some $a>0$.
Suppose further that there exists a constant $c>0$
such that $|f^{(k)}(\xi)|\leq c\xi^{\sigma-k}$ for a $\sigma \in (-1,1)$ and
all $\xi>0$, $k\in \{0,1,2\}$.
Then we have the estimate
\[ \left | \int_0^\infty e^{it \xi}f(\xi)d\xi \right |\lesssim c \langle t \rangle^{-1-\sigma}\]
for all $t\geq 0$.
\end{lemma}

\begin{proof}
If $t \in [0,1]$ we simply note
\[ \int_0^\infty \left |e^{it\xi}f(\xi)\right |d\xi \leq c \int_0^a \xi^{\sigma}d\xi\lesssim c. \]
Thus, we may focus on $t> 1$.
Oscillations can only be exploited if $t\xi\geq 1$ and thus, it is natural to distinguish
between $\xi\leq 1/t$ and $\xi\geq 1/t$.
In the case $\xi\leq 1/t$ we put absolute values inside and obtain
\[ \int_0^{1/t} \left |e^{it\xi}f(\xi)\right | d\xi 
\leq c\int_0^{1/t} \xi^\sigma d\xi
\lesssim c t^{-1-\sigma}. \]
In the case $\xi \geq 1/t$ we exploit oscillations by
performing two integrations by parts, i.e., 
\begin{align*} 
\int_{1/t}^\infty e^{it\xi}f(\xi)d\xi&=\left .
\frac{1}{it}e^{it\xi}f(\xi)\right |_{1/t}^\infty
-\frac{1}{it}\int_{1/t}^\infty e^{it\xi}f'(\xi)d\xi \\
&=O_\C(ct^{-1-\sigma})+ \left .\frac{1}{t^2}e^{it\xi}f'(\xi) \right |_{1/t}^\infty
+\frac{1}{t^2}\int_{1/t}^a O_\C(c\xi^{\sigma-2})d\xi \\
&=O_\C(ct^{-1-\sigma}).
\end{align*}
\end{proof}

\begin{lemma}
\label{lem:sinL1Linf}
We have the estimate
\[ \left \|\langle \cdot \rangle^{-\frac12}\frac{\sin(t\sqrt{ A})}{\sqrt{ A}}
g\right\|_{L^\infty(\R_+)}\lesssim \langle t \rangle^{-\frac12}
\|\langle \cdot \rangle^\frac12 g\|_{L^1(\R_+)} \]
for all $t\geq 0$ and all $g\in \mc S(\R_+)$.
\end{lemma}

\begin{proof}
The functional calculus for $ A$ and Lemma \ref{lem:Fourier} yield
\begin{align*} 
\frac{\sin(t\sqrt{ A})}{\sqrt{ A}}g(x)&=\int_0^\infty \frac{\sin(t\sqrt{\lambda})}
{\sqrt \lambda}\phi(x,\lambda)\mc F g(\sqrt \lambda)\rho(\lambda)d\lambda \\
&=2\lim_{N\to\infty}\int_{1/N}^N \int_0^\infty \sin(t\xi)\phi(x,\xi^2)\phi(y,\xi^2)g(y)\rho(\xi^2)dy d\xi.
\end{align*}
By Fubini-Tonelli we may interchange the order of integration and obtain
\[ 
\frac{\sin(t\sqrt{ A})}{\sqrt{ A}}g(x)=\lim_{N\to\infty}\int_0^\infty K_N(x,y;t)g(y)dy
\]
where
\[ K_N(x,y;t)=2\int_{1/N}^N \sin(t\xi)\phi(x,\xi^2)\phi(y,\xi^2)\rho(\xi^2)d\xi. \]
Thus, it suffices to establish the bound $\langle x\rangle^{-\frac12}
|K_N(x,y;t)|\langle y\rangle^{-\frac12}\lesssim \langle t\rangle^{-\frac12}$ for all $N \in \N$ and
all $x,y,t \in \R_+$.

We distinguish different cases and to this end we introduce a smooth cut-off $\chi$
satisfying $\chi(x)=1$ for $|x|\leq \frac12$ and $\chi(x)=0$ for $|x|\geq 1$.
We set 
\[ K_{N,1}(x,y;t):=2\int_{1/N}^N \chi(\xi)\chi(x\xi)\chi(y\xi)\sin(t\xi)\phi(x,\xi^2)\phi(y,\xi^2)\rho(\xi^2)d\xi. \]
Lemmas \ref{lem:rho}, \ref{lem:phi}, and Corollary \ref{cor:phi} yield
\[ \chi(\xi)\chi(x\xi)\chi(y\xi)\phi(x,\xi^2)\phi(y,\xi^2)\rho(\xi^2)=\langle x\rangle^\frac12 \langle y\rangle^\frac12 \chi(\xi)O(x^0 y^0 \xi^0) \]
where the $O$-term behaves like a symbol.
Thus, we have
\[ \langle x\rangle^{-\frac12}K_{N,1}(x,y;t)\langle y\rangle^{-\frac12}
=\int_{1/N}^N \sin(t\xi)\chi(\xi)O(x^0 y^0 \xi^0)d\xi=O(\langle t \rangle^{-1}) \]
by Lemma \ref{lem:osc}.

Next, we consider
\[ K_{N,2}(x,y;t):=2\int_{1/N}^N \chi(\xi)[1-\chi(x\xi)]\chi(y\xi)\sin(t\xi)\phi(x,\xi^2)\phi(y,\xi^2)\rho(\xi^2)d\xi. \]
From Lemmas \ref{lem:rho}, \ref{lem:phi}, and Corollary \ref{cor:phi} we infer
\begin{align*} 
\chi(\xi)[1-\chi(x\xi)]\chi(y\xi)\phi(x,\xi^2)\phi(y,\xi^2)\rho(\xi^2)=&\langle y\rangle^\frac12 \chi(\xi)
[1-\chi(x\xi)]\xi^{-\frac12} \\
&\times \Re [e^{ix\xi}O_\C(x^0 \xi^0)]
\end{align*}
with an $O_\C$-term that behaves like a symbol.
Consequently, it suffices to estimate the expression
\begin{align*} 
I(t,x):=\langle x\rangle^{-\frac12}\int_0^\infty e^{\pm i\xi(t\pm x)}\chi(\xi)[1-\chi(x\xi)]O_\C(\xi^{-\frac12})d\xi
 \end{align*}
 where any sign combination in the exponent of the exponential may occur and the $O_\C$-term
 behaves like a symbol.
Since $|I(t,x)|\lesssim 1$ we may restrict ourselves to $t\geq 1$.
We distinguish two cases. If $|t\pm x|\geq \frac12 t$ we note that 
$[1-\chi(x\xi)]\langle x\rangle^{-\frac12}\xi^{-\frac12}=O(x^0 \xi^0)$ and apply Lemma \ref{lem:osc}. This yields
$|I(t,x)|\lesssim |t\pm x|^{-1}\lesssim t^{-1}$.
If $|t\pm x|\leq \frac12 t$ we have $x\geq \frac12 t$ and we use $\langle x\rangle^{-\frac12}\lesssim t^{-\frac12}$
to conclude $|I(t,x)|\lesssim t^{-\frac12}$.
 In summary, we obtain 
 \[ \langle x\rangle^{-\frac12}|K_{N,2}(x,y;t)|\langle y\rangle^{-\frac12}\lesssim \langle t\rangle^{-\frac12}. \]
By symmetry, the
same bound is true for $K_{N,3}(x,y;t):=K_{N,2}(y,x;t)$.

We continue with
\begin{align*} 
K_{N,4}(x,y;t):=&2\int_{1/N}^N \chi(\xi)[1-\chi(x\xi)][1-\chi(y\xi)] \\
&\times \sin(t\xi)\phi(x,\xi^2)\phi(y,\xi^2)\rho(\xi^2)d\xi. 
\end{align*}
By Lemmas \ref{lem:rho} and \ref{lem:phi} this reduces to the study of integrals of the type
\begin{align*}
I(t,x,y):=&\langle x\rangle^{-\frac12}\langle y\rangle^{-\frac12}\int_0^\infty e^{\pm i\xi(t\pm x\pm y)}\chi(\xi)
\\
&\times [1-\chi(x\xi)][1-\chi(y\xi)]O_\C(\xi^{-1})d\xi
\end{align*}
where any sign combination in the exponent may occur and the $O_\C$-term behaves like a symbol. Again, 
we may restrict ourselves to $t\geq 1$.
As before, we distinguish the cases $|t\pm x\pm y|\geq \frac12 t$ and $|t\pm x\pm y|\leq \frac12 t$. 
In the former case we use 
\[ [1-\chi(x\xi)][1-\chi(y\xi)]\langle x\rangle^{-\frac12}\langle y\rangle^{-\frac12}O_\C(\xi^{-1})=O_\C(x^0 y^0 \xi^0)\] 
and apply Lemma \ref{lem:osc} which yields $|I(t,x,y)|\lesssim |t\pm x\pm y|^{-1}\lesssim t^{-1}$.
If, on the other hand, $|t\pm x\pm y|\leq \frac12 t$ then we must have $x\gtrsim t$ or $y\gtrsim t$.
Consequently, we infer 
\[ \langle x\rangle^{-\frac12}\langle y\rangle^{-\frac12}[1-\chi(x\xi)][1-\chi(y\xi)]O_\C(\xi^{-1})=t^{-\frac12}
O_\C(x^0 y^0 \xi^{-\frac12}) \]
and this yields $|I(t,x,y)|\lesssim t^{-\frac12}$.
We conclude that 
\[ \langle x\rangle^{-\frac12}K_{N,4}(x,y;t)\langle y \rangle^{-\frac12}\lesssim \langle t\rangle^{-\frac12} \]
as desired.

Finally, we turn to the large frequency case $\xi\geq 1$ and set
\[ K_{N,5}(x,y;t)=2\int_{1/N}^N [1-\chi(\xi)]\sin(t\xi)\phi(x,\xi^2)\phi(y,\xi^2)\rho(\xi^2)d\xi. \]
Lemmas \ref{lem:rho}, \ref{lem:phi} imply
\[ [1-\chi(\xi)]\phi(x,\xi^2)\phi(y,\xi^2)\rho(\xi^2)=-\tfrac{1}{\pi}\cos(x\xi)\cos(y\xi)\xi^{-1}+O(x^0 y^0 \xi^{-2}) \]
and by the trigonometric identity
\begin{align*} 
\sin a \cos b \cos c=&\tfrac14 [\sin(a+b-c)+\sin(a-b-c) \\
&+\sin(a-b+c)+\sin(a+b+c)] 
\end{align*}
this case reduces to estimating the integrals
\begin{align*} 
I_1(t,x,y)&:=\langle x\rangle^{-\frac12}\langle y \rangle^{-\frac12}\int_0^\infty 
[1-\chi(\xi)]\frac{\sin((t\pm x\pm y)\xi)}{\xi}d\xi \\
I_2(t,x,y)&:=\langle x \rangle^{-\frac12}\langle y \rangle^{-\frac12}\int_0^\infty [1-\chi(\xi)]e^{\pm i\xi(t\pm x\pm y)}O_\C(\xi^{-2})d\xi
\end{align*}
where any sign combination may occur.
Obviously, $|I_2(t,x,y)|\lesssim 1$ and since
\begin{equation}
\label{eq:intsin} 
\left |\int_1^\infty \frac{\sin(a\xi)}{\xi}d\xi \right |\lesssim 1 
\end{equation}
for all $a\in \R$, we also have $|I_1(t,x,y)|\lesssim 1$.
Consequently, it suffices to consider $t\geq 1$.
If $|t\pm x\pm y|\geq \frac12 t$ we integrate by parts and obtain $|I_j(t,x,y)|\lesssim 
|t\pm x\pm y|^{-1}\lesssim t^{-1}$, $j\in \{1,2\}$.
In the case $|t\pm x\pm y|\leq \frac12 t$ we use $\langle x\rangle^{-\frac12}\langle y \rangle^{-\frac12}\lesssim t^{-\frac12}$
and Eq.~\eqref{eq:intsin} to conclude $|I_j(t,x,y)|\lesssim t^{-\frac12}$, $j\in \{1,2\}$.
As a consequence, we infer 
\[ \langle x\rangle^{-\frac12}|K_{N,5}(x,y;t)|\langle y\rangle^{-\frac12}\lesssim \langle t\rangle^{-\frac12} \]
and, since $K_N=\sum_{k=1}^5 K_{N,k}$, this finishes the proof.
\end{proof}

\begin{remark}
\label{rem:sinL1Linf}
Inspection of the proof of Lemma \ref{lem:sinL1Linf} shows that we also have the bound
\[ \left \|\langle \cdot \rangle^{-\epsilon} \frac{\sin(t\sqrt A)}{\sqrt A}g\right \|_{L^\infty(\R_+)}
\lesssim \tfrac{1}{\epsilon} \|g\|_{L^1(\R_+)} \]
for any $\epsilon \in (0,1)$. This comment will be useful in connection with energy bounds, see below.
\end{remark}

Similarly, we obtain the following bound for the cosine evolution.

\begin{lemma}
\label{lem:cosL1Linf}
We have
\[ \left \|\langle \cdot \rangle^{-\frac12}
\cos(t\sqrt{ A})f \right \|_{L^\infty(\R_+)}\lesssim \langle t\rangle^{-\frac12}\left (
\|\langle \cdot \rangle^\frac12 f'\|_{L^1(\R_+)}
+\|\langle \cdot \rangle^\frac12f\|_{L^1(\R_+)} \right )
 \]
for all $t\geq 0$ and all $f\in \mc S(\R_+)$.
\end{lemma}

\begin{proof}
As in the proof of Lemma \ref{lem:sinL1Linf} it suffices to bound the kernel
\[ K_N(x,y;t)=2 \int_{1/N}^N \xi\cos(t\xi)\phi(x,\xi^2)\phi(y,\xi^2)\rho(\xi^2)d\xi. \]
We distinguish between small and large frequencies by decomposing
$K_N=K_N^-+K_N^+$ with
\[ K_N^-(x,y;t)=2\int_{1/N}^N \chi(\xi)\xi\cos(t\xi)\phi(x,\xi^2)\phi(y,\xi^2)\rho(\xi^2)d\xi \]
where $\chi$ is the cut-off introduced in the proof of Lemma \ref{lem:sinL1Linf}.
From the proof of Lemma \ref{lem:sinL1Linf} we immediately obtain the bound
\[ \langle x\rangle^{-\frac12}|K_N^-(x,y;t)|\langle y\rangle^{-\frac12}\lesssim \langle t \rangle^{-\frac12} \]
for all $t,x,y \in \R_+$ and all $N\in \N$.
Thus, we may restrict ourselves to $K_N^+(x,y;t)$.
We recall from Lemma \ref{lem:phi} that, for $\xi\geq \frac12$,
\[ \phi(y,\xi^2)=2 \Re[e^{i\xi y}b(y,\xi)] \]
where $b(y,\xi)=-\frac12+O_\C(\xi^{-1})+
O_\C(\langle y\rangle^{-1}\xi^{-1})$ with symbol behavior of both $O_\C$-terms.
Consequently, we obtain
\begin{align*} 
\phi(y,\xi^2)&=2\xi^{-1}\partial_y \Im[e^{i\xi y}b(y,\xi)]+O(\langle y\rangle^{-2}\xi^{-2}) \\
&=2\xi^{-1}\partial_y \Im [e^{i\xi y}b(y,\xi)-b(0,\xi)]+O(\langle y\rangle^{-2}\xi^{-2})
\end{align*}
and an integration by parts yields
\begin{align*} 
\mc F f(\xi)&=\int_0^\infty \phi(y,\xi^2)f(y)dy \\
&=2\xi^{-1}\Im[e^{i\xi y}b(y,\xi)-b(0,\xi)]f(y)\Big |_0^\infty \\
&\quad -2\xi^{-1}\int_0^\infty 
\Im[e^{i\xi y}b(y,\xi)-b(0,\xi)]f'(y)dy \\
&\quad +\int_0^\infty O(\langle y\rangle^{-2}\xi^{-2})f(y)dy \\
&=-2\xi^{-1}\int_0^\infty \Im[e^{i\xi y}b(y,\xi)]f'(y)dy+\int_0^\infty O(y^0\xi^{-2})[f'(y)+f(y)]dy
\end{align*}
by noting that $\Im b(0,\xi)=O(\xi^{-1})$. 
Consequently, the extra factor $\xi^{-1}$ and the fact that $\Im[e^{i\xi y}b(y,\xi)]$ behaves
like $\sin(y\xi)$ to leading order for large $\xi$, reduces us to the situation treated in the
proof of Lemma \ref{lem:sinL1Linf}.
\end{proof}

\subsection{Improved decay}

Next, we improve the decay in time at the cost of introducing stronger weights in space.

\begin{lemma}
\label{lem:sinL1Linfw}
We have the bound
\[ \left \|\langle \cdot \rangle^{-1}\frac{\sin(t\sqrt{ A})}{\sqrt{ A}}g \right \|_{L^\infty(\R_+)}
\lesssim \langle t\rangle^{-1} \|\langle \cdot \rangle g\|_{L^1(\R_+)} \]
for all $t\geq 0$ and all $g \in \mc S(\R_+)$.
\end{lemma}

\begin{proof}
The claim follows by inspection of the proof of Lemma \ref{lem:sinL1Linf}.
The point is that the terms treated by integration by parts (or Lemma \ref{lem:osc}) already
exhibited the stronger decay $\langle t\rangle^{-1}$.
Only the terms where the decay comes from the weight yielded the weaker $\langle t\rangle^{-\frac12}$.
With the stronger weights $\langle x\rangle^{-1}\langle y\rangle^{-1}$ we now obtain
the decay $\langle t \rangle^{-1}$ also for these terms.
\end{proof}

As usual, a similar result holds for the cosine evolution.

\begin{lemma}
\label{lem:cosL1Linfw}
We have the bound
\[ \left \| \langle \cdot \rangle^{-1}\cos(t\sqrt{ A})f \right \|_{L^\infty(\R_+)}
\lesssim \langle t \rangle^{-1}\left ( \|\langle \cdot\rangle f'\|_{L^1(\R_+)}
+\|\langle \cdot \rangle f\|_{L^1(\R_+)} \right ) \]
for all $t\geq 0$ and all $f\in \mc S(\R_+)$.
\end{lemma}

\begin{proof}
The claim follows from the proof of Lemma \ref{lem:sinL1Linfw} by inspection.
\end{proof}

As a consequence, we obtain the following weighted decay bounds.

\begin{corollary}
\label{cor:disp}
The solution $u$ to the initial value problem \eqref{eq:init}
satisfies the bound
\begin{align*} 
\left \|\langle \cdot \rangle^{-\sigma}u(t, \cdot)\right \|_{L^\infty(\R_+)}
\lesssim \langle t\rangle^{-\sigma}\Big (&\|\langle \cdot \rangle^\sigma f'\|_{L^1(\R_+)}
+\|\langle \cdot \rangle^\sigma f\|_{L^1(\R_+)} \\
&+\|\langle \cdot \rangle^\sigma g\|_{L^1(\R_+)}
\Big )
 \end{align*}
for all $t\geq 0$, $\sigma \in [\frac12,1]$, and all $f,g \in \mc S(\R_+)$.
\end{corollary}

\begin{proof}
Since the solution $u$ is given by
\[ u(t,\cdot)=\cos(t\sqrt{A})f+\frac{\sin(t\sqrt{A})}{\sqrt{A}}g, \]
the statement is a consequence of
Lemmas \ref{lem:sinL1Linf}, \ref{lem:cosL1Linf}, 
\ref{lem:sinL1Linfw}, \ref{lem:cosL1Linfw}, and a simple interpolation argument.
\end{proof}

\subsection{Comparison with the free case}
\label{sec:free}

As we have seen, the
decisive quantity that determines the decay of the wave evolution is 
$\phi(x,\lambda)\phi(y,\lambda)\rho(\lambda)$ for $\lambda>0$ small.
For simplicity we set $x=y=1$ and recall the estimate
$|\phi(1,\lambda)^2\rho(\lambda)|\lesssim 1$ for $\lambda$ small.
It is instructive to compare this behavior with the free case.
In fact, there are two different ``free'' cases one might consider.
For the free one-dimensional case we replace the operator $A$ by $A_1 f:=-f''$
with $\mc D(A_1)=\mc D(A)$, i.e., we retain the Neumann condition at $x=0$.
Note that this is in fact a resonant case (with $f(x)=1$ being the resonance function).
If we denote the spectral quantities associated to $A_1$ by a subscript ``1'', we infer
\begin{align*}
\phi_1(x,\lambda)&=-\cos(x\lambda^\frac12) \\
\theta_1(x,\lambda)&=\lambda^{-\frac12}\sin(x\lambda^\frac12) \\
f_{+,1}(x,\xi)&=e^{ix\xi},
\end{align*} which yields 
\[ \rho_1(\lambda)=\tfrac{1}{\pi}\Im \frac{W(\theta_1(\cdot,\lambda),f_{+,1}(\cdot,\lambda^\frac12))}{
W(f_{+,1}(\cdot,\lambda^\frac12),\phi_1(\cdot,\lambda))}=\tfrac{1}{\pi}\Im \frac{-1}{i\lambda^\frac12}
=\tfrac{1}{\pi}\lambda^{-\frac12} \]
and thus, 
$|\phi_1(1,\lambda)^2 \rho_1(\lambda)|\simeq \lambda^{-\frac12}$ for $\lambda$ small.
Consequently, $\phi_1(1,\lambda)^2\rho_1(\lambda)$ is much more singular than $\phi(1,\lambda)^2\rho(\lambda)$
and there is no way to compensate for this singularity by introducing weights in $x$.
This, of course, is a manifestation of the fact that there is no dispersive decay for free one-dimensional
wave evolution.

In order to compare $\phi(1,\lambda)^2\rho(\lambda)$ with the free two-dimensional case, we set 
\[ A_2 f(x):=-f''(x)-\tfrac{1}{4x^2}f(x) \]
which is a strongly singular Schr\"odinger operator in the sense of \cite{GesZin06}.
A fundamental system for $A_2 f(x)=0$ is given by $x^\frac12$ and $x^\frac12 \log x$.
Consequently, the endpoint $x=0$ is in the limit-circle case.
Recall that if $A_2 f=0$ then $|\cdot|^{-\frac12}f$ satisfies the radial two-dimensional
Laplace equation and thus, the correct boundary condition is $f(x)\simeq x^\frac12$ as $x\to 0+$
which leads to the domain
\[ \mc D(A_2)=\{f\in L^2(\R_+): f,f' \in \mathrm{AC}_\mathrm{loc}(\R_+), W(|\cdot|^\frac12, f)=0, f'' \in L^2(\R_+)\}. \]
The operator $A_2$ can be studied by Weyl-Titchmarsh theory \cite{EveKal07} but for our purposes 
it is easier to use the (self-inverse) radial two-dimensional Fourier transform
\[ \tilde F(R)=2\pi \int_0^\infty J_0(2\pi r R)\tilde f(r)r dr, \]
see e.g.~\cite{SteSha03}.
Setting 
\[ f(x):=(2\pi)^{-\frac14}x^\frac12 \tilde f((2\pi)^{-\frac12}x),\quad F(\xi):=(2\pi)^{-\frac14}\xi^\frac12 \tilde F((2\pi)^{-\frac12}\xi) \]
we find
\begin{equation}
\label{eq:Hankel}
 F(\xi)=\int_0^\infty (x\xi)^\frac12 J_0(x\xi)f(x)dx
\end{equation}
which is known as the Hankel transform. The map $f\mapsto F$ defined by Eq.~\eqref{eq:Hankel} is again 
self-inverse which means that $f$ can be reconstructed from $F$ by the formula
\[ f(x)=\int_0^\infty (x\xi)^\frac12 J_0(x\xi)F(\xi)d\xi. \]
Since $\phi_2(x,\xi^2):=(x\xi)^\frac12 J_0(x\xi)$ satisfies
\[ (-\partial_x^2-\tfrac{1}{4x^2})\phi_2(x,\xi^2)=\xi^2 \phi_2(x,\xi^2),\qquad W(|\cdot|^\frac12, \phi_2(\cdot,\xi))=0, \]
the Hankel transform Eq.~\eqref{eq:Hankel} yields the
spectral transformation $\mc U_2$ associated to $A_2$, upon switching from $\xi$
to the spectral parameter $\lambda=\xi^2$.
Explicitly, we obtain
\[ \mc U_2 f(\lambda)=\int_0^\infty \phi_2(x,\lambda)f(x)dx \]
with inverse
\[ \mc U_2^{-1} \hat f(x)=\int_0^\infty \phi_2(x,\lambda)\hat f(\lambda)\rho_2(\lambda)d\lambda \]
where $\rho_2(\lambda)=\frac12 \lambda^{-\frac12}$.
Consequently, we infer $\phi_2(1,\lambda)^2 \rho_2(\lambda)\simeq \langle \lambda \rangle^{-\frac12}$ 
which shows that $\phi(1,\lambda)^2 \rho(\lambda)$ behaves exactly like in the free
two-dimensional case.

\section{Energy bounds}
\label{sec:energy}
Next, we turn to energy bounds. 
Of course, we have energy conservation for the solution $u$ of Eq.~\eqref{eq:init} in the sense that
\[ \|\sqrt A u(t,\cdot)\|_{L^2(\R_+)}^2+\|\partial_t u(t,\cdot)\|_{L^2(\R_+)}^2=\|\sqrt A f\|_{L^2(\R_+)}^2
+\|g\|_{L^2(\R_+)}^2 \]
for all $t\geq 0$.
In the following, we prove some generalizations of this basic energy bound.

\subsection{Properties of the distorted Fourier transform}

We start with a simple result that relates the derivative on the physical side
to a weight on the distorted Fourier side.
\begin{lemma}
\label{lem:derwei}
Let $f\in \mc S(\R_+)$ satisfy $f'(0)=0$. Then we have
\[ \|f\|_{H^2(\R_+)}\simeq \|\langle \cdot \rangle^2 \mc F f\|_{L^2_{\tilde \rho}(\R_+)}. \]
\end{lemma}

\begin{proof}
First of all, by Plancherel we have
\[ \|f\|_{L^2(\R_+)}= \|\mc F f\|_{L^2_{\tilde \rho}(\R_+)}\lesssim
\|\langle \cdot \rangle^2 \mc F f\|_{L^2_{\tilde \rho}(\R_+)}. \]
Furthermore, an integration by parts yields
\begin{align*} \mc F(f'')(\xi)&=\int_0^\infty \phi(x,\xi^2)f''(x)dx=\int_0^\infty \partial_x^2 \phi(x,\xi^2)f(x)dx
\end{align*}
where the boundary term at zero vanishes thanks to $f'(0)=\phi'(0,\xi^2)=0$.
Now recall that $[-\partial_x^2+V(x)]\phi(x,\xi^2)=\xi^2 \phi(x,\xi^2)$ and thus,
\begin{equation}
\label{eq:derwei} 
\mc F(f'')(\xi)=-\xi^2 \mc F f(\xi)+\mc F(Vf). 
\end{equation}
Since $V \in L^\infty(\R_+)$ we infer by Plancherel 
\[ \|f''\|_{L^2(\R_+)}\lesssim \|| \cdot |^2 \mc F f\|_{L^2_{\tilde \rho}(\R_+)}
+\|Vf\|_{L^2(\R_+)}\lesssim \|\langle \cdot\rangle^2 \mc F f\|_{L^2_{\tilde \rho}(\R_+)} \]
and this shows $\|f\|_{H^2(\R_+)}\lesssim \|\langle \cdot \rangle^2\mc F f\|_{L^2_{\tilde \rho}(\R_+)}$.
For the reverse inequality we proceed analogously and estimate
\begin{align*}
\|\langle \cdot \rangle^2 \mc F f\|_{L^2_{\tilde \rho}(\R_+)}&\lesssim \||\cdot|^2 \mc F f\|_{L^2_{\tilde \rho}(\R_+)}
+\|\mc F f\|_{L^2_{\tilde \rho}(\R_+)} \\
&\lesssim \|\mc F(f'')\|_{L^2_{\tilde \rho}(\R_+)}+\|\mc F (Vf)\|_{L^2_{\tilde \rho}(\R_+)}+\|\mc F f\|_{L^2_{\tilde \rho}(\R_+)} \\
&\lesssim \|f\|_{H^2(\R_+)}
\end{align*}
by Eq.~\eqref{eq:derwei}, Plancherel, and $V\in L^\infty(\R_+)$.
\end{proof}

Lemma \ref{lem:derwei} easily generalizes to higher Sobolev spaces.
Note carefully that we need an additional assumption on the potential $V$ here.

\begin{lemma}
\label{lem:derweis}
Let $f\in \mc S(\R)$ be even and assume that $V^{(2j+1)}(0)=0$ for all $j\in \N_0$. 
Then, for any $s\geq 0$, we have
\[ \|f\|_{H^s(\R_+)}\simeq \|\langle \cdot \rangle^s \mc F f\|_{L^2_{\tilde \rho}(\R_+)}. \]
\end{lemma}

\begin{proof}
By complex interpolation it suffices to prove the statement for $s=2k$ where $k\in \N$.
From the equation $[-\partial_x^2+V(x)]\phi(x,\xi^2)=\xi^2 \phi(x,\xi^2)$ we obtain by 
repeated differentiation that $\phi^{(2j+1)}(0,\xi^2)=0$ for all $j\in \N_0$ (here
we use $\phi'(0,\xi^2)=0$ and the assumption on $V$).
Consequently, integration by parts yields
\[ \mc F(f^{(2k)})(\xi)=\int_0^\infty \partial_x^{2k}\phi(x,\xi^2)f(x)dx \]
where all boundary terms vanish since $f^{(2j+1)}(0)=\phi^{(2j+1)}(0,\xi^2)=0$ for all $j\in \N_0$.
Furthermore, we have
\[ [-\partial_x^2+V(x)]^{k}\phi(x,\xi^2)=\xi^{2k} \phi(x,\xi^2) \]
and since $V\in C^\infty(\R_+)$ with $\|V^{(j)}\|_{L^\infty(\R_+)}\leq C_j$ for all $j\in \N_0$,
the statement is proved inductively by following the reasoning in the proof of Lemma \ref{lem:derwei}.
\end{proof}

We also need a slight variant of Lemma \ref{lem:derweis} with the same conclusion but slightly
different assumptions.
First, however, we make the following simple observation.

\begin{lemma}
\label{lem:F-1Cinfc}
Let $f\in C^\infty_c(0,\infty)$. Then $\mc F^{-1}f \in \mc S(\R_+)$.
If, in addition, $V^{(2j+1)}(0)=0$ for all $j\in \N_0$, then we also have
\[ \left (\mc F^{-1}f\right )^{(2j+1)}(0)=0 \] 
for all $j\in \N_0$.
\end{lemma}

\begin{proof}
From Lemma \ref{lem:phi} and the equation
\[ -\phi''(x,\xi^2)+V(x)\phi(x,\xi^2)=\xi^2 \phi(x,\xi^2) \]
with $\|V^{(k)}\|_{L^\infty(\R_+)}\leq C_k$ for all $k\in \N_0$ it follows that
$|\partial_x^k \phi(x,\xi^2)|\leq C_k (\xi \langle \xi\rangle^{-1})^{-\frac12-}\langle \xi\rangle^k$ for all $k\in \N_0$.
We have
\[ \mc F^{-1}f(x)=\int_0^\infty \phi(x,\xi^2)f(\xi)\tilde \rho(\xi)d\xi \]
and thus, $\mc F^{-1}f \in C^\infty(\R_+)$ since $f\in C_c^\infty(0,\infty)$.
If $V^{(2j+1)}(0)=0$ for all $j\in \N_0$ it follows that $\phi^{(2j+1)}(0,\xi^2)=0$ and thus,
$(\mc F^{-1}f)^{(2j+1)}(0)=0$ for all $j\in \N_0$.
Furthermore, if $x\xi\leq 1$ we may trade $\langle x\rangle$ for $\xi^{-1}$ and we
immediately infer $|\mc F^{-1}f(x)|\leq C_N \langle x\rangle^{-N}$
for any $N\in \N$ since $f$ has support away from the origin.
In the case $x\xi\geq 1$ we exploit the oscillations in $\phi$ (see Lemma \ref{lem:phi})
and perform integrations by parts to arrive at the same conclusion.
Applying this logic to the derivatives as well yields $\mc F^{-1}f\in \mc S(\R_+)$.
\end{proof}

\begin{lemma}
\label{lem:mapF-1}
Let $f\in L^2_{\tilde \rho}(\R_+)$, $s\geq 0$, and assume that
$\|\langle \cdot \rangle^s f\|_{L^2_{\tilde \rho}(\R_+)}\lesssim 1$.
Furthermore, assume that $V^{(2j+1)}(0)=0$ for all $j\in \N_0$. 
Then $\mc F^{-1} f\in H^s(\R_+)$ and we have
\[ \|\mc F^{-1}f\|_{H^s(\R_+)}\simeq \|\langle \cdot\rangle^s f\|_{L^2_{\tilde \rho}(\R_+)}. \]
\end{lemma}

\begin{proof}
We approximate $f$ by a sequence $(f_n)\subset C^\infty_c(0,\infty)$, i.e.,
\[ \|\langle \cdot \rangle^s (f-f_n)\|_{L^2_{\tilde \rho}(\R_+)}\to 0 \] as $n\to\infty$.
By Lemma \ref{lem:F-1Cinfc} we infer $\mc F^{-1}f_n \in \mc S(\R_+)$ for all $n\in \N$
and by the assumption on $V$, $\mc F^{-1}f_n$ extends to an even function in $\mc S(\R)$.
Consequently, Lemma \ref{lem:derweis} yields
\[ \|\mc F^{-1}f_n-\mc F^{-1}f_k\|_{H^s(\R_+)}\simeq \|\langle \cdot \rangle^s (f_n-f_k)\|_{L^2_{\tilde \rho}(\R_+)}
\to 0 \]
as $n,k \to \infty$.
Thus, there exists a $g\in H^s(\R_+)$ such that $\mc F^{-1}f_n \to g$ in $H^s(\R_+)$ and we have
$\|g\|_{H^s(\R_+)}\simeq \|\langle \cdot\rangle^s f\|_{L^2_{\tilde \rho}(\R_+)}$.
On the other hand, by the unitarity of $\mc F$ we have $\mc F^{-1}f_n \to \mc F^{-1}f$ in $L^2(\R_+)$
and thus, $g=\mc F^{-1}f$.
\end{proof}

\subsection{Generalized energy bounds}
After these preliminaries we turn to the energy bounds.
First, we prove a simple result which shows that under appropriate assumptions
one may formally differentiate the operators $\cos(t\sqrt A)$ and $\sin(t\sqrt A)$
with respect to $t$.
In the following, we will make freely use of this result.

\begin{lemma}
\label{lem:difft}
Let $f\in \mc S(\R)$ be even and assume $V^{(2j+1)}(0)=0$ for all $j\in \N_0$.
Then $\cos(t\sqrt A)f(x)$ and $\sin(t\sqrt A)f(x)$ can be differentiated with respect to $t$
and the derivatives are given by the respective formal expressions.
For instance, we have
\[ \partial_t [\cos(t\sqrt A)f(x)]=-\sqrt A \sin(t\sqrt A)f(x) \]
and analogously for the higher derivatives.
\end{lemma}

\begin{proof}
The assumptions on $f$ and $V$ as well as Lemma \ref{lem:derweis} imply 
\begin{equation}
\label{eq:Ffs} 
\|\langle \cdot \rangle^s \mc F f\|_{L^2_{\tilde \rho}(\R_+)}\lesssim 1 
\end{equation}
for any $s\geq 0$.
By definition, the operator $\cos(t\sqrt A)$ is given by
\begin{align*} 
\cos(t\sqrt A)f(x)&=\mc F^{-1}\left [\cos(t|\cdot|)\mc F f\right ](x) \\
&=\int_0^\infty \phi(x,\xi^2)\cos(t\xi)\mc F f(\xi)\tilde \rho(\xi)d\xi.
\end{align*} 
Consequently, by Eq.~\eqref{eq:Ffs} and dominated convergence we infer
\begin{align*} \partial_t [\cos(t\sqrt A)f(x)]&=\int_0^\infty 
\phi(x,\xi^2)[-\xi\sin(t\xi)]\mc F f(\xi)\tilde \rho(\xi)d\xi \\
&=-\sqrt A \sin(t\sqrt A)f(x). 
\end{align*}
The corresponding results for the higher derivatives and the sine evolution follow accordingly.
\end{proof}

\begin{lemma}
\label{lem:energy}
Let $f,g \in \mc S(\R)$ be even and assume that $V^{(2j+1)}(0)=0$ for all $j\in \N_0$.
Then the solution $u$ of the initial value problem \eqref{eq:init} satisfies the bounds
\begin{align*} 
\|\partial_t^\ell \sqrt A u(t,\cdot)\|_{H^k(\R_+)}&\leq C_{k,\ell} \left (
\|\sqrt A f\|_{H^{k+\ell}(\R_+)}+\|g\|_{H^{k+\ell}(\R_+)} \right ) \\
\|\partial_t^\ell u_t(t,\cdot)\|_{H^k(\R_+)}&\leq C_{k,\ell}\left (\|\sqrt A f\|_{H^{k+\ell}(\R_+)}
+\|g\|_{H^{k+\ell}(\R_+)} \right ) 
\end{align*}
for all $t\geq 0$ and all $k,\ell \in \N_0$.
\end{lemma}

\begin{proof}
We start with the sine evolution.
By definition, we have
\[ \mc F\left [\partial_t \frac{\sin(t\sqrt A)}{\sqrt A}g \right ](\xi)=\partial_t \frac{\sin(t\xi)}{\xi}\mc F g(\xi)
=\cos(t\xi)\mc F g(\xi) \]
and $\mc F[\sin(t\sqrt A)g](\xi)=\sin(t\xi)\mc F g(\xi)$.
This yields
\[ \left |\mc F\left (\partial_t^{\ell}\partial_t  \frac{\sin(t\sqrt A)}{\sqrt A}g \right )(\xi)\right |
\leq \langle \xi \rangle^{\ell}|\mc F g(\xi)| \]
and $|\mc F[\partial_t^\ell \sin(t\sqrt A)g](\xi)|\lesssim \langle \xi\rangle^\ell |\mc F g(\xi)|$.
Consequently, by Lemma \ref{lem:derweis} we infer
\begin{align*} \|g\|_{H^{k+\ell}(\R_+)}&\simeq \|\langle \cdot \rangle^{k+\ell}\mc Fg\|_{L^2_{\tilde \rho}(\R_+)}
\gtrsim \left \|\langle \cdot\rangle^k 
\mc F\left (\partial_t^{\ell}\partial_t  \frac{\sin(t\sqrt A)}{\sqrt A}g \right ) \right \|_{L^2_{\tilde \rho}(\R_+)}
\end{align*}
and Lemma \ref{lem:mapF-1} yields
\[ \|g\|_{H^{k+\ell}(\R_+)}\gtrsim \left \|\partial_t^{\ell}\partial_t  \frac{\sin(t\sqrt A)}{\sqrt A}g\right \|_{H^k(\R_+)}. \]
Analogously, we obtain $\|\partial_t^\ell \sin(t\sqrt A)g\|_{L^2(\R_+)}\lesssim \|g\|_{H^{k+\ell}(\R_+)}$.
The cosine evolution is treated in the same fashion\footnote{Note carefully that $f\in \mc S(\R)$ even and the assumption on $V$
imply $(\sqrt A f)^{(2j+1)}(0)=0$. This is a consequence of $\sqrt A f=\mc F^{-1}(|\cdot| \mc F f)$ and
Lemma \ref{lem:F-1Cinfc} combined with a standard approximation argument.}.
\end{proof}

A combination of the $L^\infty$ bound from Remark \ref{rem:sinL1Linf} with the energy bound
in Lemma \ref{lem:energy} allows us to control the \emph{free} energy as well.

\begin{lemma}
\label{lem:energyfree}
Let $f,g \in \mc S(\R)$ be even and assume that $V^{(2j+1)}(0)=0$ for all $j\in \N_0$.
Then the solution $u$ of the initial value problem \eqref{eq:init} satisfies the bounds
\[
\|\nabla u(t,\cdot)\|_{H^k(\R_+)}\leq C_{k} \left (
\|f\|_{H^{1+k}(\R_+)}+\|g\|_{H^{k}(\R_+)}+\|g\|_{L^1(\R_+)} \right ) 
\]
for all $t\geq 0$ and $k \in \N_0$.
\end{lemma}

\begin{proof}
We start with the case $k=0$.
For the sine evolution we have
\begin{align*} \left \|\sqrt A\frac{\sin(t\sqrt A)}{\sqrt A}g\right \|_{L^2(\R_+)}^2=&\left \|\nabla \frac{\sin(t\sqrt A)}{\sqrt A}
g \right \|_{L^2(\R_+)}^2 \\
&+\int_0^\infty V(x)\left |\frac{\sin(t\sqrt A)}{\sqrt A}g(x)\right |^2 dx. 
\end{align*}
By Remark \ref{rem:sinL1Linf} and $|V(x)|\lesssim \langle x\rangle^{-2}$ we have
\begin{align*} 
\left \||V|^\frac12 \frac{\sin(t\sqrt A)}{\sqrt A}g\right \|_{L^2(\R_+)}& 
\lesssim \left \|\langle \cdot \rangle^\epsilon |V|^\frac12 \right \|_{L^2(\R_+)}\left \|\langle \cdot\rangle^{-\epsilon} 
\frac{\sin(t\sqrt A)}{\sqrt A}g\right \|_{L^\infty(\R_+)} \\
&\lesssim \|g\|_{L^1(\R_+)}
\end{align*}
and thus, Lemma \ref{lem:energy} yields
\[ \left \|\nabla \frac{\sin(t\sqrt A)}{\sqrt A}
g \right \|_{L^2(\R_+)}\lesssim \|g\|_{L^1(\R_+)}+\|g\|_{L^2(\R_+)}. \]
For the cosine evolution we infer $\|\nabla \cos(t\sqrt A)f\|_{L^2(\R_+)}\lesssim \|f\|_{H^1(\R_+)}$ by Lemma \ref{lem:energy}
since
\begin{align*} 
\int_0^\infty |V(x)||\cos(t\sqrt A)f(x)|^2 dx&\lesssim \|V\|_{L^\infty(\R_+)}\|\cos (t\sqrt A)f\|_{L^2(\R_+)}^2 \\
&\lesssim \|f\|_{L^2(\R_+)}^2.
\end{align*}

For the higher derivatives we note that
\[ \nabla^2 u(t,\cdot)=-A u(t,\cdot)+V u(t,\cdot) \]
and, more generally,
\[ \nabla^k \nabla u(t,\cdot)=-\nabla^{k-1}Au(t,\cdot)+\nabla^{k-1}[Vu(t,\cdot)] \]
for $k\in \N$.
Thus, the claim follows inductively provided we obtain a suitable bound for 
$\|Au(t,\cdot)\|_{H^{k-1}(\R_+)}$.
This, however, is easy since we have
\[ \|A\cos(t\sqrt A)f\|_{H^{k-1}(\R_+)}\simeq \|\langle \cdot \rangle^{k-1} |\cdot|^2 \cos(t|\cdot|)
\mc F f\|_{L^2_{\tilde \rho}(\R_+)}\lesssim \|f\|_{H^{1+k}(\R_+)} \]
as well as
\[ \left \|A\frac{\sin(t\sqrt A)}{\sqrt A}g \right \|_{H^{k-1}(\R_+)}
\simeq \|\langle \cdot \rangle^{k-1}|\cdot|\sin(t|\cdot|)\mc F g\|_{L^2_{\tilde \rho}(\R_+)}
\lesssim \|g\|_{H^k(\R_+)} \]
by Lemmas \ref{lem:derweis} and \ref{lem:mapF-1}.
\end{proof}

\section{Vector field bounds}
\label{sec:vf}

In this section we develop a vector field method on the distorted Fourier side. 
In order to motivate our approach, we consider the free wave equation on ($1+d$)-dimensional Minkowski space
$\R^{1+d}$.
Recall the well-known scaling vector field\footnote{Here we use the Einstein summation convention with latin indices running from $1$ to $d$
and greek indices running from $0$ to $d$ where $d$ denotes the spatial dimension.
As usual, we set $x^0=t$.
Furthermore, we use the Minkowski metric 
$(\eta^{\mu \nu})=(\eta_{\mu \nu})=\mathrm{diag}(
-1,1,\dots,1)$ and write $\partial_\mu=\frac{\partial}{\partial x^\mu}$.
We also set $x_\mu=\eta_{\mu \nu}x^\nu$ and $\partial^\mu=\eta^{\mu \nu}\partial_\nu$.}
$S=t\partial_t + x^j\partial_j$
which, in four-dimensional notation, may be written as 
$S=x^\mu \partial_\mu$. 
We have $\partial_\nu (x^\mu \partial_\mu)=\partial_\nu+x^\mu \partial_\mu \partial_\nu$
and thus,
\begin{align*} 
\Box S=\partial^\nu \partial_\nu+\partial^\nu(x^\mu \partial_\mu \partial_\nu)=\partial^\nu\partial_\nu
+\eta^{\mu \nu}\partial_\mu \partial_\nu+x^\mu \partial_\mu \partial^\nu \partial_\nu
=2\Box+S \Box
\end{align*}
which implies the crucial commutator relation $[S,\Box]=-2\Box$.
Observe that the vector field $S$ has a nice representation on the Fourier side as well.
Indeed, if
\[ \mc F_d f(x):=\int_{\R^d}e^{-2\pi i \xi_k x^k}f(x)dx \]
denotes that standard Fourier transform on $\R^d$, we obtain
\[ [\mc F_d Su(t,\cdot)](\xi)=(t\partial_t-\xi_j \partial_{\xi_j}-d)[\mc F_d u(t,\cdot)](\xi). \]
In sloppy (but probably more intuitive) notation this may be written as
\[ \mc F_d(t\partial_t+x^j\partial_{x^j})=(t\partial_t-\xi_j \partial_{\xi_j}-d)\mc F_d. \]
The point is that the (standard) Fourier transform of $-\Delta$ looks \emph{formally} 
the same as the distorted Fourier transform of $-\Delta+V$ (provided there is no negative spectrum).
Thus, it is natural to \emph{define} a ``vector field'' $\Gamma$ by its action on the distorted Fourier side
given by $t\partial_t-\xi_j \partial_{\xi_j}-d$.
In order for this to be useful, the difference between $S$ and $\Gamma$ must be well-behaved
(and, in a sense, negligible). 
In the following we develop these ideas rigorously.

\subsection{The operator $B$}
The operator $B$ is defined as the error one makes if one replaces $x\partial_x$ on the distorted
Fourier side by $-\xi\partial_\xi-1$. 
More precisely, we write $Df(x):=x f'(x)$ and define $B$ by
\begin{equation}
\label{def:B}
 \mc F D f=-D\mc F f-\mc F f+B \mc F f. 
 \end{equation}
Thus, $B$ acts on the distorted Fourier side and it measures the deviation from the 
free case. 
Since
\begin{align*} 
\mc FD f(\xi)&=\int_0^\infty \phi(x,\xi^2)xf'(x)dx \\
&=-\int_0^\infty x\partial_x \phi(x,\xi^2)f(x)dx-\mc F f(\xi)
\end{align*}
for $f\in \mc S(\R_+)$, say, we obtain the explicit representation
\[ B\mc F f(\xi)=\int_0^\infty (\xi\partial_\xi-x\partial_x)\phi(x,\xi^2)f(x)dx. \]
It is important to realize that the operator $\xi\partial_\xi-x\partial_x$ annihilates any 
function of the form $f(x\xi)$, in particular $e^{\pm ix\xi}$.
Our first goal is to show that
$B$ is bounded on $L^2_{\tilde \rho}(\R_+)$.

\subsection{Preliminaries from distribution theory}
For the following we use some elements of distribution theory.
To fix notation, we write $J:=(0,\infty)$, $\mc D(J):=C_c^\infty(J)$, and, as usual, we equip
$\mc D(J)$ with the inductive limit topology, i.e., $\varphi_n\to \varphi$ in $\mc D(J)$
means that there exists a compact set $K\subset J$ such that $\supp(\varphi_n)\subset K$
for all $n$ and $\|\varphi_n^{(k)}-\varphi^{(k)}\|_{L^\infty(K)}\to 0$ as $n\to\infty$
for any $k \in \N_0$.
The space of distributions (i.e., bounded linear functionals on $\mc D(J)$) is denoted
by $\mc D'(J)$ and equipped with the weak topology, i.e., $u_n\to u$ in $\mc D'(J)$ means
$|u_n(\varphi)-u(\varphi)|\to 0$ as $n\to\infty$ for any $\varphi \in \mc D(J)$.
Finally, the tensor product $\varphi \otimes \psi \in \mc D(J\times J)$ 
of two test functions $\varphi,\psi \in \mc D(J)$
is defined by $(\varphi \otimes \psi)(x,y):=\varphi(x)\psi(y)$.
Note that $B$ may be viewed as a map from $\mc D(J)$ to $\mc D'(J)$.

\begin{lemma}
\label{lem:Schwartzkernel}
There exists a (unique) distribution $K\in \mc D'(J\times J)$, the Schwartz kernel
of $B: \mc D(J)\to \mc D'(J)$, such that
\[ (Bf)(\varphi)=K(\varphi \otimes f) \]
for all $f, \varphi \in \mc D(J)$. 
\end{lemma}
 
\begin{proof}
By the Schwartz Kernel Theorem (see e.g.~\cite{Hor03}, p.~128, Theorem 5.2.1) it suffices
to prove that $B: \mc D(J)\to \mc D'(J)$ is continuous, i.e., that $f_n \to 0$ in $\mc D(J)$
implies $Bf_n \to 0$ in $\mc D'(J)$ but this follows immediately by Lemma \ref{lem:F-1Cinfc}
and dominated convergence. 
\end{proof} 
 
The next simple observation will allow us to separate the diagonal from the off-diagonal
behavior of the Schwartz kernel.
For $j=1,2$ we write $\pi_j: J\times J\to J$, $\pi_1(x,y):=x$, $\pi_2(x,y):=y$.
Note that $\pi_j \in C^\infty(J\times J)$. 

\begin{lemma}
\label{lem:Schwartzod}
Let $B': \mc D(J)\to \mc D'(J)$ be defined by
\[ B'f:=|\cdot|^2 Bf-B(|\cdot|^2 f). \]
Then $B'$ is continuous and its Schwartz kernel $K' \in \mc D'(J\times J)$ is given by
\[ K'=(\pi_1^2-\pi_2^2) K \]
where $K$ is the Schwartz kernel of $B$.
\end{lemma}

\begin{proof}
Let $f, \varphi \in \mc D(J)$. Then
\begin{align*} 
(|\cdot|^2 Bf)(\varphi)&=(Bf)(|\cdot|^2 \varphi)=K(|\cdot|^2 \varphi \otimes f)
=K(\pi_1^2(\varphi \otimes f)) \\
&=\pi_1^2 K(\varphi \otimes f)
\end{align*}
and
\begin{align*}
[B(|\cdot|^2 f)](\varphi)&=K(\varphi \otimes |\cdot|^2 f)=K(\pi_2^2(\varphi \otimes f)) \\
&=\pi_2^2 K(\varphi\otimes f).
\end{align*}
\end{proof}

\subsection{The kernel of $B$ away from the diagonal}
After these preliminaries we turn to more substantial issues.
The next result yields an explicit expression for 
the Schwartz kernel $K'$ from Lemma \ref{lem:Schwartzod}.

\begin{lemma}
\label{lem:F}
For $f\in C_c^\infty(0,\infty)$ we have
\[ \xi^2 Bf(\xi)-B(|\cdot|^2 f)(\xi)=2 \int_0^\infty 
F(\xi,\eta)\eta \rho(\eta^2)f(\eta)d\eta \]
where
\[ F(\xi,\eta)=\int_0^\infty \phi(x,\xi^2)\phi(x,\eta^2)U(x)dx \]
and $U(x)=-2V(x)-xV'(x)$.
\end{lemma}

\begin{proof}
Explicitly, we have
\[ Bf(\xi)=2\int_0^\infty (\xi\partial_\xi-x\partial_x)\phi(x,\xi^2)\int_0^\infty \phi(x,\eta^2)f(\eta)\eta \rho(\eta^2)
d\eta dx. \]
Since $f\in C^\infty_c(0,\infty)$ implies $\mc F^{-1}f \in \mc S(\R_+)$, $Bf(\xi)$ is well-defined
for all $\xi> 0$.
By recalling that $[-\partial_x^2+V(x)]\phi(x,\eta^2)=\eta^2 \phi(x,\eta^2)$, we infer
\begin{align*} 
B(|\cdot|^2 f)(\xi)&=2\int_0^\infty (\xi\partial_\xi-x\partial_x)\phi(x,\xi^2)
\int_0^\infty \phi(x,\eta^2)\eta^2 f(\eta)\eta \rho(\eta^2)
d\eta dx \\
&=2\int_0^\infty (\xi\partial_\xi-x\partial_x)\phi(x,\xi^2) \\
&\quad \times \int_0^\infty [-\partial_x^2+V(x)]\phi(x,\eta^2)f(\eta)\eta \rho(\eta^2)
d\eta dx.
\end{align*}
Furthermore, by dominated convergence we obtain
\begin{align*} 
\int_0^\infty &[-\partial_x^2+V(x)]\phi(x,\eta^2)f(\eta)\eta \rho(\eta^2)d\eta \\
&=[-\partial_x^2+V(x)]\int_0^\infty \phi(x,\eta^2)f(\eta)\eta \rho(\eta^2)
d\eta 
\end{align*}
and integration by parts with respect to $x$ yields
\begin{align*}
B(|\cdot|^2 f)(\xi)&=\int_0^\infty [-\partial_x^2+V(x)](\xi\partial_\xi-x\partial_x)\phi(x,\xi^2)\mc F^{-1}f(x)dx
\end{align*}
where the boundary terms vanish thanks to $\mc F^{-1}f\in \mc S(\R_+)$ and $\phi'(0,\xi^2)=0$
(which implies $(\mc F^{-1}f)'(0)=0$).
Now observe that 
\begin{align*}
[-\partial_x^2+V(x)]x\partial_x&=x\partial_x [-\partial_x^2+V(x)]-2\partial_x^2-xV'(x) \\
&=x\partial_x [-\partial_x^2+V(x)]+2[-\partial_x^2+V(x)]+U(x)
\end{align*}
where $U(x):=-2V(x)-xV'(x)$.
Consequently, we infer
\begin{align}
\label{eq:Beta} 
B(|\cdot|^2 f)(\xi)=&\int_0^\infty (\xi\partial_\xi-x\partial_x)[\xi^2 \phi(x,\xi^2)]\mc F^{-1}f(x)dx \nonumber \\
&-2\int_0^\infty \xi^2 \phi(x,\xi^2)\mc F^{-1}f(x)dx \nonumber \\
&-\int_0^\infty U(x)\phi(x,\xi^2)\mc F^{-1}f(x)dx.
\end{align}
On the other hand, we have
\begin{align}
\label{eq:Bxi} 
\xi^2 Bf(\xi)&=\int_0^\infty \xi^2 (\xi\partial_\xi-x\partial_x)\phi(x,\xi^2)\mc F^{-1}f(x)dx \nonumber \\
&=\int_0^\infty (\xi\partial_\xi-x\partial_x)[\xi^2 \phi(x,\xi^2)]\mc F^{-1}f(x)dx \nonumber \\
&\quad -2\int_0^\infty \xi^2 \phi(x,\xi^2)\mc F^{-1}f(x)dx
\end{align}
and subtracting Eq.~\eqref{eq:Beta} from Eq.~\eqref{eq:Bxi} we find
\[ \xi^2 Bf(\xi)-B(|\cdot|^2 f)(\xi)=2\int_0^\infty U(x)\phi(x,\xi^2)\int_0^\infty \phi(x,\eta^2)
f(\eta)\eta \rho(\eta^2)d\eta dx. \]
Since $U(x)=O(\langle x\rangle^{-2-\alpha})$ (note the ``magic'' cancellation) and $|\phi(x,\xi^2)|\lesssim \langle x\rangle^{\frac12+}$ 
for fixed $\xi$ by Lemma \ref{lem:phi}, we may apply Fubini to change the order of integration
and this yields
\[  \xi^2 Bf(\xi)-B(|\cdot|^2 f)(\xi)=2\int_0^\infty F(\xi,\eta)\eta \rho(\eta^2)f(\eta)d\eta \]
where
\[ F(\xi,\eta)=\int_0^\infty \phi(x,\xi^2)\phi(x,\eta^2)U(x)dx. \]
\end{proof}

As a consequence of Lemmas \ref{lem:F} and \ref{lem:Schwartzod}, the Schwartz kernel
$K$ of $B$ away from the diagonal is a function and given by
\begin{equation}
\label{eq:Schwartzkernel}
 (\xi^2-\eta^2)K(\xi,\eta)=2F(\xi,\eta)\eta\rho(\eta^2),\quad \xi \not= \eta. 
 \end{equation}

\subsection{Bounds for $F$}

Next, we establish bounds for the function $F$ in Lemma \ref{lem:F}.

\begin{lemma}
\label{lem:boundsF}
The function $F$ given in Lemma \ref{lem:F} satisfies the bounds
\[ |F(\xi,\eta)|\lesssim 1,\quad |\partial_\xi F(\xi,\eta)|+|\partial_\eta F(\xi,\eta)|
\lesssim \xi^{-1}\langle \xi\rangle+\eta^{-1}\langle \eta \rangle \]
for all $\xi,\eta>0$.
\end{lemma}

\begin{proof}
We distinguish different cases and to this end we use the cut-off $\chi$ from the proof of
Lemma \ref{lem:sinL1Linf}.
We set
\[ F_1(\xi,\eta):=\int_0^\infty \chi(\xi)\chi(\eta)\chi(x\xi)\chi(x\eta)\phi(x,\xi^2)\phi(x,\eta^2)U(x)dx. \]
By Lemma \ref{lem:phi} we have $\chi(\xi)\chi(x\xi)\phi(x,\xi^2)=O(\langle x \rangle^{\frac12+}\xi^0)$
where the $O$-term behaves like a symbol.
Thus, since $U(x)=O(\langle x\rangle^{-2-\alpha})$, we infer
\[ F_1(\xi,\eta)=\int_0^\infty O(\langle x\rangle^{-1-}\xi^0 \eta^0)dx \]
with an $O$-term of symbol type.
This yields $|F_1(\xi,\eta)|\lesssim 1$ and
\[ |\partial_\xi F_1(\xi,\eta)|+|\partial_\eta F_1(\xi,\eta)|\lesssim \xi^{-1}+\eta^{-1}. \]

Next, we consider
\[ F_2(\xi,\eta):=\int_0^\infty \chi(\xi)\chi(\eta)[1-\chi(x\xi)]\chi(x\eta)\phi(x,\xi^2)\phi(x,\eta^2)U(x)dx. \]
By Lemma \ref{lem:phi}, $F_2$ is composed of terms of the form
\[ I_\pm(\xi,\eta)=\int_0^\infty e^{\pm ix\xi}[1-\chi(x\xi)]\chi(x\eta)O_\C(\langle x\rangle^{-1-}\xi^0 \eta^0)dx \]
where the $O_\C$-term behaves like a symbol.
Consequently, we infer 
\[ |I_\pm(\xi,\eta)|\lesssim 1,\quad |\partial_\eta I_\pm(\xi,\eta)|\lesssim \eta^{-1} \] as well
as $|\partial_\xi I_\pm(\xi,\eta)|\lesssim \xi^{-1}+\eta^{-1}$.
By symmetry, we also obtain the desired bounds for $F_3(\xi,\eta):=F_2(\eta,\xi)$.

We continue with
\[ F_4(\xi,\eta):=\int_0^\infty \chi(\xi)\chi(\eta)[1-\chi(x\xi)][1-\chi(x\eta)]\phi(x,\xi^2)\phi(x,\eta^2)U(x)dx. \]
According to Lemma \ref{lem:phi}, $F_4$ is composed of $I_\pm$ and $\overline{I_\pm}$ where
\[ I_\pm(\xi,\eta):=\int_0^\infty e^{ix(\xi\pm\eta)}[1-\chi(x\xi)][1-\chi(x\eta)]O_\C(\langle x\rangle^{-2-}\xi^{-\frac12}
 \eta^{-\frac12})dx \]
with an $O_\C$-term of symbol type.
This yields $|I_\pm(\xi,\eta)|\lesssim 1$ and
\[ |\partial_\xi I_\pm(\xi,\eta)|+|\partial_\eta I_\pm(\xi,\eta)|\lesssim \xi^{-1}+\eta^{-1} \]
as desired.

Next, we study
\[ F_5(\xi,\eta):=\int_0^\infty [1-\chi(\xi)]\chi(\eta)\chi(x\eta)\phi(x,\xi^2)\phi(x,\eta^2)U(x)dx. \]
By Lemma \ref{lem:phi} this reduces to estimate
\[ I_\pm(\xi,\eta):=\int_0^\infty e^{\pm ix\xi}\chi(\eta)\chi(x\eta)O_\C(\langle x\rangle^{-\frac32-}\xi^0\eta^0)dx \]
with an $O_\C$-term of symbol type.
We obtain $|I_\pm(\xi,\eta)|\lesssim 1$ and
\[ |\partial_\xi I_\pm(\xi,\eta)|\lesssim \xi^{-1}+\eta^{-1},\quad |\partial_\eta I_\pm(\xi,\eta)|\lesssim \eta^{-1}
\lesssim \xi^{-1}+\eta^{-1}. \]
By symmetry, we have the same bounds for $F_6(\xi,\eta):=F_5(\eta,\xi)$.

By Lemma \ref{lem:phi}, the case
\[ F_7(\xi,\eta):= \int_0^\infty [1-\chi(\xi)]\chi(\eta)[1-\chi(x\eta)]
\phi(x,\xi^2)\phi(x,\eta^2)U(x)dx \]
reduces to the study of
\[ I_\pm(\xi,\eta):=\int_0^\infty e^{ix(\xi\pm \eta)}\chi(\eta)[1-\chi(x\eta)]O_\C(\langle x\rangle^{-2-}\xi^0\eta^{-\frac12})dx \]
with an $O_\C$-term of symbol type.
Consequently, we infer $|I_\pm(\xi,\eta)|\lesssim 1$ and
\[ |\partial_\xi I_\pm(\xi,\eta)|\lesssim \xi^{-1}+\eta^{-1},\quad
|\partial_\eta I_\pm(\xi,\eta)|\lesssim \eta^{-1}\lesssim \xi^{-1}+\eta^{-1} \]
and the same bounds are true for $F_8(\xi,\eta):=F_7(\eta,\xi)$.

Finally, it remains to estimate
\[ F_9(\xi,\eta):=\int_0^\infty [1-\chi(\xi)][1-\chi(\eta)]
\phi(x,\xi^2)\phi(x,\eta^2)U(x)dx \]
which reduces to
\[ I_\pm(\xi,\eta):=\int_0^\infty e^{ix(\xi\pm \eta)}O_\C(\langle x\rangle^{-2-}\xi^0\eta^0)dx \]
where the $O_\C$-term behaves like a symbol.
We infer
\[ |I_\pm(\xi,\eta)|\lesssim 1,\quad |\partial_\xi I_\pm(\xi,\eta)|\lesssim 1, \quad 
|\partial_\eta I_\pm(\xi,\eta)|\lesssim 1 \]
and, since $F=\sum_{k=1}^9 F_k$, this finishes the proof.
\end{proof}

\subsection{Representation as a singular integral operator}
Our next goal is to show that the ``off-diagonal part'' of $B$ can be realized
as a singular integral operator which is bounded on $L^2_{\tilde \rho}(\R_+)$.
We will make use of the following result.

\begin{lemma}
\label{lem:L2bound}
Suppose $T: \mc D(T)\subset L^2(\R_+)\to L^2(\R_+)$ is given by
\[ \mc D(T):=C_c^\infty(0,\infty),\quad Tf(x):=\int_0^\infty K(x,y)f(y)dy \]
where the kernel $K \in L^1_{\mathrm{loc}}((0,\infty)\times (0,\infty))$ satisfies the pointwise bound
\[ |K(x,y)|\lesssim \min\{x^{-1+\delta}y^{-\delta},x^{-\delta}y^{-1+\delta}\} \]
for all $x,y>0$ and some fixed $\delta \in [0,\frac12)$.
Then $T$ extends to a bounded operator on $L^2(\R_+)$.
\end{lemma}

\begin{proof}
A naive application of Cauchy-Schwarz leads to logarithmic divergencies and thus, we need
to introduce a dyadic decomposition.
We set $I_j:=[2^{j-1},2^{j+1}]$, $j\in \Z$, and write $1_j$ for the characteristic function
of $I_j$.
Furthermore, we abbreviate $L^2:=L^2(\R_+)$.
Now observe that $T f=\frac12 \sum_{k\in\Z}T(1_k f)$ and thus,
\begin{align*}
\|T f\|_{L^2}^2 &=\tfrac12 \sum_{j\in\Z}\|1_j T f\|_{L^2}^2
=\tfrac18 \sum_{j\in\Z}\left \|\sum_{k\in\Z}1_j T(1_k f) \right \|_{L^2}^2 \\
&\leq \sum_{j\in\Z} \left (\sum_{k\in \Z} \|1_j T(1_k f)\|_{L^2}\right )^2.
\end{align*}
This yields
\[ \|T f\|_{L^2}\leq \left ( \sum_{j\in \Z} \left |\sum_{k\in\Z}
\|T_{jk}\|_{L^2}\|1_k f\|_{L^2} \right |^2 \right )^{1/2} \]
where $T_{jk}$ has the kernel $1_j(x)K(x,y)1_k(y)$.
If $j\geq k$ we use $|K(x,y)|\lesssim x^{-1+\delta}y^{-\delta}$
to conclude
\begin{align*} 
\|T_{jk}\|_{L^2}^2&\leq \int_0^\infty \int_0^\infty 1_j(x)|K(x,y)|^2 1_k(y)dy dx \\
&\lesssim 2^{j+k}2^{2(-1+\delta)j}2^{-2\delta k}=2^{-(1-2\delta)(j-k)}.
\end{align*}
If $j<k$ we use $|K(x,y)|\lesssim x^{-\delta}y^{-1+\delta}$ and infer $\|T_{jk}\|_{L^2}^2\lesssim 2^{-(1-2\delta)(k-j)}$.
Consequently, we obtain
\[ \|T f\|_{L^2}\lesssim \left ( \sum_{j\in \Z} \left |\sum_{k\in\Z}
2^{-(\frac12-\delta)|j-k|}\|1_k f\|_{L^2} \right |^2 \right )^{1/2}. \]
The expression on the right-hand side is the $\ell^2(\Z)$-norm of the convolution
of $(2^{-(\frac12-\delta)|k|})_k$ with $(\|1_k f\|_{L^2})_k$ and by assumption on $\delta$,
$(2^{-(\frac12-\delta) |k|})_k$
belongs to $\ell^1(\Z)$. Consequently, Young's inequality yields
\[ \|T f\|_{L^2}\lesssim \left (\sum_{k\in \Z} \|1_k f\|_{L^2}^2 \right )^{1/2}=\sqrt 2 \|f\|_{L^2}. \]
\end{proof}

\begin{proposition}
\label{prop:B0}
The singular integral operator
\[ B_0 f(\xi):=2\int_0^\infty \frac{F(\xi,\eta)}{\xi^2-\eta^2}\eta \rho(\eta^2)f(\eta)d\eta \]
exists for $f\in C_c^\infty(0,\infty)$ in the principal value sense and extends to a bounded
operator on $L^2_{\tilde \rho}(\R_+)$.
\end{proposition}

\begin{proof}
We attach the weight $\tilde \rho$ to the kernel
and distinguish between diagonal and off-diagonal behavior. To this end we set\footnote{$B_d f$ is interpreted in the principal value sense.}
\begin{align*} B_d f(\xi)&:=\int_0^\infty \chi\left (\frac{4(\xi-\eta)^2}{(\xi+\eta)^2}\right )
\tilde \rho(\xi)^\frac12 \frac{F(\xi,\eta)}{\xi^2-\eta^2}\tilde \rho(\eta)^\frac12 f(\eta)d\eta  \\
B_{nd}f(\xi)&:=\int_0^\infty \left [1- \chi\left (\frac{4(\xi-\eta)^2}{(\xi+\eta)^2}\right ) \right ]
\tilde \rho(\xi)^\frac12 \frac{F(\xi,\eta)}{\xi^2-\eta^2}\tilde \rho(\eta)^\frac12 f(\eta)d\eta
\end{align*}
and our goal is to prove that $B_d$ and $B_{nd}$, initially defined on $C_c^\infty(0,\infty)$, extend to bounded operators on $L^2(\R_+)$.
We start with the off-diagonal part.
On the support of the cut-off in the definition of $B_{nd}$ we have $|\xi-\eta|\gtrsim \xi+\eta$ and thus,
from Lemma \ref{lem:boundsF} we infer
\[ \left [1- \chi\left (\frac{4(\xi-\eta)^2}{(\xi+\eta)^2}\right ) \right ]
\left |\tilde \rho(\xi)^\frac12 \frac{F(\xi,\eta)}{\xi^2-\eta^2}\tilde \rho(\eta)^\frac12 \right | 
\lesssim \frac{\xi^\frac12 \eta^\frac12}{\xi^2+\eta^2}. \]
If $\xi\geq \eta$ we have
\[ \frac{\xi^\frac12 \eta^\frac12}{\xi^2+\eta^2}
\leq \xi^{-\frac32}\eta^\frac12\leq \xi^{-1}\leq \min\{\xi^{-1},\eta^{-1}\} \]
and this estimate is symmetric in $\xi$ and $\eta$.
Consequently, Lemma \ref{lem:L2bound} yields $\|B_{nd}f\|_{L^2(\R_+)}\lesssim \|f\|_{L^2(\R_+)}$.

Thus, it remains to study the diagonal part.
To this end we employ a dyadic covering of the diagonal given by $I_j \times I_j$
with $I_j=[2^{j-1},2^{j+1}]$, $j\in \Z$, and we write $1_j$ for the characteristic function of $I_j$.
We set 
\[ G(\xi,\eta):=\chi\left (\frac{4(\xi-\eta)^2}{(\xi+\eta)^2}\right )
\frac{\tilde \rho(\xi)^\frac12 F(\xi,\eta)\tilde \rho(\eta)^\frac12}{\xi+\eta}. \]
With this notation the kernel of $B_d$ equals $\frac{G(\xi,\eta)}{\xi-\eta}$.
Note that $G$ is supported in $\bigcup_{j\in\Z}I_j\times I_j$.
Consequently, we have
\begin{align*}
\|B_d f\|_{L^2}^2&\simeq \sum_{j\in\Z}\|1_j B_d f\|_{L^2}^2\simeq \sum_{j\in\Z}\|1_j B_d(1_j f)\|_{L^2}^2 \\
&\lesssim \sum_{j\in\Z}\|B_{d,j}\|_{L^2}^2\|1_jf\|_{L^2}^2
\end{align*}
where $B_{d,j}$ has the kernel $1_j(\xi)\frac{G(\xi,\eta)}{\xi-\eta}1_j(\eta)$ and $L^2:=L^2(\R_+)$.
Thus, it suffices to bound the operator norm $\|B_{d,j}\|_{L^2}$, uniformly in $j\in\Z$.
By Lemmas \ref{lem:rho} and \ref{lem:boundsF} we have
\[ |G(\xi,\eta)|\lesssim 1,\quad |\partial_\xi G(\xi,\eta)|\lesssim \xi^{-1},\quad
|\partial_\eta G(\xi,\eta)|\lesssim \xi^{-1} \]
where we used the fact that $G(\xi,\eta)$ is supported on $\xi \simeq \eta$.
We write
\[ G(\xi,\eta)=g(\xi)+(\xi-\eta)G_1(\xi,\eta) \]
where $g(\xi):=G(\xi,\xi)$ and $G_1(\xi,\eta):=-\int_0^1 \partial_2 G(\xi,\xi+s(\eta-\xi))ds$.
We have the bounds $|g(\xi)|\lesssim 1$ and $|G_1(\xi,\eta)|\lesssim \xi^{-1}$.
Hence, we obtain the decomposition
$B_{d,j}=B_{1,j}+B_{2,j}$ where
\begin{align*} 
B_{1,j}f(\xi)&=\int_0^\infty 1_j(\xi)\frac{g(\xi)}{\xi-\eta}1_j(\eta)f(\eta)d\eta \\
B_{2,j}f(\xi)&=\int_0^\infty 1_j(\xi)G_1(\xi,\eta)1_j(\eta)f(\eta)d\eta.
\end{align*}
We have $1_j(\xi)|G_1(\xi,\eta)|1_j(\eta)\lesssim 2^{-j}$ and thus,
\[ \|B_{2,j}\|_{L^2}^2=\int_0^\infty \int_0^\infty 1_j(\xi)|G_1(\xi,\eta)|^2 1_j(\eta)d\eta d\xi \lesssim 1 \]
for all $j\in \Z$.
Furthermore, we may write $B_{1,j}f=\pi g 1_j H(1_j f)$ where $H$ is the Hilbert transform. 
Consequently, we infer
\begin{align*} 
\|B_{1,j}f\|_{L^2(\R_+)}&\lesssim \|g\|_{L^\infty(\R_+)}\|H(1_j f)\|_{L^2(\R)}\lesssim
\|1_j f\|_{L^2(\R)} \\
&\lesssim \|f\|_{L^2(\R_+)}
\end{align*}
and this shows $\|B_{1,j}\|_{L^2}\lesssim 1$ for all $j\in\Z$.
\end{proof}

\subsection{The diagonal part}
It remains to study the ``diagonal part'' of $B$.
To this end we consider the regularization
\[ B_\epsilon^d f(\xi):=2\int_0^\infty (\xi\partial_\xi-x\partial_x)\phi(x,\xi^2)
\int_0^\infty \chi\left (\tfrac{(\xi-\eta)^2}{\epsilon^2}\right )\phi(x,\eta^2)f(\eta)\eta\rho(\eta^2)d\eta dx\]
with the usual cut-off introduced in the proof of Lemma \ref{lem:sinL1Linf}.
We begin with a preliminary result which will allow us to discard certain contributions 
to $B_\epsilon^d$ from the onset.

\begin{lemma}
\label{lem:conv0}
Let $a \in C^1(\R^3)$ and assume the bounds
\[ |a(\xi,x,\eta)|\lesssim \langle x\rangle^{-1},\quad |\partial_\eta a(\xi, x,\eta)|\lesssim \langle x\rangle^{-1}
\]
for all $\xi,x,\eta \in \R$. For $\epsilon>0$ define operators $S_\epsilon^\pm$ on $C_c^\infty(\R)$ by
\[ S_\epsilon^\pm f(\xi):=\int_\R e^{ix\xi}\int_\R e^{\pm ix\eta}
\chi\left (\tfrac{(\xi-\eta)^2}{\epsilon^2}\right )a(\xi,x,\eta)f(\eta)d\eta dx. \]
Then, for any $\xi\in \R$, we have
\[ \lim_{\epsilon\to 0+}S_\epsilon^\pm f(\xi)=0. \]
\end{lemma}

\begin{proof}
Let $f\in C_c^\infty(\R)$.
First of all, integration by parts with respect to $\eta$ yields decay in $x$ and this shows
that for any $\epsilon>0$ and $\xi\in \R$, $S^\pm_\epsilon f(\xi)$ is well-defined.
Furthermore, a change of variables yields
\[ S^\pm_\epsilon f(\xi)=\int_\R e^{i\frac{x}{\epsilon}(\xi\pm \xi)}
A_\epsilon(\xi,\tfrac{x}{\epsilon}, \eta)dx \]
where
\[ A_\epsilon(\xi,y,\eta):=\int_\R e^{\pm i\epsilon y\eta}
\chi(\eta^2)a(\xi, y, \xi+\epsilon\eta)f(\xi+\epsilon \eta)d\eta. \]
Clearly, we have $|A_\epsilon(\xi,\frac{x}{\epsilon},\eta)|\lesssim \langle \tfrac{x}{\epsilon}\rangle^{-1}$
and an integration by parts yields
\[ |A_\epsilon(\xi,\tfrac{x}{\epsilon},\eta)|\lesssim \langle \tfrac{x}{\epsilon} \rangle^{-1}|x|^{-1}. \]
By interpolation we infer $|A_\epsilon(\xi,\tfrac{x}{\epsilon},\eta)|\lesssim \langle \tfrac{x}{\epsilon}\rangle^{-1}|x|^{-\frac12}$
and thus,
\[ |S^\pm_\epsilon f(\xi)|\lesssim \int_{|x|\leq \epsilon}dx+\epsilon \int_{|x|>\epsilon}|x|^{-\frac32}dx
\lesssim \epsilon^\frac12. \]
\end{proof}

Other types of operators we will encounter are handled by the following lemma.

\begin{lemma}
\label{lem:delta}
For $\epsilon>0$ define operators $T_\epsilon^\pm$ on $C_c^\infty(\R)$ by
\[ T_\epsilon^\pm f(\xi):=\int_\R e^{ix\xi} \int_\R e^{\pm ix\eta}
\chi\left (\tfrac{(\xi-\eta)^2}{\epsilon^2}\right )f(\eta)d\eta dx. \]
Then we have
\begin{align*} 
\lim_{\epsilon \to 0+}B_\epsilon^- f(\xi)&=2\pi f(\xi) \\
\lim_{\epsilon \to 0+} B_\epsilon^+ f(\xi)&=\left \{ \begin{array}{ll}
2\pi f(0), & \xi=0 \\
0, & \xi \in \R\backslash \{0\} \end{array} \right . .
\end{align*}
\end{lemma}

\begin{proof}
We argue as in the classical proof of Fourier inversion, i.e., we throw in the convergence
factor $e^{-\delta^2 x^2}$ and exploit the fact that the Fourier transform of Gaussians
is explicit,
\begin{align*}
T_\epsilon^\pm f(\xi)&=\lim_{\delta\to 0+}\int_\R e^{ix\xi}e^{-\delta^2 x^2}\int_\R
e^{\pm i\eta x}\chi\left (\tfrac{(\xi-\eta)^2}{\epsilon^2}\right )f(\eta)d\eta dx \\
&=\lim_{\delta\to 0+}\int_\R \int_\R e^{ix(\xi\pm \eta)}e^{-\delta^2 x^2}dx 
\chi\left (\tfrac{(\xi-\eta)^2}{\epsilon^2}\right )f(\eta) d\eta \\
&=\sqrt \pi \lim_{\delta \to 0+}\int_\R \tfrac{1}{\delta}e^{-\frac{(\xi\pm \eta)^2}{4\delta^2}}
\chi\left (\tfrac{(\xi-\eta)^2}{\epsilon^2}\right )f(\eta) d\eta. 
\end{align*}
Consequently, we infer
\begin{align*} 
T_\epsilon^+ f(\xi)&=\sqrt \pi 
\lim_{\delta\to 0+} \int_\R \tfrac{1}{\delta}e^{-\frac{\eta^2}{4\delta^2}}
\chi \left (\tfrac{(2\xi-\eta)^2}{\epsilon^2} \right )f(\eta-\xi)d\eta \\
&=2\sqrt \pi \lim_{\delta \to 0+}\int_\R e^{-\eta^2}\chi \left (\tfrac{4(\xi-\delta \eta)^2}{\epsilon^2}
\right )f(2\delta \eta-\xi)d\eta \\
&=2\pi \chi \left (\tfrac{4\xi^2}{\epsilon^2}\right )f(-\xi)
\end{align*}
and this yields the claim for $T_\epsilon^+$.
The same calculation shows
\[ T_\epsilon^- f(\xi)=2\pi \chi(0)f(\xi)=2\pi f(\xi). \] 
\end{proof}

\begin{lemma}
\label{lem:Bdiag}
For any $f\in C_c^\infty(0,\infty)$,
we have
\[ \lim_{\epsilon\to 0+}B_\epsilon^d f(\xi)=h(\xi) f(\xi),\quad \xi> 0 \]
where $h\in C^\infty(0,\infty)$ satisfies $\|h\|_{L^\infty(0,\infty)}\lesssim 1$
and behaves like a symbol under differentiation.
\end{lemma}

\begin{proof}
We use the representation
\[ \phi(x,\xi^2)=a(\xi)f_+(x,\xi)+\overline{a(\xi)f_+(x,\xi)} \]
where $a(\xi)=-\frac12 i \xi^{-1}\overline{W(f_+(\cdot,\xi), \phi(\cdot,\xi^2))}$.
Furthermore, since $\xi>0$ is fixed and $f\in C^\infty_c(0,\infty)$ is supported
away from the origin, only the large frequency asymptotics  
\[ f_+(x,\xi)=e^{ix\xi}[1+O_\C(\langle x\rangle^{-1}\xi^{-1})] \] from Lemma \ref{lem:Jostlg}
are relevant.
Consequently, by Lemmas \ref{lem:conv0} and \ref{lem:delta} the only nonzero contribution (up to complex conjugates)
is given by
\[ B^d_{\epsilon,1} f(\xi)=\tfrac12 \xi a'(\xi)\int_\R e^{ix\xi}\int_\R \overline{a(\eta)}e^{-ix\eta}
 \chi\left (\tfrac{(\xi-\eta)^2}{\epsilon^2}\right )f(\eta)|\eta|\rho(\eta^2)d\eta dx \]
 where we have extended $a$, $f$, and $\rho$ from $[0,\infty)$ to $\R$ as even functions.
 Lemma \ref{lem:delta} yields
 \[ \lim_{\epsilon\to 0+}B^d_{\epsilon,1}f(\xi)=\pi \xi a'(\xi)\overline{a(\xi)}\xi \rho(\xi^2)f(\xi),\quad \xi>0 \]
 and from Lemmas \ref{lem:rho} and \ref{lem:phi} we have the bound
 \[ |\xi a'(\xi)\overline{a(\xi)}\xi \rho(\xi^2)|\lesssim 1 \]
 for all $\xi>0$.
\end{proof}

As a corollary we finally obtain the desired boundedness of $B$.

\begin{corollary}
\label{cor:B}
The operator $B$ extends to a bounded operator on $L^2_{\tilde \rho}(\R_+)$.
\end{corollary}

\begin{proof}
Let $f\in C_c^\infty(0,\infty)$.
For $\epsilon>0$ we set
\begin{align*}
B_\epsilon^{nd}f(\xi):=&2\int_0^\infty (\xi\partial_\xi-x\partial_x)\phi(x,\xi^2) \\
&\times \int_0^\infty \left [1-\chi\left (\tfrac{(\xi-\eta)^2}{\epsilon^2}\right )\right ]
\phi(x,\eta^2)f(\eta)\eta\rho(\eta^2)d\eta. 
\end{align*}
Then we have $Bf=B_\epsilon^d f+B_\epsilon^{nd}f$ for all $f\in C_c^\infty(0,\infty)$ and all $\epsilon>0$.
By Eq.~\eqref{eq:Schwartzkernel} the operator $B_\epsilon^{nd}$ has the kernel
\[ 2 \left [1-\chi\left (\tfrac{(\xi-\eta)^2}{\epsilon^2}\right )\right ]\frac{F(\xi,\eta)}{\xi^2-\eta^2}\eta \rho(\eta^2) \]
and Proposition \ref{prop:B0} shows that
\[ \lim_{\epsilon\to 0+}B_\epsilon^{nd}f(\xi)=B_0 f(\xi) \]
for any $f\in C^\infty_c(0,\infty)$ and $B_0$ is bounded on $L^2_{\tilde \rho}(\R_+)$.
Consequently, the claim follows from Lemma \ref{lem:Bdiag}.
\end{proof}

\subsection{Boundedness on weighted spaces}

We also need $B$ to be bounded on weighted $L^2_{\tilde \rho}$ spaces.
As a matter of fact, the diagonal part is not at all affected by the introduction
of a weight.
However, for the off-diagonal part we need more refined estimates for the function $F$.

\begin{lemma}
\label{lem:Fref}
Suppose $V^{(2j+1)}(0)=0$ for all $j\in \N_0$.
Then the function $F$ from Lemma \ref{lem:F} satisfies the bounds
\[ |F(\xi,\eta)|\leq \frac{C_k}{\xi^k+\eta^k} \]
for all $\xi,\eta >0 $ with $|\xi-\eta|\gtrsim \xi+\eta$ and all $k\in \N_0$.
\end{lemma}

\begin{proof}
The assumption on $V$ implies that $V$ and $\phi(\cdot,\xi^2)$ extend
to smooth even functions on $\R$, for any $\xi\geq 0$.
As a consequence, the function $U(x)=-2V(x)-xV'(x)$ from the definition of $F$ also extends
to a smooth even function on $\R$ and we obtain
\[ F(\xi,\eta)=\tfrac12 \int_\R \phi(x,\xi^2)\phi(x,\eta^2)U(x)dx. \]

If $\xi,\eta \leq 1$ the stated bound follows from Lemma \ref{lem:boundsF}.
Thus, it suffices to consider the cases $\xi\geq 1\geq \eta$ and $\xi,\eta\geq 1$.
We start with the former and set
\[ F_1(\xi,\eta):=\int_\R [1-\chi(\xi)]\chi(\eta)\chi(|x\eta|)\phi(x,\xi^2)\phi(x,\eta^2)U(x)dx. \]
Lemma \ref{lem:phi} implies $|\partial_x \phi(x,\eta^2)|\lesssim \langle \eta\rangle$
for all $x,\eta\geq 0$. Consequently, the equation
\[ \partial_x^2 \phi(x,\eta^2)=V(x)\phi(x,\eta^2)-\eta^2 \phi(x,\eta^2) \]
and an induction yields the bounds 
\[ |\partial_x^k \phi(x,\eta^2)|\leq C_k \langle x\rangle^{\frac12+} \]
for all $x\in \R$ and $0<\eta\leq 1$.
By the large frequency asymptotics of $\phi$ in Lemma \ref{lem:phi} we see that
$F_1$ is composed of $I$ and $\overline{I}$ where
\[ I(\xi,\eta)=\int_\R e^{ix\xi}O_\C(\langle x\rangle^0 \xi^0)[1-\chi(\xi)]\chi(\eta)\chi(|x\eta|)\phi(x,\eta^2)
U(x)dx \]
and the $O_\C$-term behaves like a symbol.
Consequently, repeated integration by parts yields the bounds
\[ |I(\xi,\eta)|\leq \frac{C_k}{\xi^k}\leq \frac{C_k}{\xi^k+1}\leq \frac{C_k}{\xi^k+\eta^k} \]
for all $k\in \N_0$.

Next, we consider the case $\xi\geq 1\geq \eta$ and $|x\eta| \geq 1$, i.e., we set
\[ F_2(\xi,\eta):=\int_\R [1-\chi(\xi)]\chi(\eta)[1-\chi(|x\eta|)]\phi(x,\xi^2)\phi(x,\eta^2)U(x)dx. \]
By Lemma \ref{lem:phi} it follows that $F_2$ is composed of $I_\pm$ and $\overline{I_\pm}$ where
\[ I_\pm(\xi,\eta)=\int_\R e^{ix(\xi\pm \eta)}O_\C(\langle x\rangle^{\frac12+}\xi^0 \eta^0)U(x)dx \]
with an $O_\C$-term of symbol type. Consequently, for $|\xi-\eta|\gtrsim \xi+\eta$ we infer
by repeated integration by parts the bounds
\[ |I_\pm (\xi,\eta)|\leq \frac{C_k}{(\xi+\eta)^k}\leq \frac{C_k}{\xi^k+\eta^k} \]
for all $k\in \N_0$.

Finally, we consider
\[ F_3(\xi,\eta):=\int_\R [1-\chi(\xi)][1-\chi(\eta)]\phi(x,\xi^2)\phi(x,\eta^2)U(x)dx. \]
By the large frequency asymptotics of $\phi$ from Lemma \ref{lem:phi} it follows that
$F_3$ is composed of $I_\pm$ and $\overline{I_\pm}$ where
\[ I_\pm(\xi,\eta):=\int_\R e^{ix(\xi\pm \eta)}O_\C(\langle x\rangle^{-2}\xi^0\eta^0)dx \]
and the $O_\C$-term behaves like a symbol.
Thus, repeated integration by parts in conjunction with the assumption $|\xi-\eta|\gtrsim \xi+\eta$
yields the bound
\[ |I_\pm (\xi,\eta)|\leq \frac{C_k}{(\xi+\eta)^k}\leq 
\frac{C_k}{\xi^k+\eta^k}. \]
\end{proof}

\begin{lemma}
\label{lem:Bweighted}
Let $s\in \R$ and assume $V^{(2j+1)}(0)=0$ for all $j\in \N_0$. Then we have the bound
\[ \|\langle \cdot \rangle^s B f\|_{L^2_{\tilde \rho}(\R_+)}\lesssim \|\langle \cdot \rangle^s f\|_{L^2_{\tilde \rho}(\R_+)} \]
for all $f\in C_c^\infty(0,\infty)$.
\end{lemma}

\begin{proof}
It suffices to prove $L^2(\R_+)$-boundedness of the operator
\[ \tilde B_{nd}f(\xi):=\int_0^\infty \tilde K(\xi,\eta) f(\eta)d\eta \]
where
\[ \tilde K(\xi,\eta):=\left [1- \chi\left (\frac{4(\xi-\eta)^2}{(\xi+\eta)^2}\right ) \right ]
\langle \xi\rangle^s \tilde \rho(\xi)^\frac12 \frac{F(\xi,\eta)}{\xi^2-\eta^2}\tilde \rho(\eta)^\frac12
\langle \eta\rangle^{-s}, \]
cf.~the proof of Proposition \ref{prop:B0}.
In the case $\xi,\eta\leq 1$ we have $|\tilde K(\xi,\eta)|\lesssim \min\{\xi^{-1},\eta^{-1}\}$, as
in the proof of Proposition \ref{prop:B0}.
If $0<\eta\leq 1\leq \xi$ we use the bound from Lemma \ref{lem:Fref} to conclude
\[ |\tilde K(\xi,\eta)|\lesssim \frac{\xi^s}{\xi^k+\eta^k}\lesssim \xi^{-1}\lesssim \eta^{-1} \]
by taking $k$ large enough and we obtain $|\tilde K(\xi,\eta)|\lesssim \min\{\xi^{-1},\eta^{-1}\}$.
The same bound is true in the case $0<\xi\leq 1\leq \eta$.
Similarly, if $\xi,\eta\geq 1$ we infer
\[ |\tilde K(\xi,\eta)|\lesssim \frac{\xi^s \eta^{-s}}{\xi^k+\eta^k}\lesssim
\min\{\xi^{-1},\eta^{-1}\} \]
by taking $k$ large enough.
Consequently, Lemma \ref{lem:L2bound} yields the claim.
\end{proof}

The corresponding result on the physical side reads as follows.

\begin{corollary}
\label{cor:E}
Let $f\in \mc S(\R)$ be even, $s\geq 0$, and assume $V^{(2j+1)}(0)=0$ for all $j\in \N_0$.
Then the operator $E:=\mc F^{-1}B\mc F$ satisfies the bound
\[ \|E f\|_{H^s(\R_+)}\lesssim \|f\|_{H^s(\R_+)}. \]
\end{corollary}

\begin{proof}
By Lemmas \ref{lem:derweis}, \ref{lem:Bweighted}, and \ref{lem:mapF-1} we infer
\[ \|f\|_{H^s(\R_+)}\simeq \|\langle \cdot \rangle^s \mc F f\|_{L^2_{\tilde \rho}(\R_+)}
\gtrsim \|\langle \cdot \rangle^s B\mc F f\|_{L^2_{\tilde \rho}(\R_+)}\simeq \|\mc F^{-1}B\mc F f\|_{H^s(\R_+)}. \]
\end{proof}

\subsection{Basic vector field bounds}
Now we are ready to prove a first estimate involving the scaling vector field $S=t\partial_t+x\partial_x$.
We start with a basic commutator result.

\begin{lemma}
\label{lem:S}
We have the identities
\begin{align*}
(t\partial_t+x\partial_x)\left [\cos(t\sqrt A)f(x)\right ]&=\cos(t\sqrt A)\left (|\cdot|f'\right )(x) \\
&\quad +\left [E, \cos(t\sqrt A) \right ]f(x) \\
(t\partial_t+x\partial_x)\left [\frac{\sin(t\sqrt A)}{\sqrt A}g(x)\right ]&=
\frac{\sin(t\sqrt A)}{\sqrt A}\left (|\cdot|g' + g\right )(x) \\
&\quad +\left [E, \frac{\sin(t\sqrt A)}{\sqrt A} \right ]g(x) 
\end{align*}
where $E:=\mc F^{-1}B\mc F$.
\end{lemma}

\begin{proof}
We have
\[ \mc F\left (S\cos(t\sqrt A)f \right )(\xi)=(t\partial_t-\xi\partial_\xi-1)\left [\cos(t\xi)\mc F f(\xi) \right ]
+B\mc F \cos(t\sqrt A)f(\xi) \]
and
\begin{align*} (t\partial_t-\xi\partial_\xi-1)\left [\cos(t\xi)\mc F f(\xi) \right ]&=-\cos(t\xi)\left [\xi(\mc F f)'(\xi)
+\mc F f(\xi) \right ] \\
&=\cos(t\xi)\left [\mc F(|\cdot|f')(\xi)-B\mc F f(\xi)\right ]
\end{align*}
which yields the stated expression for the cosine evolution.
For the sine evolution it suffices to note that
\[ (t\partial_t-\xi\partial_\xi-1)\left (\frac{\sin(t\xi)}{\xi}\mc F g(\xi)\right )
=-\frac{\sin(t\xi)}{\xi}\xi (\mc F g)'(\xi). \]
\end{proof}

We need one more commutator estimate.
Note carefully the smoothing effect at small frequencies which requires the nonresonant
condition $a_1\not= 0$.
The point is that only if $a_1\not=0$ we have 
\[ \int_0^1 \xi^{-2}\tilde \rho(\xi)d\xi \lesssim \int_0^1 \frac{1}{\xi(1+|a_1| \log^2 \xi)}d\xi \lesssim 1, \]
see Lemma \ref{lem:rho}. 

\begin{lemma}
\label{lem:commAB}
Let $f\in \mc S(\R)$ be even, $s\geq 0$, and assume $V^{(2j+1)}(0)=0$ for all $j\in \N_0$. 
Furthermore, assume that the constant $a_1$ from Lemma \ref{lem:en0} is nonzero.
Then the commutator $[\sqrt A,E]$ satisfies the estimate
\[ \left \|[\sqrt A,E]f \right \|_{H^s(\R_+)}\lesssim \|\sqrt A f\|_{H^{s-1}(\R_+)}. \]
\end{lemma}

\begin{proof}
We have 
\[ \mc F[\sqrt A ,E]\mc F^{-1}\hat f(\xi)=\xi B \hat f(\xi)-B(|\cdot|\hat f)(\xi). \]
Consequently, $\mc F[\sqrt A,B]\mc F^{-1}$ defines a continuous map from $\mc D(J)$ to $\mc D'(J)$
and the Schwartz kernel theorem implies that the kernel $\tilde K$ of $\mc F[\sqrt A,B]\mc F^{-1}$ 
is given by $\tilde K=(\pi_1-\pi_2)K$ where $K$ is the Schwartz kernel of $B$, cf.~Lemma \ref{lem:Schwartzod}.
By Eq.~\eqref{eq:Schwartzkernel} we see that the kernel $\tilde K$ is a function and given by
\[ \tilde K(\xi,\eta)=2\frac{F(\xi,\eta)}{\xi+\eta}\eta\rho(\eta^2) \]
for $\xi,\eta>0$. By Lemmas \ref{lem:boundsF} and \ref{lem:Fref} we infer the bound
\[ \left |\langle \xi\rangle^s \frac{\tilde \rho(\xi)^\frac12
F(\xi,\eta) \tilde \rho(\eta)^\frac12}{\xi+\eta} \eta^{-1}\langle \eta\rangle^{-s+1} \right |
\lesssim \xi^{-1}\tilde \rho(\xi)^\frac12 
\eta^{-1}\tilde \rho(\eta)^\frac12 \]
for all $\xi,\eta>0$ and by the assumption $a_1\not= 0$, 
we see that this kernel induces an operator which is bounded on $L^2(\R_+)$.
Consequently, we infer 
\begin{align*} 
\|\sqrt A f\|_{H^{s-1}(\R_+)}&\simeq \|\langle \cdot \rangle^{s-1}|\cdot| \mc F f\|_{L^2_{\tilde \rho}(\R_+)}
\gtrsim \|\langle \cdot \rangle^s \mc F [\sqrt A,E]f\|_{L^2_{\tilde \rho}(\R_+)} \\
&\simeq \|[\sqrt A,E]f\|_{H^s(\R_+)}
\end{align*}
by applying Lemmas \ref{lem:derweis} and \ref{lem:mapF-1}.
\end{proof}

The above considerations lead to the following estimate for the solution of the wave equation.

\begin{lemma}
\label{lem:vft}
Let $f,g \in \mc S(\R)$ be even and assume $V^{(2j+1)}(0)=0$ for all $j\in \N_0$.
Furthermore, assume that the constant $a_1$ from Lemma \ref{lem:en0} is nonzero.
Then the solution $u$ of the initial value problem \eqref{eq:init} satisfies the bounds
\begin{align*} 
\|\partial_t^\ell \partial_t S u(t,\cdot)\|_{H^k(\R_+)}\leq C_{k,\ell} 
\Big ( &\left \|\sqrt A \left (|\cdot| f'\right ) \right \|_{H^{k+\ell}(\R_+)} +
\left \|\sqrt A f \right \|_{H^{k+\ell}(\R_+)} \\
&+\||\cdot| g'\|_{H^{k+\ell}(\R_+)}+\|g\|_{H^{k+\ell}(\R_+)} \Big )
\end{align*}
for all $t\geq 0$ and all $k,\ell\in \N_0$.
\end{lemma}

\begin{proof}
By Lemmas \ref{lem:S} and \ref{lem:difft} we obtain
\begin{align*} 
\partial_t (t\partial_t+x\partial_x)[\cos(t\sqrt A)f(x)]&=-\sin(t\sqrt A)\sqrt A\left (|\cdot|f'\right )(x) \\
&\quad -E \sin(t\sqrt A)\sqrt A f(x) \\
&\quad +\sin (t\sqrt A) \sqrt A E f(x)
\end{align*}
and
\[ \sin (t\sqrt A) \sqrt A E f(x)=\sin(t\sqrt A)E\sqrt A f(x)+\sin(t\sqrt A)[\sqrt A,E]f(x). \]
Similarly, for the sine evolution we infer
\begin{align*}
\partial_t (t\partial_t+x\partial_x)\left [\frac{\sin(t\sqrt A)}{\sqrt A}g(x) \right ]&=
\cos(t\sqrt A)\left (|\cdot|g'+g\right )(x) \\
&\quad +[E,\cos(t\sqrt A)]g(x).
\end{align*}
Consequently, Corollary \ref{cor:E} and Lemma \ref{lem:commAB} yield the stated bound
for $\ell=0$.
The case $\ell\geq 1$ follows accordingly by noting that
\[ \|\sqrt A^\ell f\|_{H^k(\R_+)}\simeq \|\langle \cdot \rangle^k |\cdot|^\ell \mc F f\|_{L^2_{\tilde \rho}(\R_+)}
\lesssim \|\langle \cdot\rangle^{k+\ell}\mc F f\|_{L^2_{\tilde \rho}(\R_+)}\simeq \|f\|_{H^{k+\ell}(\R_+)}, \]
see Lemmas \ref{lem:derweis} and \ref{lem:mapF-1}.
\end{proof}

As usual, the same bounds hold for $\sqrt A u(t,\cdot)$.

\begin{lemma}
\label{lem:vfA}
Under the assumptions of Lemma \ref{lem:vft}, 
the solution $u$ of the initial value problem \eqref{eq:init} satisfies the bounds
\begin{align*} 
\|\partial_t^\ell \sqrt AS u(t,\cdot)\|_{H^k(\R_+)}\leq C_{k,\ell} 
\Big ( &\left \|\sqrt A \left (|\cdot| f'\right ) \right \|_{H^{k+\ell}(\R_+)} +
\left \|\sqrt A f \right \|_{H^{k+\ell}(\R_+)} \\
&+\||\cdot| g'\|_{H^{k+\ell}(\R_+)}+\|g\|_{H^{k+\ell}(\R_+)} \Big )
\end{align*}
for all $t\geq 0$ and all $k,\ell\in \N_0$.
\end{lemma}

\begin{proof}
Completely analogous to the proof of Lemma \ref{lem:vft}.
\end{proof}

\subsection{Bounds involving the ordinary derivative}

In view of nonlinear applications we would like to substitute the operator
$\sqrt A$ in Lemma \ref{lem:vfA} by the simpler (but in this context unnatural)
ordinary derivative $\nabla$.
As before with the energy bounds, this is not possible within the $L^2$-framework
and we have to require suitable $L^1$-bounds for the data.
Only the sine evolution causes problems in this respect due to the singularity of $\frac{\sin(t\xi)}{\xi}$
at $\xi=0$.
As a preparation we need an $L^\infty$-bound for the commutator $[E,A^{-1/2}\sin(t A^{1/2})]$.
The following result should be compared to Remark \ref{rem:sinL1Linf}.

\begin{lemma}
\label{lem:comminf}
Let $\epsilon \in (0,1)$. Then we have the bound
\[ \left \|\langle \cdot \rangle^{-\epsilon}
\left [E,\frac{\sin(t\sqrt A)}{\sqrt A}\right ]g\right \|_{L^\infty(\R_+)} 
\leq \tfrac{C}{\epsilon} \left \|g \right \|_{L^1(\R_+)} \]
for all $t\geq 0$ and all $g\in \mc S(\R_+)$.
\end{lemma}

\begin{proof}
By definition we have $E=\mc F^{-1}B\mc F$ and we recall that
\[ x f'(x)=-\mc F^{-1}\left [|\cdot|(\mc F f)'\right ](x)-f(x)+\mc F^{-1}B \mc F f(x). \]
We obtain
\begin{align*} 
B \mc F f(\xi)&=\mc F(|\cdot|f')(\xi)+\xi (\mc F f)'(\xi)+\mc F f(\xi) \\
&=\int_0^\infty \phi(x,\xi^2)xf'(x)dx+\int_0^\infty \xi\partial_\xi \phi(x,\xi^2)f(x)dx+\mc F f(\xi) \\
&=\int_0^\infty (\xi\partial_\xi-x\partial_x)\phi(x,\xi^2)f(x)dx
\end{align*}
for all $f\in \mc S(\R_+)$.
Similarly, we infer 
\begin{align*}
\mc F^{-1}B\hat f(x)&=x\left (\mc F^{-1}\hat f\right )'(x)+\mc F^{-1}\left (|\cdot| \hat f'\right )(x)+\mc F^{-1}\hat f(x) \\
&=\int_0^\infty x\partial_x \phi(x,\xi^2)\hat f(\xi)\tilde \rho(\xi)d\xi
+\int_0^\infty \phi(x,\xi^2)\xi \hat f'(\xi)\tilde \rho(\xi)d\xi \\
&\quad +\mc F^{-1}\hat f(x) \\
&=\int_0^\infty (x\partial_x-\xi\partial_\xi)\phi(x,\xi^2)\hat f(\xi)\tilde \rho(\xi)d\xi \\
&\quad -\int_0^\infty \phi(x,\xi^2)\frac{\xi\tilde \rho'(\xi)}{\tilde \rho(\xi)}\hat f(\xi)\tilde \rho(\xi)d\xi
\end{align*}
for $\hat f\in \mc S(\R_+)$.
Consequently, we obtain the explicit representation
\begin{align*}
E\frac{\sin(t\sqrt A)}{\sqrt A}g(x)&=\mc F^{-1}B\left (\frac{\sin(t|\cdot|)}{|\cdot|}\mc F g \right )(x) \\
&=\int_0^\infty \frac{\sin(t\xi)}{\xi}\left [x\partial_x-\xi\partial_\xi
-\frac{\xi\tilde \rho'(\xi)}{\tilde \rho(\xi)}\right ]\phi(x,\xi^2)\mc F g(\xi)\tilde \rho(\xi)d\xi \\
&=2\int_0^\infty \int_0^\infty \sin(t\xi)\left [x\partial_x-\xi\partial_\xi
-\frac{\xi\tilde \rho'(\xi)}{\tilde \rho(\xi)}\right ]\phi(x,\xi^2) \\
&\quad \times \phi(y,\xi^2)g(y)dy \rho(\xi^2)d\xi.
\end{align*}
By Fubini we infer
\[ E\frac{\sin(t\sqrt A)}{\sqrt A}g(x)=\lim_{N\to \infty}\int_0^\infty K_N(x,y;t)g(y)dy \]
where
\[ K_N(x,y;t)=2\int_{1/N}^N \sin(t\xi)\left [x\partial_x-\xi\partial_\xi
-\frac{\xi\tilde \rho'(\xi)}{\tilde \rho(\xi)}\right ]\phi(x,\xi^2)\phi(y,\xi^2)\rho(\xi^2)d\xi. \]
Thus, the situation is reminiscent of Lemma \ref{lem:sinL1Linf} and by noting that
\[\left [x\partial_x-\xi\partial_\xi
-\frac{\xi\tilde \rho'(\xi)}{\tilde \rho(\xi)}\right ]\phi(x,\xi^2) \]
satisfies the same pointwise bounds as $\phi(x,\xi^2)$, one obtains the estimate
\[ \left \|\langle \cdot \rangle^{-\epsilon}E\frac{\sin(t\sqrt A)}{\sqrt A}g \right \|_{L^\infty(\R_+)}
\lesssim \tfrac{1}{\epsilon} \|g\|_{L^1(\R_+)} \]
by following the logic of the proof of Lemma \ref{lem:sinL1Linf}.
Analogously, we infer
\begin{align*}
\frac{\sin(t\sqrt A)}{\sqrt A}E g(x)&=\mc F^{-1}\left [\frac{\sin(t|\cdot|)}{|\cdot|}B\mc F g \right ](x) \\
&=\int_0^\infty \int_0^\infty \phi(x,\xi^2)\frac{\sin(t\xi)}{\xi}(\xi\partial_\xi-y\partial_y)\phi(y,\xi^2)g(y)dy
\tilde \rho(\xi)d\xi
\end{align*}
and this leads to
\[ \frac{\sin(t\sqrt A)}{\sqrt A}E g(x)=\lim_{N\to\infty}\int_0^\infty \tilde K_N(x,y;t)g(y)dy \]
with
\[ \tilde K_N(x,y;t)=2\int_{1/N}^N \sin(t\xi)\phi(x,\xi^2)(\xi\partial_\xi-y\partial_y)\phi(y,\xi^2)\rho(\xi^2)d\xi \]
which can be treated in the same fashion.
\end{proof}

Based on the commutator bound Lemma \ref{lem:comminf} we can now conclude the following.

\begin{lemma}
\label{lem:vfn}
Under the assumptions of Lemma \ref{lem:vft} we have the bounds
\begin{align*} 
\|\nabla S u(t,\cdot)\|_{H^k(\R_+)}\leq C_k \Big ( &
\||\cdot|f'\|_{H^{1+k}(\R_+)}+\|f\|_{H^{1+k}(\R_+)} \\
&\||\cdot|g'\|_{H^{k}(\R_+)}+\|g\|_{H^{k}(\R_+)} \\
&\||\cdot|g'\|_{L^1(\R_+)}+\|g\|_{L^1(\R_+)} \Big )
\end{align*}
for all $t\geq 0$ and $k\in \N_0$.
\end{lemma}

\begin{proof}
We start with the case $k=0$.
In view of Lemma \ref{lem:vfA} and
\[ \|\sqrt A Su(t,\cdot)\|_{L^2(\R_+)}^2=\|\nabla S u(t,\cdot)\|_{L^2(\R_+)}^2+\int_0^\infty V(x)|Su(t,x)|^2 dx \]
it suffices to control $\||V|^\frac12 Su(t,\cdot)\|_{L^2(\R_+)}$.
First, we consider the sine evolution.
By Lemmas \ref{lem:S}, \ref{lem:comminf}, and Remark \ref{rem:sinL1Linf} we obtain
\begin{align*}
\left \||V|^\frac12 S \frac{\sin(t\sqrt A)}{\sqrt A}g \right \|_{L^2(\R_+)}
&\leq \left \|\langle \cdot \rangle^{\epsilon}|V|^\frac12 \right \|_{L^2(\R_+)}
\left \|\langle \cdot \rangle^{-\epsilon}S\frac{\sin(t\sqrt A)}{\sqrt A}g\right \|_{L^\infty(\R_+)} \\
&\lesssim \left \|\langle \cdot \rangle^{-\epsilon}\frac{\sin(t\sqrt A)}{\sqrt A}(|\cdot|g'+g)\right \|_{L^\infty(\R_+)} \\
&\quad +\left \|\langle \cdot \rangle^{-\epsilon}\left [E,\frac{\sin(t\sqrt A)}{\sqrt A}\right ]g\right \|_{L^\infty(\R_+)} \\
&\lesssim \||\cdot|g'\|_{L^1(\R_+)}+\|g\|_{L^1(\R_+)}.
\end{align*}
For the cosine evolution we infer
\begin{align*} 
\left \||V|^\frac12 S\cos(t\sqrt A)f \right \|_{L^2(\R_+)}&\leq \left \||V|^\frac12 \right\|_{L^\infty(\R_+)}
\left \|S\cos(t\sqrt A)f\right \|_{L^2(\R_+)} \\
&\lesssim \left \|\cos(t\sqrt A)(|\cdot|f')\right \|_{L^2(\R_+)} \\
&\quad +\left \|\left [E,\cos(t\sqrt A)\right ]f\right \|_{L^2(\R_+)} \\
&\lesssim \||\cdot|f'\|_{L^2(\R_+)}+\|f\|_{L^2(\R_+)}
\end{align*}
by Lemma \ref{lem:S} and Corollary \ref{cor:E}.

For the case $k\geq 1$ we proceed as in the proof of Lemma \ref{lem:energyfree} and use the
fact that
\[ \nabla^{k}\nabla S u(t,\cdot)=-\nabla^{k-1}ASu(t,\cdot)+\nabla^{k-1}[V Su(t,\cdot)]. \]
By induction it suffices to control $\nabla^{k-1}ASu(t,\cdot)$.
Applying Lemma \ref{lem:S} we see that we have to bound
\begin{align*}
&\nabla^{k-1}A\cos(t\sqrt A)(|\cdot|f'), & & \nabla^{k-1}A[E,\cos(t\sqrt A)]f, \\
&\nabla^{k-1}\sqrt A\sin(t\sqrt A)(|\cdot|g'+g), & & \nabla^{k-1}A\left [E,\frac{\sin(t\sqrt A)}{\sqrt A}\right ]g
\end{align*}
in $L^2(\R_+)$.
We have
\begin{align*} 
\left \| A\cos(t\sqrt A)(|\cdot|f') \right \|_{H^{k-1}(\R_+)}&\simeq 
\left \|\langle \cdot \rangle^{k-1}|\cdot|^2 \cos(t|\cdot|)\mc F(|\cdot|f') \right \|_{L^2_{\tilde \rho}(\R_+)} \\
&\lesssim \left \|\langle \cdot \rangle^{k+1}\mc F(|\cdot|f')\right \|_{L^2_{\tilde \rho}(\R_+)} \\
&\simeq \left \||\cdot|f' \right \|_{H^{1+k}(\R_+)}
\end{align*}
by Lemmas \ref{lem:derweis} and \ref{lem:mapF-1}.
Similarly, we obtain
\[ \left \|\sqrt A\sin(t\sqrt A)(|\cdot|g')\right \|_{H^{k-1}(\R_+)}\lesssim \left \||\cdot|g'\right \|_{H^k(\R_+)} \]
and, by using Corollary \ref{cor:E}, 
\begin{align*} 
\left \|A [E,\cos(t\sqrt A)]f\right \|_{H^{k-1}(\R_+)}&\lesssim 
\left \|[E,\cos(t\sqrt A)]f\right \|_{H^{1+k}(\R_+)} \\
&\lesssim \|f\|_{H^{1+k}(\R_+)}.
\end{align*}
Finally, for the commutator with the sine evolution we use
\begin{align*}
A\left [E,\frac{\sin(t\sqrt A)}{\sqrt A}\right ]g&=\sqrt A \left (\sqrt A E \frac{\sin(t\sqrt A)}{\sqrt A}g 
-\sin(t\sqrt A)Eg \right ) \\
&=\sqrt A \left ([E,\sin(t\sqrt A)]g+[\sqrt A,E]\frac{\sin(t\sqrt A)}{\sqrt A}g \right )
\end{align*}
and we have
\[ \left \|\sqrt A[E,\sin(t\sqrt A)]g\right \|_{H^{k-1}(\R_+)}\lesssim \|g\|_{H^k(\R_+)}. \]
For the last term we invoke Lemma \ref{lem:commAB} and infer
\begin{align*}
\left \|\sqrt A[\sqrt A,E]\frac{\sin(t\sqrt A)}{\sqrt A}g \right \|_{H^{k-1}(\R_+)}&\lesssim
\left \|[\sqrt A,E]\frac{\sin(t\sqrt A)}{\sqrt A}g \right \|_{H^k(\R_+)} \\
&\lesssim \left \|\sin(t\sqrt A)g\right \|_{H^{k-1}(\R_+)} \\
&\lesssim \|g\|_{H^{k-1}(\R_+)}.
\end{align*}
\end{proof}

\section{Higher order vector field bounds}

We generalize our previous results to allow for multiple applications of
the vector field $S$, i.e., we derive bounds for $S^m u$.

\subsection{More commutator estimates}
In order to proceed, we need to study the commutator of $E$ and $S$.
More precisely, we set $D f(\xi):=\xi f'(\xi)$ and consider the commutator
$[D,B]$.
The key result in this respect is the fact that $(\xi\partial_\xi+\eta\partial_\eta)^m F(\xi,\eta)$
satisfies the same bounds as $F(\xi,\eta)$.

\begin{lemma}
\label{lem:derF}
Suppose $V^{(2j+1)}(0)=0$ for all $j\in \N_0$.
Then the functions 
\[ F_m(\xi,\eta):=(\xi\partial_\xi+\eta\partial_\eta)^m F(\xi,\eta), \quad m\in \N_0, \] 
with $F$ from Lemma \ref{lem:F},
satisfy the bounds
\begin{align*}
|F_m(\xi,\eta)|&\leq C_m \\
|\partial_\xi F_m(\xi,\eta)|
+|\partial_\eta F_m(\xi,\eta)|&\leq C_m \left ( \xi^{-1}\langle \xi\rangle+\eta^{-1}\langle \eta\rangle \right )
\end{align*}
for all $\xi,\eta>0$ and all $m\in \N_0$.
In addition, we have the off-diagonal bounds
\[ |F_m(\xi,\eta)|\leq \frac{C_{m,k}}{\xi^k+\eta^k} \]
for all $\xi,\eta>0$ satisfying $|\xi-\eta|\gtrsim \xi+\eta$ and all $k,m\in \N_0$.
\end{lemma}

\begin{proof}
As in the proof of Lemma \ref{lem:Fref} we use the representation
\[ F(\xi,\eta)=\tfrac12 \int_\R \phi(x,\xi^2)\phi(x,\eta^2)U(x)dx. \]
We set
\[ F_1(\xi,\eta):=\tfrac12 \int_\R \chi(\xi)\chi(\eta)\chi(|x\xi|)\chi(|x\eta|)\phi(x,\xi^2)\phi(x,\eta^2)U(x)dx. \]
From Lemma \ref{lem:phi} we obtain the bounds
\begin{align*} 
|(\xi \partial_\xi)^m [\chi(\xi)\chi(|x\xi|)\phi(x,\xi^2)]|&\leq C_m \langle x\rangle^{\frac12+} \\
|\partial_\xi(\xi\partial_\xi)^m [\chi(\xi)\chi(|x\xi|)\phi(x,\xi^2)]|&\leq C_m \langle x\rangle^{\frac12+}\xi^{-1}
\end{align*}
and thus,
\begin{align*} 
|(\xi\partial_\xi+\eta\partial_\eta)^m F_1(\xi,\eta)|&\leq C_m \\
|\partial_\xi (\xi\partial_\xi+\eta\partial_\eta)^m F_1(\xi,\eta)|
+|\partial_\eta (\xi\partial_\xi+\eta\partial_\eta)^m F_1(\xi,\eta)|&\leq C_m \left (\xi^{-1}\langle \xi\rangle
+\eta^{-1}\langle \eta\rangle \right )
\end{align*}
since $|U(x)|\lesssim \langle x\rangle^{-2-\alpha}$.

Next, we consider
\[ F_2(\xi,\eta):=\tfrac12 \int_\R \chi(\xi)\chi(\eta)[1-\chi(|x\xi|)]\chi(|x\eta|)\phi(x,\xi^2)\phi(x,\eta^2)U(x)dx. \]
Observe that $\xi\partial_\xi e^{ix\xi}=x\partial_x e^{ix\xi}$ and thus,
\[ \xi\partial_\xi \int_\R e^{ix\xi}O_\C(\langle x\rangle^{-1-}\xi^0 \eta^0)dx
=\int_\R e^{ix\xi}O_\C(\langle x\rangle^{-1-}\xi^0\eta^0)dx \]
by integration by parts, provided the $O_\C$-term is of symbol type. 
Note that, by Lemma \ref{lem:phi}, we have 
\[ |\chi(\eta)\chi(|x\eta|)\partial_x \phi(x,\eta^2)|\lesssim \langle x\rangle^{-\frac12+} \]
and thus, the equation
\[ \partial_x^2 \phi(x,\eta^2)=V(x)\phi(x,\eta^2)-\eta^2 \phi(x,\eta^2) \]
yields inductively the bounds
\[ |\chi(\eta)\chi(|x\eta|)\partial_x^k \phi(x,\eta^2)|\leq C_k \langle x\rangle^{\frac12-k+} \]
for all $k\in \N_0$ since $|\chi(\eta)\chi(|x\eta|)\eta^2|\lesssim \langle x\rangle^{-2}$.
It follows that 
\[ \chi(\eta)\chi(|x\eta|)\phi(x,\eta^2)=\chi(\eta)\chi(|x\eta|)O(\langle x\rangle^{\frac12+}\eta^0) \]
where the $O$-term behaves like a symbol.
Consequently, we see that
$(\xi \partial_\xi+\eta \partial_\eta)^m F_2(\xi,\eta)$ is composed of terms of the form
\[ I_\pm(\xi,\eta)=\int_\R e^{\pm ix\xi}\chi(\xi)\chi(\eta)[1-\chi(|x\xi|)]
\chi(|x\eta|)O_\C(\langle x\rangle^{-1-}\xi^0 \eta^0)dx \]
and the corresponding bounds follow as in the proof of Lemma \ref{lem:boundsF}.
By symmetry, the function $F_3(\xi,\eta):=F_2(\eta,\xi)$ can be treated analogously.

We continue with
\[ F_4(\xi,\eta):=\tfrac12 \int_\R \chi(\xi)\chi(\eta)[1-\chi(|x\xi|)][1-\chi(|x\eta|)]
\phi(x,\xi^2)\phi(x,\eta^2)U(x)dx. \]
Here we use the observation that $(\xi\partial_\xi+\eta\partial_\eta)e^{ix(\xi\pm\eta)}=x\partial_x e^{ix(\xi\pm \eta)}$
and hence,
\[ (\xi\partial_\xi+\eta\partial_\eta)\int_\R e^{ix(\xi\pm \eta)}O_\C(\langle x\rangle^{-1-}\xi^0\eta^0)dx
=\int_\R e^{ix(\xi\pm \eta)}O_\C(\langle x\rangle^{-1-}\xi^0\eta^0)dx \]
provided the $O_\C$-term is of symbol type.
Consequently, Lemma \ref{lem:phi} shows that $(\xi\partial_\xi-\eta\partial_\eta)^m F_4(\xi,\eta)$ is composed
of terms $I_\pm$ and $\overline{I_\pm}$ where
\[ I_\pm(\xi,\eta)=\int_\R e^{ix(\xi\pm \eta)}[1-\chi(|x\xi|)][1-\chi(|x\eta|)]
O_\C(\langle x\rangle^{-2-}\xi^{-\frac12} \eta^{-\frac12})dx \]
where the $O_\C$-term behaves like a symbol.
Thus, the corresponding bounds follow as in the proof of Lemma \ref{lem:boundsF}.

The next contribution we study is
\[ F_5(\xi,\eta):=\tfrac12 \int_\R [1-\chi(\xi)]\chi(\eta)\chi(|x\eta|)\phi(x,\xi^2)\phi(x,\eta^2)U(x)dx. \]
As above, by using $\xi\partial_\xi e^{ix\xi}=x\partial_x e^{ix\xi}$, we infer that 
$(\xi\partial_\xi+\eta\partial_\eta)^m F_4(\xi,\eta)$ is composed of terms of the form
\[ I_\pm(\xi,\eta)=\int_\R e^{\pm ix\xi}\chi(\eta)\chi(|x\eta|)O_\C(\langle x\rangle^{-\frac32-}\xi^0 \eta^0)dx \]
with an $O_\C$-term of symbol type and the desired bounds follow from the proof of Lemma \ref{lem:boundsF}.
In addition, by repeated integrations by parts we obtain 
\[ |I(\xi,\eta)|\leq C_k \xi^{-k} \leq \frac{C_k}{\xi^k+\eta^k} \]
for all $k\in \N_0$.
The same bounds can be concluded for the function $F_6(\xi,\eta):=F_5(\eta,\xi)$ by symmetry.

Next, we consider
\[ F_7(\xi,\eta):=\tfrac12 \int_\R [1-\chi(\xi)]\chi(\eta)[1-\chi(|x\eta|)]\phi(x,\xi^2)\phi(x,\eta^2)U(x)dx. \]
As before, by $(\xi\partial_\xi+\eta\partial_\eta)e^{ix(\xi\pm \eta)}=x\partial_x e^{ix(\xi\pm \eta)}$, integration
by parts, and Lemma \ref{lem:phi} we see that $(\xi \partial_\xi+\eta \partial_\eta)^m F_7(\xi,\eta)$ is composed
of $I_\pm$ and $\overline{I_\pm}$ where
\[ I_\pm(\xi,\eta)=\int_\R e^{ix(\xi\pm \eta)}\chi(\eta)[1-\chi(x\eta)]O_\C(\langle x\rangle^{-2-}\xi^0 \eta^{-\frac12})
dx \]
with an $O_\C$-term of symbol type. Consequently, the respective bounds follow as in the proofs of Lemmas
\ref{lem:boundsF} and \ref{lem:Fref}.
The same is true for $F_8(\xi,\eta):=F_7(\eta,\xi)$.

The final contribution
\[ F_9(\xi,\eta):=\tfrac12 \int_\R [1-\chi(\xi)][1-\chi(\eta)]\phi(x,\xi^2)\phi(x,\eta^2)U(x)dx \]
is handled analogously.
\end{proof}

Lemma \ref{lem:derF} allows us to obtain the following commutator estimate.

\begin{lemma}
\label{lem:DB}
Let $s\in \R$ and assume $V^{(2j+1)}(0)=0$ for all $j\in \N_0$.
Then we have the bound
\[ \left \|\langle \cdot \rangle^s [D,B]f \right \|_{L^2_{\tilde \rho}(\R_+)}
\lesssim \|\langle \cdot \rangle^s f\|_{L^2_{\tilde \rho}(\R_+)} \]
for all $f\in C^\infty_c(0,\infty)$.
\end{lemma}

\begin{proof}
For the diagonal part $B_\Delta f:=hf$ (see Lemma \ref{lem:Bdiag}) it suffices to note that
\[ [D,B_\Delta]f=D(hf)-h Df=Dh f \]
and $Dh\in C^\infty(0,\infty)$, $\|Dh\|_{L^\infty(0,\infty)}\lesssim 1$ by Lemma \ref{lem:Bdiag}.
Thus, only the off-diagonal part requires work and it suffices to consider the singular
integral operator $B_0$ with kernel $K(\xi,\eta)=\frac{F(\xi,\eta)}{\xi^2-\eta^2}\tilde \rho(\eta)$
from Proposition \ref{prop:B0}.
Let $f\in C^\infty_c(0,\infty)$.
An integration by parts and the identity 
\[ (\xi\partial_\xi+\eta\partial_\eta)\frac{1}{\xi^2-\eta^2}=-\frac{2}{\xi^2-\eta^2} \]
yield
\begin{align*}
[D,B_0]f(\xi)&=\int_0^\infty \xi\partial_\xi K(\xi,\eta)f(\eta)d\eta-\int_0^\infty K(\xi,\eta)\eta f'(\eta)d\eta \\
&=\int_0^\infty (\xi\partial_\xi+\eta\partial_\eta)K(\xi,\eta)f(\eta)d\eta+B_0 f(\xi) \\
&=\int_0^\infty \frac{(\xi\partial_\xi+\eta\partial_\eta)[F(\xi,\eta)\tilde\rho(\eta)]}{\xi^2-\eta^2}f(\eta)d\eta
-B_0 f(\xi).
\end{align*}
Since $\tilde \rho(\eta)$ behaves like a symbol, it follows that $\eta\partial_\eta \tilde \rho(\eta)$
satisfies the same bounds as $\tilde \rho(\eta)$.
Consequently, Lemma \ref{lem:derF} shows that the operator $[D,B_0]$ has the same type of kernel
as $B_0$ and the claim follows from Proposition \ref{prop:B0} and Lemma \ref{lem:Bweighted}.
\end{proof}

Lemma \ref{lem:DB} is easily generalized to cover iterated commutators $[D,[D,B]]$, 
$[D,[D,[D,B]]]$, etc.
To this end, it is convenient to introduce the notation 
\[ \ad_D^0(B):=B,\quad \ad_D^m(B):=[D,\ad_D^{m-1}(B)], \quad m\in \N. \]

\begin{lemma}
\label{lem:adD}
Under the assumptions of Lemma \ref{lem:DB} we have the bound
\[ \|\langle \cdot \rangle^s \ad_D^m(B)f\|_{L^2_{\tilde \rho}(\R_+)}\leq C_m \|\langle \cdot \rangle^s f\|_{L^2_{\tilde \rho}(\R_+)} \]
for all $f\in C_c^\infty(0,\infty)$ and all $m\in \N_0$
\end{lemma}

\begin{proof}
The claim is proved inductively by following the reasoning in the proof of Lemma \ref{lem:DB}.
\end{proof}

Finally, we also need the analogue of Lemma \ref{lem:commAB} for $\ad^m_D(B)$.

\begin{lemma}
\label{lem:commAEm}
For $m\in \N_0$ set $E_m:=\mc F^{-1}\ad_D^m(B)\mc F$.
Then, under the assumptions of Lemma \ref{lem:commAB}, we have the bound
\[ \left \|[\sqrt A,E_m] f \right \|_{H^s(\R_+)}\leq C_m \left \|\sqrt A f \right \|_{H^{s-1}(\R_+)} \]
for all $m\in \N_0$.
\end{lemma}

\begin{proof}
Since $E$ and $E_m$ have the same type of kernel, the proof of Lemma \ref{lem:commAB} can
be copied verbatim upon replacing $E$ by $E_m$.
\end{proof}

\subsection{Higher order vector field bounds}
Based on Lemmas \ref{lem:adD} and \ref{lem:commAEm} one can now follow the
reasoning in Section \ref{sec:vf} to prove generalizations
of Lemmas \ref{lem:vft} and \ref{lem:vfA} for $S^m u$.
Indeed, by repeated application of the operator $S$ to $\cos(t\sqrt A)f$ and $\frac{\sin(t\sqrt A)}{\sqrt A}g$,
we obtain expressions similar to the ones stated in Lemma \ref{lem:S} with commutators
and iterated commutators of $E_m$ and the evolution operators as error terms, e.g.
\begin{align*}
S^2 \frac{\sin(t\sqrt A)}{\sqrt A}g=&\frac{\sin(t\sqrt A)}{\sqrt A}(D^2 g+2 D g+g) \\
&+2\left [E,\frac{\sin(t\sqrt A)}{\sqrt A}\right ](Dg+g)+\left [E_1,\frac{\sin(t\sqrt A)}{\sqrt A}\right ]g \\
&+\left [E,\left [E,\frac{\sin(t\sqrt A)}{\sqrt A}\right ] \right ]g
\end{align*}
where, as before, $Dg(x)=xg'(x)$.

\begin{lemma}
\label{lem:vfhigh}
Under the assumptions of Lemma \ref{lem:vft} we have the bound
\[ \left \|\partial_t^\ell \partial_t S^m u(t,\cdot) \right \|_{H^k(\R_+)}\leq C_{k,\ell,m}
\sum_{j=0}^m \left [ \|\sqrt A D^j f\|_{H^{k+\ell}(\R_+)}+\|D^j g\|_{H^{k+\ell}(\R_+)} \right ] \]
for all $t\geq 0$ and $k,\ell,m\in \N_0$.
In addition, the same bounds hold for $\|\partial_t^\ell \sqrt A S^m u(t,\cdot)\|_{H^k(\R_+)}$.
\end{lemma}

The bounds involving the ordinary derivative $\nabla$ instead of $\sqrt A$ also carry 
over to $S^m u$.
However, we do not study this systematically but restrict ourselves to the case $m=2$.

\begin{lemma}
\label{lem:vfodhigh}
Under the assumptions of Lemma \ref{lem:vft} we have the bound
\begin{align*}
\left \|\nabla S^2 u(t,\cdot)\right \|_{H^k(\R_+)}\leq C_k \sum_{j=0}^2
\Big [ &\|D^j f\|_{H^{1+k}(\R_+)}+\|D^j g\|_{H^k(\R_+)} \\
&+\|D^j g\|_{L^1(\R_+)} \Big ]
\end{align*}
for all $t\geq 0$ and $k\in \N_0$.
\end{lemma}

\begin{proof}
Following the logic in the proof of Lemma \ref{lem:vfn}, we see by Lemma \ref{lem:S} that it suffices to establish
the bound
\[ \left \|\langle \cdot \rangle^{-1}S \left [E,\frac{\sin(t\sqrt A)}{\sqrt A}\right ]g\right \|_{L^2(\R_+)}
\lesssim \sum_{j=0}^2\|D^jg\|_{L^1(\R_+)}. \]
Recall that $E=\mc F^{-1}B\mc F$ and
\[ B\mc F f(\xi)=\int_0^\infty (\xi\partial_\xi-x\partial_x)\phi(x,\xi^2)f(x)dx. \]
Consequently, we infer the explicit expression
\begin{align*} 
\langle x\rangle^{-1}S\frac{\sin(t\sqrt A)}{\sqrt A}Eg(x)=&\langle x\rangle^{-1}(t\partial_t+x\partial_x)\int_0^\infty
\phi(x,\xi^2)\tilde \rho(\xi)\frac{\sin(t\xi)}{\xi} \\
&\int_0^\infty (\xi\partial_\xi-y\partial_y)\phi(y,\xi^2)g(y)dy d\xi 
\end{align*}
and we want to place this in $L^2(\R_+)$.
As always, one has to distinguish between small and large $\xi$ and oscillatory and nonoscillatory regimes.
In the nonoscillatory case the operator $(t\partial_t+x\partial_x)$ is harmless since $t\partial_t \sin(t\xi)=\xi\partial_\xi \sin(t\xi)$
and the stated bound follows immediately. 
If $\phi(y,\xi^2)$ oscillates, one picks up a factor $y$ from the integration by parts with respect to $\xi$ but this
weight can be transferred to $g$ (leading to $Dg$) by means of another integration by parts with respect to $y$. 
If $\phi(x,\xi^2)$ oscillates one may apply the same argument by using the identity
$(t\partial_t+x\partial_x)[e^{ix\xi}\sin(t\xi)]=\xi\partial_\xi[e^{ix\xi}\sin(t\xi)]$.
For the large frequency case we note that 
\[ (\xi\partial_\xi-y\partial_y)\phi(y,\xi^2)=O_\C(\langle \xi\rangle^{-1}) \]
and the stated bound follows by means of two integrations by parts (one with respect to $\xi$ and
one with respect to $y$).
The operator $SE\frac{\sin(t\sqrt A)}{\sqrt A}$ is handled analogously by noting that
\begin{align*} 
E\frac{\sin(t\sqrt A)}{\sqrt A}g(x)=&\int_0^\infty \left [x\partial_x-\xi\partial_\xi-\frac{\xi\tilde \rho'(\xi)}{\tilde \rho(\xi)}
\right ]\phi(x,\xi^2)\tilde \rho(\xi)\frac{\sin(t\xi)}{\xi} \\
&\times \int_0^\infty \phi(y,\xi^2)g(y)dy d\xi, 
\end{align*}
cf.~the proof of Lemma \ref{lem:comminf}.
\end{proof}

\subsection{The inhomogeneous problem}
We conclude this section by considering the inhomogeneous problem
\begin{equation}
\label{eq:inhom}
 \left \{ \begin{array}{l}\partial_t^2 u(t,\cdot)+Au(t,\cdot)=f(t,\cdot), \quad t\geq 0 \\
u(t_0,\cdot)=\partial_t u(t,\cdot)|_{t=0}=0 \end{array} \right . 
\end{equation}
where $f: [0,\infty)\times \R_+ \to \R$ is prescribed. 
By Duhamel's principle, the solution to Eq.~\eqref{eq:inhom} is given by
\begin{equation}
\label{eq:Duhamel} 
u(t,\cdot)=\int_{0}^t \frac{\sin((t-s)\sqrt A)}{\sqrt A}f(s,\cdot)ds 
\end{equation}
and the dispersive and energy bounds for the sine evolution from Sections \ref{sec:disp} and \ref{sec:energy}
can be directly applied to derive
corresponding bounds for the solution of Eq.~\eqref{eq:inhom}.
We leave this to the reader.
However, the vector field bounds require some modifications because for the inhomogeneous problem,
the right-hand sides of the corresponding estimates will depend on time.
Thus, we need to understand the action of the vector field $S$ on the Duhamel formula.

\begin{lemma}
\label{lem:SDuhamel}
Let $u$ be the solution of Eq.~\eqref{eq:inhom} given by the Duhamel formula Eq.~\eqref{eq:Duhamel}.
Then we have
\begin{align*} 
Su(t,\cdot)=&\int_{0}^t \frac{\sin((t-s)\sqrt A)}{\sqrt A}S f(s,\cdot)ds+2u(t,\cdot) \\
&+\int_{0}^t \left [E,\frac{\sin((t-s)\sqrt A)}{\sqrt A}\right ]f(s,\cdot)ds
\end{align*}
for $f \in \mc S([0,\infty)\times \R_+)$.
\end{lemma}

\begin{proof}
We have
\begin{align*} \mc F Su(t,\cdot)(\xi)=&(t\partial_t-\xi\partial_\xi-1)\int_{0}^t \frac{\sin((t-s)\xi)}{\xi}
\mc F f(s,\cdot)(\xi)ds+B\mc Fu(t,\cdot)(\xi).
\end{align*}
Since 
\[ \xi\partial_\xi \frac{\sin((t-s)\xi)}{\xi}=(t-s)\cos((t-s)\xi)-\frac{\sin((t-s)\xi)}{\xi} \]
and
\[ t\partial_t \int_{0}^t \frac{\sin((t-s)\xi)}{\xi}\mc F f(s,\cdot)(\xi)ds=
\int_{0}^t t\cos((t-s)\xi)\mc F f(s,\cdot)(\xi)ds, \]
we obtain
\begin{align*}
\mc F S u(t,\cdot)(\xi)=&\int_{0}^t s\cos((t-s)\xi)\mc F f(s,\cdot)(\xi)ds \\
&-\int_{0}^t \frac{\sin((t-s)\xi)}{\xi}\xi\partial_\xi \mc F f(s,\cdot)(\xi)ds+B\mc Fu(t,\cdot)(\xi).
\end{align*}
Consequently, an integration by parts yields
\begin{align*}
\mc F Su(t,\cdot)(\xi)=&\int_{0}^t \frac{\sin((t-s)\xi)}{\xi}(s\partial_s-\xi\partial_\xi-1)\mc F f(s,\cdot)(\xi)ds \\
&+2 \mc Fu(t,\cdot)(\xi)+B \mc F u(t,\cdot)(\xi)
\end{align*}
and the claim follows.
\end{proof}

Analogous expressions hold for $S^m u$ and we arrive at the desired estimate where we
assume, for notational convenience, that $\partial_t^\ell f(t,\cdot)|_{t=0}=0$ for all $\ell\in \N_0$.

\begin{lemma}
\label{lem:vfinhom}
Let $f \in \mc S([0,\infty)\times \R)$ with $f(t,\cdot)$ even for any $t\geq 0$ and 
$\partial_t^\ell f(t,\cdot)|_{t=0}=0$ for all $\ell\in \N_0$.
Furthermore, suppose that $V^{(2j+1)}(0)=0$ for all $j\in \N_0$ and assume that the constant $a_1$
from Lemma \ref{lem:en0} is nonzero.
Then the solution $u$ of Eq.~\eqref{eq:inhom} satisfies the bounds
\begin{align*} 
\left \|\partial_t^\ell \partial_t S^m u(t,\cdot)\right \|_{H^k(\R_+)}
\leq C_{k,\ell,m}\sum_{j=0}^m 
&\int_{0}^t 
\left \|\partial_s^\ell S^j f(s,\cdot) \right \|_{H^{k}(\R_+)}ds 
\end{align*}
for all $t\geq 0$ and $k,\ell,m \in \N_0$.
In addition, the same bounds hold for $\|\partial_t^\ell \sqrt A S^m u(t,\cdot)\|_{H^k(\R_+)}$.
\end{lemma}

As before, one may also obtain bounds on $\|\nabla S^2 u(t,\cdot)\|_{H^k(\R_+)}$ by requiring
corresponding $L^1$-bounds on the source function $f$.
To be more precise, we have the following estimate.

\begin{lemma}
\label{lem:vfinhomfree}
Under the assumptions of Lemma \ref{lem:vfinhom} we have the bounds
\begin{align*} 
\left \|\nabla S^m u(t,\cdot)\right \|_{H^k(\R_+)}
\leq C_k\sum_{j=0}^m 
&\int_{0}^t 
\left \|S^j f(s,\cdot) \right \|_{H^k(\R_+)\cap L^1(\R_+)}ds 
\end{align*}
for all $t\geq 0$, $k \in \N_0$, and $m\in \{0,1,2\}$.
\end{lemma}

\section{Local energy decay}
Next, we prove Theorem \ref{thm:locen}.
For simplicity we first consider the case $m=0$.

\subsection{Basic local energy decay}

\begin{lemma}
\label{lem:locencos}
Assume that $V$ satisfies Hypothesis $\ref{hyp:B}$.
Then we have the bound
\[ \left \|\langle x \rangle^{-\frac12-}\partial_t^\ell e^{it\sqrt A}f(x)\right \|_{L^2_t(\R_+) 
H_x^{k}(\R_+)}
\leq C_{k,\ell}\left \|f\right \|_{H^{k+\ell}(\R_+)} \]
for all $k,\ell\in \N_0$ and all $f\in \mc S(\R)$ even.
\end{lemma}

\begin{proof}
We distinguish between low and high frequencies and start with the former.
By Lemma \ref{lem:derweis} we infer
\begin{align*}
\|\chi(A)f\|_{H^k(\R_+)}&\simeq \left \|\langle \cdot\rangle^k \mc F\left (\chi(A)f\right )\right\|
_{L^2_{\tilde \rho}(\R_+)}\lesssim \|\mc F(\chi(A) f)\|_{L^2_{\tilde \rho}(\R_+)} \\
&\simeq \|\chi(A)f\|_{L^2(\R_+)}
\end{align*}
and thus,
for low frequencies it suffices to consider the case $k=\ell=0$.
We have the representation
\[ e^{it\sqrt A}f(x)=\int_0^\infty \int_0^\infty e^{it\xi}\phi(x,\xi^2)\phi(y,\xi^2)\tilde \rho(\xi)
f(y)dy d\xi \]
and we decompose the integral as
\begin{align*}
A_1(t,x)&:=\int_0^\infty \int_0^\infty \chi(\xi)\chi(x\xi)\chi(y\xi)e^{it\xi}\phi(x,\xi^2)\phi(y,\xi^2)\tilde \rho(\xi)
f(y)dy d\xi \\
A_2(t,x)&:=\int_0^\infty \int_0^\infty \chi(\xi)[1-\chi(x\xi)]\chi(y\xi)\\
&\quad \times e^{it\xi}\phi(x,\xi^2)\phi(y,\xi^2)\tilde \rho(\xi)
f(y)dy d\xi \\
A_3(t,x)&:=\int_0^\infty \int_0^\infty \chi(\xi)\chi(x\xi)[1-\chi(y\xi)] \\
&\quad \times e^{it\xi}\phi(x,\xi^2)\phi(y,\xi^2)\tilde \rho(\xi)
f(y)dy d\xi \\
A_4(t,x)&:=\int_0^\infty \int_0^\infty \chi(\xi)[1-\chi(x\xi)][1-\chi(y\xi)]\\
&\quad \times e^{it\xi}\phi(x,\xi^2)\phi(y,\xi^2)\tilde \rho(\xi)
f(y)dy d\xi.
\end{align*}
From Lemmas \ref{lem:rho}, \ref{lem:phi}, Corollary \ref{cor:phi}, 
and $\tilde \rho(\xi)=2\xi\rho(\xi^2)$ we infer
\begin{equation}
\label{eq:prooflocencosA1} A_1(t,x)=\int_0^\infty \int_0^\infty \chi(\xi)\chi(x\xi)\chi(y\xi)e^{it\xi}
O(\langle x\rangle^{\frac12}\xi)O(\langle y\rangle^{\frac12}\xi^0)
 f(y)dy d\xi 
 \end{equation}
 where all $O$-terms behave like symbols.
Cauchy-Schwarz yields
\[ \chi(\xi)\int_0^\infty \chi(y\xi)O(\langle y\rangle^{\frac12}\xi^0)f(y)dy
=O(\xi^{-1})\|f\|_{L^2(\R_+)} \]
with an $O$-term of symbol type and thus,
\begin{align*} 
A_1(t,x)&=\|f\|_{L^2(\R_+)}\int_0^\infty \chi(\xi)\chi(x\xi)e^{it\xi}
O(\langle x\rangle^{\frac12}\xi^0)d\xi  \\
&=O(t^0\langle x\rangle^{-\frac12})\|f\|_{L^2(\R_+)}.
\end{align*}
On the other hand, Lemma \ref{lem:osc} yields 
$A_1(t,x)=O(\langle t\rangle^{-1}\langle x\rangle^\frac12)\|f\|_{L^2(\R_+)}$
and by interpolation we infer $|A_1(t,x)|\lesssim \langle t\rangle^{-\frac12-}
\langle x\rangle^{0+}\|f\|_{L^2(\R_+)}$ which implies
\[ \|\langle x \rangle^{-\frac12-}A_1(t,x)\|_{L^2_t(\R_+)L^2_x(\R_+)}\lesssim \|f\|_{L^2(\R_+)}. \]

We continue with $A_2$. Here we use
\begin{align*} 
\chi(\xi)[1-\chi(x\xi)]\chi(y\xi)\phi(x,\xi^2)\phi(y,\xi^2)\tilde \rho(\xi)&=
\xi^{\frac12}\langle y\rangle^\frac12 \Re[e^{ix\xi}O_\C(x^0\xi^0)]
\end{align*}
and thus, it suffices to control
\[ \tilde A_2(t,x):=\int_0^\infty \int_0^\infty \chi(\xi)[1-\chi(x\xi)]\chi(y\xi)e^{i\xi (t\pm x)}O_\C(\langle y\rangle^{\frac12}\xi^{\frac12})
f(y)dy d\xi \]
with an $O_\C$-term of symbol type.
As before, by Cauchy-Schwarz, we infer
\begin{align*} \tilde A_2(t,x)&=\|f\|_{L^2(\R_+)}\int_0^\infty \chi(\xi)[1-\chi(x\xi)]e^{i\xi(t\pm x)}
O_\C(x^0 \xi^{-\frac12})d\xi \\
&=\|f\|_{L^2(\R_+)}\int_0^\infty \chi(\xi)[1-\chi(x\xi)]e^{i\xi(t\pm x)}
O_\C(\langle x\rangle^{0+} \xi^{-\frac12+})d\xi 
\end{align*}
and Lemma \ref{lem:osc} yields the bound 
\[ |\tilde A_2(t,x)|\lesssim \langle t\pm x\rangle^{-\frac12-}
\langle x\rangle^{0+}\|f\|_{L^2(\R_+)}. \]
Consequently, by using the fact that $\|\langle t\pm x\rangle^{-\frac12-}\langle x\rangle^{-\frac12-}
\|_{L^2_t(\R)L^2_x(\R)}\lesssim 1$ (a simple consequence of Tonelli's theorem)
we find
\[ 
\|\langle x\rangle^{-\frac12-}\tilde A_2(t,x)
\|_{L^2_t(\R_+)L^2_x(\R_+)}\lesssim \|f\|_{L^2(\R_+)} \]
as desired.

For $A_3$ it suffices to consider
\begin{align*} \tilde A_3(t,x)&=\int_0^\infty \int_0^\infty \chi(\xi)\chi(x\xi)
[1-\chi(y\xi)]e^{i \xi(t\pm y)}O_\C(\langle x\rangle^{\frac12}y^0 \xi^{\frac12})f(y)dy d\xi \\
&=\int_0^\infty \underbrace{\int_0^\infty \chi(\xi)\chi(x\xi)
[1-\chi(y\xi)]e^{i \xi(t\pm y)}O_\C(\langle x\rangle^{0+}y^0 \xi^{0+})d\xi}_{=:I(t,x,y)} f(y) dy 
\end{align*}
where we have applied Fubini to interchange the order of integration.
From Lemma \ref{lem:osc}
we obtain the bound $|I(t,x,y)|\lesssim \langle x\rangle^{0+}\langle t\pm y\rangle^{-1-}$
and thus,
\begin{align*} |\tilde A_3(t,x)|&\lesssim 
\langle x\rangle^{0+} \int_\R \langle t\pm y\rangle^{-1-} |f(y)|dy \\
&=\langle x\rangle^{0+}\int_\R \langle y\rangle^{-1-}|f(y\mp t)|dy.
\end{align*}
Consequently, Minkowski's integral inequality implies
\begin{align*} 
\|\tilde A_3(t,x)\|_{L^2_t(\R_+)}&\lesssim \langle x\rangle^{0+}\left \|\int_\R
\langle y\rangle^{-1-}|f(y\mp t)|dy \right \|_{L^2_t(\R)} \\
&\leq\langle x\rangle^{0+}\int_\R \left \|\langle y\rangle^{-1-}f(y\mp t)\right \|_{L^2_t(\R)}dy \\
&\lesssim \langle x\rangle^{0+}\|f\|_{L^2(\R)}\simeq \langle x\rangle^{0+}\|f\|_{L^2(\R_+)}
\end{align*}
and by Fubini we obtain
\begin{align*} 
\|\langle x\rangle^{-\frac12-}\tilde A_3(t,x)\|_{L^2_t(\R_+)L^2_x(\R_+)}&=
\|\langle x\rangle^{-\frac12-}\tilde A_3(t,x)\|_{L^2_x(\R_+)L^2_t(\R_+)} \\
&\lesssim \|\langle\cdot\rangle^{-\frac12-}\|_{L^2(\R_+)}\|f\|_{L^2(\R_+)}.
\end{align*}

For $A_4$ we have to control expressions of the form
\begin{align*} \tilde A_4(t,x)&=\int_0^\infty \int_0^\infty \chi(\xi)[1-\chi(x\xi)][1-\chi(y\xi)]
\\
&\quad \times e^{i\xi(t\pm x\pm y)}O_\C(x^0 y^0 \xi^0)f(y)dy d\xi \\
&=\int_0^\infty J(t,x,y) f(y)dy
\end{align*}
where
\[ J(t,x,y):=\int_0^\infty \chi(\xi)[1-\chi(x\xi)][1-\chi(y\xi)]
e^{i\xi(t\pm x\pm y)}O_\C(x^{0+} y^0 \xi^{0+})d\xi \]
and all $O_\C$-terms are of symbol type.
As before, by Lemma \ref{lem:osc}, we infer the bound 
\[ |J(t,x,y)|\lesssim \langle x\rangle^{0+}\langle t\pm x\pm y\rangle^{-1-} \]
and this yields
\begin{align*} |\tilde A_4(t,x)|&\lesssim \langle x\rangle^{0+}\int_\R \langle t\pm x\pm y\rangle^{-1-}
|f(y)|dy \\
&=\langle x\rangle^{0+}\int_\R \langle y\rangle^{-1-}|f(y\mp t\mp x)|dy.
\end{align*}
Thus, Minkowski's inequality implies
\begin{align*} \|A_4(t,x)\|_{L^2_t(\R_+)}&\lesssim \langle x\rangle^{0+}\int_\R \langle y\rangle^{-1-}
\left \|f(y\mp t \mp x)\right \|_{L^2_t(\R)}dy \\
&\lesssim \langle x\rangle^{0+}\|f\|_{L^2(\R_+)}
\end{align*}
and by Fubini we arrive at the desired
\[ \|\langle x\rangle^{-\frac12-}\tilde A_4(t,x)\|_{L^2_t(\R_+)L^2_x(\R_+)}\lesssim \|f\|_{L^2(\R_+)}. \]
This completes the proof for the low-frequency case.

For high frequencies we have to deal with integrals of the form
\begin{align*} B(t,x):&=\int_0^\infty [1-\chi(\xi)]e^{it\xi}e^{\pm i x \xi}O_\C(x^0\xi^0) \\
&\quad \times\int_0^\infty e^{\pm iy\xi}[1+O_\C(\langle y\rangle^{-1}\xi^{-1})]f(y)dy d\xi  \\
&=B_1(t,x)+B_2(t,x)
\end{align*}
where
\begin{align*}
B_1(t,x)&:=\int_0^\infty [1-\chi(\xi)]e^{it\xi}e^{\pm i x\xi}O_\C(x^0 \xi^0)
\int_0^\infty e^{\pm i y\xi}f(y)dy d\xi \\
B_2(t,x)&:=\int_0^\infty [1-\chi(\xi)]e^{it\xi}e^{\pm i x\xi}O_\C(x^0 \xi^0) \\
&\quad \times \int_0^\infty O_\C(\langle y\rangle^{-1}\xi^{-1})e^{\pm i y\xi}f(y)dy d\xi.
\end{align*}
By Plancherel's theorem we find
\[ \|B_1(t,x)\|_{L^2_t(\R_+)}\lesssim \|\mc F_1(1_{[0,\infty)}f)\|_{L^2(\R)}\simeq \|f\|_{L^2(\R_+)}\]
and Fubini yields
\[ \|\langle x\rangle^{-\frac12-}B_1(t,x)\|_{L^2_t(\R_+)L^2_x(\R_+)}\lesssim
\|\langle \cdot\rangle^{-\frac12-}\|_{L^2(\R_+)}\|f\|_{L^2(\R_+)}\lesssim \|f\|_{L^2(\R_+)}. \]
For $B_2$ we simply apply Cauchy-Schwarz to obtain
\[ B_2(t,x)=\|f\|_{L^2(\R_+)}\int_0^\infty [1-\chi(\xi)]e^{it\xi}O_\C(x^0 \xi^{-1})d\xi \]
and as before, by Plancherel and Fubini, we infer
\[ \left \|\langle x\rangle^{-\frac12-}B_2(t,x)\right \|_{L^2_t(\R_+)L^2_x(\R_+)}\lesssim 
\|f\|_{L^2(\R_+)}. \]
This completes the proof for $k=\ell=0$.
If $k+\ell\geq 1$ we obtain an additional weight $O(\xi^{k+\ell})$ and Lemma \ref{lem:derweis}
yields the stated bound.
\end{proof}

\begin{lemma}
\label{lem:locensin}
Under the assumptions of Lemma \ref{lem:locencos} we have the bounds
\begin{align*}
\left \|\langle x \rangle^{-1}
\frac{\sin(t\sqrt A)}{\sqrt A}g(x)\right \|_{L^2_t(\R_+)L^2_x(\R_+)}
&\leq C \left \|\langle \cdot \rangle^{\frac12+} g \right\|_{L^2(\R_+)} \\
\left \|\langle x \rangle^{-1}\frac{\sin(t\sqrt A)}{\sqrt A}g(x)\right \|_{L^2_t(\R_+)H^k_x(\R_+)}
&\leq C_k \left \|\langle \cdot \rangle^{\frac12+} g \right \|_{H^{k-1}(\R_+)}
\end{align*}
for all $k\in \N$ and all $g\in \mc S(\R)$ even.
\end{lemma}

\begin{proof}
Compared to the situation in the proof of Lemma \ref{lem:locencos}, we are missing a factor
$\xi$ now. However, this is exactly compensated by the additional weights and the 
result for the low-frequency case follows from the proof of Lemma \ref{lem:locencos}
by inspection.
For high frequencies the sine evolution here is better by a factor of $\xi^{-1}$ than the
exponential considered in
Lemma \ref{lem:locencos} and the stated result follows.
\end{proof}

\subsection{Bounds involving the scaling vector field}
Next, we consider the estimate involving the scaling vector field $S=t\partial_t+
x\partial_x$.
To this end, it is useful to introduce the operator $\tilde S$, defined by
\[ \mathcal F\left (\tilde S u(t,\cdot)\right )(\xi)=(t\partial_t-\xi\partial_\xi)
\mathcal F(u(t,\cdot))(\xi). \]
Furthermore, from Corollary \ref{cor:E} it follows that 
$S-\tilde S$ is bounded on $H^k(\R_+)$ for any $k\in \N_0$.
Note also that
\[ \tilde S e^{it\sqrt A}f=e^{it\sqrt A}\tilde S f. \]

\begin{lemma}
\label{lem:locencosvf}
Suppose the potential $V$ satisfies Hypothesis \ref{hyp:B}. 
Furthermore, set $Df(x):=xf'(x)$.
Then we have the bound
\[ \left \|\langle x\rangle^{-\frac12-}\partial_t^\ell
S^m e^{it\sqrt A}f(x)\right \|_{L^2_t(\R_+)H^k_x(\R_+)}\leq C_{k,\ell,m}
\sum_{j=0}^m \|D^j f\|_{H^{k+\ell}(\R_+)} \]
for all $k,\ell,m\in \N_0$ and $f\in \mc S(\R)$ even.
\end{lemma}

\begin{proof}
We start with $m=1$ and decompose as $S=\tilde S+(S-\tilde S)$.
Consequently, we infer
\begin{align*} 
Se^{it\sqrt A}f&=e^{it\sqrt A}\tilde S f+(S-\tilde S)e^{it\sqrt A}f \\
&=e^{it\sqrt A}Sf+e^{it\sqrt A}(\tilde S-S)f+(S-\tilde S)e^{it\sqrt A}f
\end{align*}
and Lemma \ref{lem:locencos} and Corollary \ref{cor:E} yield
\[ \left \|\langle x\rangle^{-\frac12-}\partial_t^\ell e^{it\sqrt A}Sf(x)
\right \|_{L^2_t(\R_+)H^k_x(\R_+)}\leq C_{k,\ell}\|Sf\|_{H^{k+\ell}(\R_+)}=\|Df\|_{H^{k+\ell}(\R_+)}
\]
as well as
\begin{align*}
\left \|\langle x\rangle^{-\frac12-}\partial_t^\ell e^{it\sqrt A}(\tilde S-S)f(x)
\right \|_{L^2_t(\R_+)H^k_x(\R_+)}&\leq C_{k,\ell}\|(\tilde S-S)f\|_{H^{k+\ell}(\R_+)} \\
&\leq C_{k,\ell}\|f\|_{H^{k+\ell}(\R_+)}.
\end{align*}
It remains to bound the term $(\tilde S-S)e^{it\sqrt A}f$.
To this end we use 
\begin{align*} 
\tilde S e^{it\sqrt A}f(x)&=\int_0^\infty \phi(x,\xi^2)\tilde \rho(\xi)(t\partial_t-\xi\partial_\xi)
[e^{it\xi}\mc F f(\xi)]d\xi \\
&=\int_0^\infty \phi(x,\xi^2)\tilde \rho(\xi)e^{it\xi}(-\xi\partial_\xi)\mc F f(\xi)d\xi \\
&=\int_0^\infty \left [\xi\partial_\xi+\tfrac{\xi\tilde \rho'(\xi)}{\tilde \rho(\xi)}\right ]
\phi(x,\xi^2)e^{it\xi}\mc F f(\xi)\tilde \rho(\xi)d\xi \\
&\quad +\int_0^\infty \phi(x,\xi^2)\xi\partial_\xi e^{it\xi}\mc F f(\xi)\tilde \rho(\xi)d\xi
+e^{it\sqrt A}f(x) 
\end{align*}
which yields the representation
\begin{align*} (\tilde S-S)e^{it\sqrt A}f(x)&=\int_0^\infty \left [
\xi\partial_\xi-x\partial_x+\tfrac{\xi\tilde \rho'(\xi)}{\tilde \rho(\xi)}\right ]\phi(x,\xi^2)e^{it\xi}
\mc F f(\xi)\tilde \rho(\xi)d\xi \\
&\quad +\int_0^\infty \phi(x,\xi^2)
\underbrace{(\xi\partial_\xi-t\partial_t)e^{it\xi}}_{=0}\mc F f(\xi)\tilde \rho(\xi)d\xi
+e^{it\sqrt A}f(x).
\end{align*}
Note that the function $[\xi\partial_\xi-x\partial_x
+\frac{\xi\tilde \rho'(\xi)}{\tilde \rho(\xi)}]\phi(x,\xi^2)$ satisfies the same 
bounds as $\phi(x,\xi^2)$ and thus, by copying verbatim
the proof of Lemma \ref{lem:locencos},
we find the desired
\[ \left \|\langle x\rangle^{-\frac12-}\partial_t^\ell (\tilde S-S)e^{it\sqrt A}f(x)
\right \|_{L^2_t(\R_+)H^k_x(\R_+)}\leq C_{k,\ell}\|f\|_{H^{k+\ell}(\R_+)}. \]

The general result for arbitrary $m\in \N_0$ follows inductively.
\end{proof}

We conclude the proof of Theorem \ref{thm:locen} by the following estimates for
the sine evolution.

\begin{lemma}
\label{lem:locensinvf}
Suppose the potential $V$ satisfies Hypothesis \ref{hyp:B}. Then,
with $Dg(x):=xg'(x)$, we have
the bounds
\begin{align*}
\left \|\langle x \rangle^{-1}
S^m\frac{\sin(t\sqrt A)}{\sqrt A}g(x)\right \|_{L^2_t(\R_+)L^2_x(\R_+)}
&\leq C_m \sum_{j=0}^m \left \|\langle \cdot \rangle^{\frac12+}D^j g \right\|_{L^2(\R_+)} \\
\left \|\langle x \rangle^{-1}S^m\frac{\sin(t\sqrt A)}{\sqrt A}g(x)\right \|_{L^2_t(\R_+)H^k_x(\R_+)}
&\leq C_{k,m} \sum_{j=0}^m \left \|\langle \cdot \rangle^{\frac12+}D^j g \right \|_{H^{k-1}(\R_+)}
\end{align*}
for all $k,m\in \N$ and all $g\in \mc S(\R)$ even.
\end{lemma}

\begin{proof}
We start with the case $m=1$.
As before, we write $S=\tilde S+(S-\tilde S)$ and note that
\begin{align*} 
\tilde S\frac{\sin(t\sqrt A)}{\sqrt A}g&=\frac{\sin(t\sqrt A)}{\sqrt A}\tilde S g
+\frac{\sin(t\sqrt A)}{\sqrt A}g \\
&=\frac{\sin(t\sqrt A)}{\sqrt A}Sg+\frac{\sin(\sqrt A)}{\sqrt A}(\tilde S-S)g
+\frac{\sin(t\sqrt A)}{\sqrt A}g.
\end{align*}
The terms
\[ \frac{\sin(t\sqrt A)}{\sqrt A}Sg,\quad (S-\tilde S)\frac{\sin(t\sqrt A)}{\sqrt A} \]
are treated as in the proof of Lemma \ref{lem:locencosvf}.
Consequently, it suffices to control
\[ \left \|\langle x\rangle^{-1}\partial_t^\ell \frac{\sin(t\sqrt A)}{\sqrt A}(S-\tilde S)g(x)
\right \|_{L^2_t(\R_+)H^k_x(\R_+)}. \] 
Note that
\begin{align*} 
\mc F\left [(S-\tilde S)g\right ](\xi)&=\int_0^\infty \phi(y,\xi^2)y\partial_y g(y)dy
+\xi\partial_\xi \int_0^\infty \phi(y,\xi^2)g(y)dy  \\
&=\int_0^\infty (\xi\partial_\xi-y\partial_y)\phi(y,\xi^2)g(y)dy+\mc Fg(\xi)
\end{align*}
and thus, we have the representation
\begin{align*} \frac{\sin(t\sqrt A)}{\sqrt A}(S-\tilde S)g(x)&=
\int_0^\infty \int_0^\infty \phi(x,\xi^2)(\xi\partial_\xi-y\partial_y)
\phi(y,\xi^2)\frac{\sin(t\xi)}{\xi} \\
&\quad \times g(y)dy \tilde \rho(\xi)d\xi 
 +\frac{\sin(t\sqrt A)}{\sqrt A}g(x). 
\end{align*}
Since $(\xi\partial_\xi-y\partial_y)\phi(y,\xi^2)$ satisfies the same bounds
as $\phi(y,\xi^2)$, we may apply the logic from the proof of Lemma \ref{lem:locensin}
to bound this term.

The statement for general $m\in \N_0$ is proved inductively.
\end{proof}

\subsection{The inhomogeneous case}
We conclude this section by proving similar estimates for the solution of the 
inhomogeneous problem 
\[
 \left \{ \begin{array}{l}\partial_t^2 u(t,x)-\partial_x^2 u(t,x)+V(x)u(t,x)=\tilde f(t,x),\quad t\geq 0 \\
u(0,\cdot)=\partial_t u(t,\cdot)|_{t=0}=0 \end{array} \right . 
\]
given by the Duhamel formula
\[ u(t,\cdot)=\int_0^t \frac{\sin((t-s)\sqrt A)}{\sqrt A}\tilde f(s,\cdot)ds \]
where $\tilde f(t,\cdot)=\nabla f(t,\cdot)$ or $\tilde f(t,\cdot)=\partial_t f(t,\cdot)$.
These bounds are used in \cite{DonKriSzeWon13}.

\begin{lemma}
\label{lem:loceninhnabla}
Suppose $V$ satisfies Hypothesis \ref{hyp:B} and let $f\in \mc S([0,\infty)\times \R)$.
Furthermore, assume that $f(t,\cdot)$ is odd for any $t\geq 0$ and suppose that
$\partial_t^\ell f(t,\cdot)|_{t=0}=0$ for all $\ell\in \N_0$. Then we have the bound
\begin{align*} 
&\left \|\langle x\rangle^{-\frac12-}\partial_t^\ell S^m \int_0^t \frac{\sin((t-s)\sqrt A)}{\sqrt A}\nabla f(s,\cdot)(x)ds
\right\|_{L^2_t(\R_+)H^k_x(\R_+)} \\
&\qquad \leq C_{k,\ell,m}\sum_{j=0}^m \left \|\partial_t^\ell 
S^j f(t,x)\right \|_{L^1_t(\R_+)H^k_x(\R_+)} 
\end{align*}
for all $k,\ell\in \N_0$.
\end{lemma}

\begin{proof}
First of all, we note that
\[ \partial_t \int_0^t \frac{\sin((t-s)\sqrt A)}{\sqrt A}\nabla f(s,\cdot)ds
=\int_0^t \frac{\sin((t-s)\sqrt A)}{\sqrt A}\nabla \partial_s f(s,\cdot)ds \]
as follows by means of an integration by parts and the assumption $\nabla f(0,\cdot)=0$.
In general, we obtain from $\partial_t^\ell f(t,\cdot)|_{t=0}=0$ the identity
\[ \partial_t^\ell \int_0^t \frac{\sin((t-s)\sqrt A)}{\sqrt A}\nabla f(s,\cdot)ds
=\int_0^t \frac{\sin((t-s)\sqrt A)}{\sqrt A}\nabla \partial_s^\ell f(s,\cdot)ds \]
and it suffices to consider the case $\ell=0$.

As in the proof of Lemma \ref{lem:locencos} we use the explicit representation
\begin{align*}
\int_0^t &\frac{\sin((t-s)\sqrt A)}{\sqrt A}\nabla f(s,\cdot)(x)ds \\
&=\int_0^t \int_0^\infty \int_0^\infty \phi(x,\xi^2)\tilde \rho(\xi)\frac{\sin(t\xi)}{\xi}
\mc F\left (\nabla f(s,\cdot)\right )(\xi)d\xi ds \\
&=-\sum_{j=1}^4 A_j(t,x)-B(t,x)
\end{align*}
where
\begin{align*}
A_1(t,x)&:=\int_0^t \int_0^\infty \int_0^\infty \chi(\xi)\chi(x\xi)\chi(y\xi)
\phi(x,\xi^2)\partial_y \phi(y,\xi^2)\tilde \rho(\xi)\\
&\quad \times \frac{\sin(t\xi)}{\xi}f(s,y)dy d\xi ds \\
A_2(t,x)&:=\int_0^t \int_0^\infty \int_0^\infty \chi(\xi)[1-\chi(x\xi)]\chi(y\xi)
\phi(x,\xi^2)\partial_y \phi(y,\xi^2)\tilde \rho(\xi)\\
&\quad \times \frac{\sin(t\xi)}{\xi}f(s,y)dy d\xi ds \\
A_3(t,x)&:=\int_0^t \int_0^\infty \int_0^\infty \chi(\xi)\chi(x\xi)[1-\chi(y\xi)]
\phi(x,\xi^2)\partial_y \phi(y,\xi^2)\tilde \rho(\xi)\\
&\quad \times \frac{\sin(t\xi)}{\xi}f(s,y)dy d\xi ds \\
A_4(t,x)&:=\int_0^t \int_0^\infty \int_0^\infty \chi(\xi)[1-\chi(x\xi)][1-\chi(y\xi)]
\phi(x,\xi^2)\partial_y \phi(y,\xi^2)\tilde \rho(\xi)\\
&\quad \times \frac{\sin(t\xi)}{\xi}f(s,y)dy d\xi ds
\end{align*}
and
\begin{align*} B(t,x)&:=\int_0^t \int_0^\infty \int_0^\infty [1-\chi(\xi)]
\phi(x,\xi^2)\partial_y \phi(y,\xi^2)\tilde \rho(\xi)\\
&\quad \times \frac{\sin(t\xi)}{\xi}f(s,y)dy d\xi ds.
\end{align*}
Note that we have performed one integration by parts with respect to $y$ in order to move
the derivative from $f$ to $\phi$. The boundary term vanishes since $f(t,0)=0$ for
all $t\geq 0$ by assumption.
For the low-frequency components $A_j$ it suffices to consider $k=0$.
We start with $A_1$.
By Lemmas \ref{lem:rho}, \ref{lem:phi}, and Corollary \ref{cor:phi} we see that it suffices
to bound
\begin{align*} 
\tilde A_1(t,x):=&\int_0^t \int_0^\infty \chi(\xi)\chi(x\xi)e^{i\xi(t-s)}
O(\langle x\rangle^\frac12 \xi^0) \\
&\times \int_0^\infty \chi(y\xi)O(\langle y\rangle^{-\frac12}\xi^0)f(s,y)dy d\xi ds
\end{align*}
and by Cauchy-Schwarz we infer
\[ \tilde A_1(t,x)=\int_0^t \int_0^\infty \chi(\xi)\chi(x\xi)e^{i\xi(t-s)}
O(\langle x\rangle^\frac12 \xi^{0-})\|f(s,\cdot)\|_{L^2(\R_+)}d\xi ds \]
with an $O$-term that behaves like a symbol.
By putting absolute values inside we infer
\[ \tilde A_1(t,x)=\int_0^t O_\C(t^0 s^0 \langle x\rangle^{-\frac12+})\|f(s,\cdot)\|_{L^2(\R_+)}
ds. \]
On the other hand, Lemma \ref{lem:osc} yields
\[ \tilde A_1(t,x)=\int_0^t O_\C(\langle t-s \rangle^{-1+}\langle x\rangle^\frac12)
\|f(s,\cdot)\|_{L^2(\R_+)}ds \]
and by interpolation we obtain
\[ \tilde A_1(t,x)=\int_0^t O_\C(\langle t-s \rangle^{-\frac12-}\langle x\rangle^{0+})
\|f(s,\cdot)\|_{L^2(\R_+)}ds. \]
Thus, Minkowski's inequality yields
\begin{align*}
\|\tilde A_1(t,x)\|_{L^2_t(\R_+)}&\lesssim \langle x\rangle^{0+}\left \|\int_0^\infty \langle t-s\rangle^{-\frac12-}
\|f(s,\cdot)\|_{L^2(\R_+)} ds \right \|_{L^2_t(\R)} \\
&\leq \langle x\rangle^{0+}\int_0^\infty \left \|\langle t-s\rangle^{-\frac12-}\|f(s,\cdot)\|_{L^2(\R_+)}
\right \|_{L^2_t(\R)}ds \\
&\lesssim \langle x\rangle^{0+}\|f(s,x)\|_{L^1_s(\R_+)L^2_x(\R_+)}
\end{align*}
This implies
\[ \left \|\langle x\rangle^{-\frac12-}\tilde A_1(t,x)\right \|_{L^2_t(\R_+)L^2_x(\R_+)}\lesssim
\|f(t,x)\|_{L^1_t(\R_+)L^2_x(\R_+)}. \]
The other terms are handled analogously, along the lines of the proof of Lemma \ref{lem:locencos}.
This settles the case $m=0$.

For $m=1$ one decomposes $S=\tilde S+(S-\tilde S)$ and uses the fact that the operator
$\tilde S$ essentially commutes with the Duhamel integral, cf.~Lemma \ref{lem:SDuhamel}.
One then proceeds analogously to the proof of Lemma \ref{lem:locencosvf}.
The general case $m\geq 1$ is handled inductively.
\end{proof}

\begin{remark}
\label{rem:loceninh}
The same bounds hold if one replaces $\nabla f(s,\cdot)$ by $\partial_s f(s,\cdot)$
(for $f(s,\cdot)$ even).
In this case one first performs an integration by parts with respect to $s$ in order to cancel
the singular $A^{-\frac12}$.
\end{remark}

\section{Bounds for data in divergence form}

\noindent Recall that we had to require $L^1$-bounds on the initial datum $g$ in order to be able
to replace the nonlocal operator $\sqrt A$ by the ordinary derivative $\nabla$, see Lemma \ref{lem:vfodhigh}.
It turns out that this is not necessary if the data are in divergence form, i.e., if $g=\tilde g'$.
In this section we derive the corresponding estimates for this case.
Throughout this section we assume that the potential $V$ satisfies Hypothesis \ref{hyp:B} and, in 
addition, that the coefficient $a_1$ in Lemma \ref{lem:en0} is nonzero (i.e, we restrict ourselves
to the nonresonant case).

\subsection{Bounds for the sine evolution}

It is a general feature of wave equations that the $L^2$-norm of the sine evolution
may grow like $t$ as $t\to \infty$.
This is clearly a low-frequency effect since for high frequencies we immediately obtain
\begin{align*} 
\left \|[1-\chi(A)]\frac{\sin(t\sqrt{A})}{\sqrt{A}}g\right \|_{L^2(\R_+)}&\lesssim
\left \|\frac{\sin(t|\cdot|)}{|\cdot|}\mathcal F g\right \|_{L^2_{\tilde \rho}(1,\infty)}
\lesssim \|\mc{F} g\|_{L^2_{\tilde \rho}(\R_+)}.
\end{align*}
We would like to bound the operator $\nabla \frac{\sin(t\sqrt{A})}{\sqrt{A}}$
in $L^2$ (and, more generally, $H^k$).
In other words, we would like to show that the operator $\nabla A^{-\frac12}$ is bounded
on $L^2$.
This raises problems at small frequencies.
However, it turns out that if the data are in divergence of form, i.e., if one considers
\[ \nabla \frac{\sin(t\sqrt{A})}{\sqrt{A}}(g'), \]
there exists a sufficiently good substitute.
The key observation in this respect is the following result (which is by no means sharp,
but sufficient for our purposes).

\begin{lemma}
\label{lem:lowfreq}
For all $f\in \mathcal S(\R)$ odd and $g\in \mathcal S(\R)$ even we have the bounds
\begin{align*}
\left \|\chi(A)A^{-\frac14}\nabla f\right \|_{L^2(\R_+)}&\lesssim \|f\|_{L^2(\R_+)} \\
\left \|\nabla \chi(A)A^{-\frac14} g\right \|_{L^2(\R_+)}&\lesssim \|g\|_{L^2(\R_+)}.
\end{align*}
\end{lemma}

\begin{proof}
On the Fourier side, $\chi(A)A^{-\frac14}\nabla f$ reads
$\chi(\xi^2)\xi^{-\frac12}\mc{F}(f')(\xi)$
and by Plancherel it suffices to control the $L^2_{\tilde \rho}(\R_+)$-norm of this expression
in terms
of $\|f\|_{L^2(\R_+)}$.
An integration by parts (using $f(0)=0$) yields
\[ \mc{F}(f')(\xi)=-\int_0^\infty \partial_x \phi(x,\xi^2)f(x)dx. \]
Now we distinguish between $x\xi\lesssim 1$ and $x\xi\gtrsim 1$.
In the former case we use the bound
$|\partial_x \phi(x,\xi^2)|\lesssim \langle x\rangle^{-\frac12+}\lesssim 
\langle x\rangle^{-\frac12-}\xi^{0-}$ and Cauchy-Schwarz to obtain
\[ \chi(\xi^2)\int_0^\infty |\chi(x\xi)\partial_x \phi(x,\xi^2)f(x)|dx 
\lesssim \chi(\xi^2)\xi^{0-}\|f\|_{L^2(\R_+)}. \]
In the latter case $x\xi\gtrsim 1$ we infer
\[ \partial_x \phi(x,\xi^2)=O(\xi^{\frac12-})\Re e^{ix\xi}[1+O_\C((x\xi)^{-1})] \]
which yields
\begin{align*}
\chi(\xi^2)&\int_0^\infty [1-\chi(x\xi)]\partial_x \phi(x,\xi^2)f(x)dx \\
&=\chi(\xi^2)O(\xi^{\frac12-})\Re \int_0^\infty e^{ix\xi}f(x)dx 
+\chi(\xi^2)\int_0^{\xi^{-1}}O(x^0\xi^{\frac12-})f(x)dx \\
&\quad +\chi(\xi^2)\int_{\xi^{-1}}^\infty O(x^{-1}\xi^{-\frac12-})f(x)dx
\end{align*}
and the last two terms are bounded by $\chi(\xi^2)\xi^{0-}\|f\|_{L^2(\R_+)}$.
In summary, we infer
\[ \chi(\xi^2)\xi^{-\frac12}\mc{F}(f')(\xi)=\chi(\xi^2)\left [
O(\xi^{0-})F(\xi)+O(\xi^{-\frac12-})\|f\|_{L^2(\R_+)}\right ]
\]
where $F(\xi):=\Re \int_0^\infty e^{ix\xi}f(x)dx$.
Note that by standard, one-dimensional Plancherel we infer $\|F\|_{L^2(\R_+)}\lesssim \|f\|_{L^2(\R_+)}$
and thus, we obtain
\begin{align*}
\left \|\chi(A)A^{-\frac14}\nabla f \right \|_{L^2(\R_+)}&=
\left \|\chi(|\cdot|^2)|\cdot|^{-\frac12}\mc{F}(f')\right \|_{L^2_{\tilde \rho}(\R_+)} \\
&\lesssim \left \|\tilde \rho^\frac12 |\cdot|^{0-}F \right \|_{L^2(0,1)}
+\left \|\tilde \rho^\frac12 |\cdot|^{-\frac12-}\right \|_{L^2(0,1)} \| f\|_{L^2(\R_+)} \\
&\lesssim \|f\|_{L^2(\R_+)}
\end{align*}
since $|\tilde \rho(\xi)|=2|\xi\rho(\xi^2)|\lesssim \xi$ for $\xi\lesssim 1$.

The second bound is obtained similarly by noting that
\[ \nabla \chi(A)A^{-\frac14}g(x)=\int_0^\infty 
\partial_x \phi(x,\xi^2)\chi(\xi^2)\xi^{-\frac12}\mc{F} g(\xi)\tilde \rho(\xi)d\xi. \]
\end{proof}

\begin{lemma}
\label{lem:div}
For all odd $g\in \mathcal S(\R)$ we have the estimate
\[ \left \|\partial_t^\ell \nabla \frac{\sin(t\sqrt{A})}
{\sqrt{A}}(g')\right \|_{H^k(\R_+)}
\leq C_{k,\ell} \|g\|_{H^{1+k+\ell}(\R_+)} \]
for all $t\geq 0$ and $k,\ell \in \N_0$.
\end{lemma}

\begin{proof}
For $\ell\geq 1$ the claim follows from Theorem \ref{thm:energy}.
Thus, it suffices to consider the case $\ell=0$.
We distinguish between high and low frequencies. In the high frequency case we immediately
infer
\begin{align*} 
\left \|
\nabla [1-\chi(A)]\frac{\sin(t\sqrt{A})}{\sqrt{A}}\nabla g 
\right \|_{H^k(\R_+)}&\lesssim \left \|\langle \cdot \rangle^{k+1}\frac{\sin(t|\cdot|)}{|\cdot|}
\mc{F}(g') \right \|_{L^2_{\tilde \rho}(1,\infty)} \\
&\lesssim \|g\|_{H^{1+k}(\R_+)}.
\end{align*}
For low frequencies, on the other hand, it suffices to consider the case $k=0$
and by Lemma \ref{lem:lowfreq} we see that
\[ \nabla \chi(A)\frac{\sin(t\sqrt{A})}{\sqrt{A}}(g')=
\nabla \chi(A)A^{-\frac14}\sin(t\sqrt{A})A^{-\frac14}
\nabla g \]
is a composition of $L^2$-bounded operators acting on $g$.
\end{proof}

\subsection{Bounds involving the scaling vector field}

We turn to the bounds involving the vector field $S=t\partial_t+x\partial_x$.
Recall the definition of $\tilde S$,
\[ \mathcal F\left (\tilde S u(t,\cdot)\right )(\xi)=(t\partial_t-\xi\partial_\xi)
\mathcal F(u(t,\cdot))(\xi), \]
and note that
\[ \tilde S \frac{\sin(t\sqrt{A})}{\sqrt{A}}g=
\frac{\sin(t\sqrt{A})}{\sqrt{A}}(\tilde S g+g). \]
As before, our main concern are low frequencies.
We will need the following auxiliary result.

\begin{lemma}
\label{lem:lowfreq2}
For all $f\in \mathcal S(\R)$ odd and $g\in \mathcal S(\R)$ even we have the bounds
\begin{align*}
\left \|\chi(A)A^{-\frac14}
(S-\tilde S)\nabla f\right \|_{L^2(\R_+)}&\lesssim \|f\|_{L^2(\R_+)} \\
 \left \|\nabla (S-\tilde S)\chi(A) 
A^{-\frac14}g\right \|_{L^2(\R_+)}&\lesssim \|g\|_{L^2(\R_+)}. 
\end{align*}
\end{lemma}

\begin{proof}
By definition, we have
\begin{align*} 
\mc{F}(S \nabla f)(\xi)&=\int_0^\infty \phi(x,\xi^2)x\partial_x f'(x)dx \\
&=\int_0^\infty \partial_x(x\partial_x+1)\phi(x,\xi^2)f(x)dx
\end{align*}
as well as
\begin{align*}
\mc{F}(\tilde S \nabla f)(\xi)&=(t\partial_t-\xi\partial_\xi)\mc{F}(\nabla f)(\xi)
=-\xi\partial_\xi \int_0^\infty \phi(x,\xi^2)f'(x)dx \\
&=\int_0^\infty \partial_x(\xi\partial_\xi)\phi(x,\xi^2)f(x)dx.
\end{align*}
Thus, we obtain the representation
\begin{align*} \mc{F}\big [ \chi(A)&A^{-\frac14}(S-\tilde S)
\nabla f \big ](\xi) \\
&=\chi(\xi^2)\xi^{-\frac12}\int_0^\infty \partial_x (x\partial_x-\xi\partial_\xi+1)
\phi(x,\xi^2)f(x)dx 
\end{align*}
and, since $\partial_x(x\partial_x-\xi\partial_\xi+1)\phi(x,\xi^2)$ has the same asymptotics as
$\partial_x \phi(x,\xi^2)$, the claim follows by repeating the reasoning in the proof of Lemma \ref{lem:lowfreq}.
Similarly, we have
\[ S g(x)=S \mc{F}^{-1}(\mc{F} g)(x)=\int_0^\infty x\partial_x \phi(x,\xi^2)\mc{F} g(\xi)
\tilde \rho(\xi)d\xi \]
as well as
\begin{align*} 
\tilde S g(x)&=\tilde S \mc{F}^{-1}(\mc{F} g)(x)=\int_0^\infty \phi(x,\xi^2)(-\xi\partial_\xi)\mc{F} g(\xi)
\tilde \rho(\xi)d\xi \\
&=\int_0^\infty \left [\xi\partial_\xi+\tfrac{\partial_\xi[\xi\tilde \rho(\xi)]}{\tilde \rho(\xi)}
\right ]\phi(x,\xi^2)g(x)\tilde \rho(\xi)d\xi
\end{align*}
and this yields
\begin{align*} \nabla &(S-\tilde S)\chi(A)A^{-\frac14}g(x) \\
&=\int_0^\infty \partial_x \left [x\partial_x-\xi\partial_\xi
-\tfrac{\partial_\xi[\xi\tilde \rho(\xi)]}{\tilde \rho(\xi)}
\right ]\phi(x,\xi^2)\chi(\xi^2)\xi^{-\frac12}\mc{F} g(\xi)\tilde \rho(\xi)d\xi. 
\end{align*}
Again, $\partial_x [x\partial_x-\xi\partial_\xi
-\tfrac{\partial_\xi[\xi\tilde \rho(\xi)]}{\tilde \rho(\xi)}]\phi(x,\xi^2)$ has the same asymptotics
as $\partial_x \phi(x,\xi^2)$ and the claim follows as in the proof of Lemma \ref{lem:lowfreq}.
\end{proof}

\begin{lemma}
\label{lem:divvf}
For all $g\in \mathcal S(\R)$ odd we have the estimate
\[ \left \|\partial_t^\ell \nabla S^m \frac{\sin(t\sqrt{A})}
{\sqrt{A}}(g') \right \|_{H^k(\R_+)}\leq C_{k,\ell,m}
\sum_{j=0}^m \left \|S^j g \right \|_{H^{1+k+\ell}(\R_+)} \]
for all $t\geq 0$ and $k,\ell,m \in \N_0$.
\end{lemma}

\begin{proof}
By Theorem \ref{thm:vf} it suffices to consider the case $\ell=0$.
We start with the case $m=1$ and write
\begin{equation}
\label{eq:keyest2_2_pr}
\nabla S \frac{\sin(t\sqrt{A})}{\sqrt{A}}\nabla g=
\nabla \tilde S \frac{\sin(t\sqrt{A})}{\sqrt{A}}\nabla g
+\nabla (S-\tilde S) \frac{\sin(t\sqrt{A})}{\sqrt{A}}\nabla g.
\end{equation}
In order to estimate the second term in Eq.~\eqref{eq:keyest2_2_pr} we distinguish between
high and low frequencies.
In the high-frequency case we use the $H^k$-boundedness of $S-\tilde S$ to conclude
\[ \left \|\nabla (S-\tilde S) [1-\chi(A)]
\frac{\sin(t\sqrt{A})}{\sqrt{A}}\nabla g
\right \|_{H^k(\R_+)} \lesssim \|g\|_{H^{1+k}(\R_+)}. \]
For low frequencies it suffices to consider $k=0$ and by Lemmas \ref{lem:lowfreq},
\ref{lem:lowfreq2} we infer
\[ \left \|\nabla (S-\tilde S)\chi(A)A^{-\frac14}
\sin(t\sqrt{A})A^{-\frac14}\nabla g
\right \|_{L^2(\R_+)} \lesssim \|g\|_{L^2(\R_+)}. \]
For the first term in Eq.~\eqref{eq:keyest2_2_pr} we obtain
\begin{align*} 
\nabla \tilde S \frac{\sin(t\sqrt{A})}{\sqrt{A}}\nabla g&=
\nabla \frac{\sin(t\sqrt{A})}{\sqrt{A}}(\tilde S \nabla g+\nabla g) \\
&=\nabla \frac{\sin(t\sqrt{A})}{\sqrt{A}}\left [S \nabla g+(\tilde S-S)
\nabla g+\nabla g\right ]
\end{align*}
and, since $[S,\nabla]=-\nabla$, the claim for $m=1$ follows by Lemmas \ref{lem:lowfreq},
\ref{lem:div}, and \ref{lem:lowfreq2}.
For $m=2$ we decompose
\[ S^2=\tilde S^2+2(S-\tilde S)\tilde S
+(S-\tilde S)^2+[\tilde S,S-\tilde S] \]
and note that $[\tilde S,S-\tilde S]$ is $H^k$-bounded by
Lemma \ref{lem:DB}.
Consequently, we may proceed as in the case $m=1$.
Continuing in this fashion one obtains the stated bound for all $m\in \N_0$ by induction.
\end{proof}

\subsection{Bounds for the inhomogeneous problem}
Finally, we turn to the corresponding bounds on the Duhamel formula.

\begin{lemma}
\label{lem:divinhx}
For all $f\in \mathcal S([0,\infty)\times \R)$ such that $f(t,\cdot)$ is odd for any $t\geq 0$
and $\partial_t^\ell f(t,\cdot)|_{t=0}=0$ for all $\ell \in \N_0$,
we have the bound
\begin{align*} 
\Big \| \partial_t^\ell \nabla S^m &\int_0^t 
\frac{\sin((t-s)\sqrt{A})}{\sqrt{A}}\nabla f(s,\cdot)ds \Big \|_{H^k(\R_+)} \\
&\leq C_{k,\ell,m}\sum_{j=0}^m \int_0^t \left \|\partial_s^\ell S^j f(s,\cdot)
\right \|_{H^{1+k}(\R_+)}ds
\end{align*}
for all $t\geq 0$.
\end{lemma}

\begin{proof}
Since 
\[ \partial_t^\ell \int_0^t \frac{\sin((t-s)\sqrt{A})}{\sqrt{A}}\nabla f(s,\cdot)ds
=\int_0^t \frac{\sin((t-s)\sqrt{A})}{\sqrt{A}}\nabla \partial_s^\ell f(s,\cdot)ds\]
it suffices to consider the case $\ell=0$.
For $m=0$ we have
\begin{align*} \Big \|\nabla \int_0^t 
&\frac{\sin((t-s)\sqrt{A})}{\sqrt{A}}\nabla f(s,\cdot)ds \Big \|_{H^k(\R_+)} \\
&\lesssim \int_0^t \Big \|\nabla
\frac{\sin((t-s)\sqrt{A})}{\sqrt{A}}\nabla f(s,\cdot)\Big \|_{H^k(\R_+)} 
ds \\
&\lesssim \int_0^t \|f(s,\cdot)\|_{H^{1+k}(\R_+)}ds
\end{align*}
by Lemma \ref{lem:div}.
The case $m\geq 1$ follows from Lemma \ref{lem:divvf} once we understand the action
of $\tilde S$ on the Duhamel formula. 
On the Fourier side we have
\begin{align*} 
(t\partial_t-\xi\partial_\xi-1)&\int_0^t \frac{\sin((t-s)\xi)}{\xi}\mc{F}(\nabla f(s,\cdot))(\xi)ds \\
&=\int_0^t \frac{\sin((t-s)\xi)}{\xi}(s\partial_s-\xi\partial_\xi)\mc{F}(\nabla f(s,\cdot))(\xi)ds 
\end{align*}
as follows by performing one integration by parts with respect to $s$ (and the assumption
$f(0,\cdot)=0$).
This means that essentially, the operator $\tilde S$ commutes through the Duhamel formula 
and the claim follows.
\end{proof}

\begin{lemma}
\label{lem:divinhs}
For all $f\in \mathcal S([0,\infty)\times \R)$ such that $f(t,\cdot)$ is even for any $t\geq 0$
and $\partial_t^\ell f(t,\cdot)|_{t=0}=0$ for all $\ell \in \N_0$,
we have the bound
\begin{align*} 
\Big \| \partial_t^\ell \nabla S^m &\int_0^t 
\frac{\sin((t-s)\sqrt{A})}{\sqrt{A}}\partial_s f(s,\cdot)ds \Big \|_{H^k(\R_+)} \\
&\leq C_{k,\ell,m}\sum_{j=0}^m \int_0^t \left \|\partial_s^\ell S^j f(s,\cdot)
\right \|_{H^{1+k}(\R_+)}ds 
\end{align*}
for all $t\geq 0$.
\end{lemma}

\begin{proof}
Integration by parts with respect to $s$ yields
\begin{align*}
\partial_t^\ell \int_0^t 
\frac{\sin((t-s)\sqrt{A})}{\sqrt{A}}\partial_s f(s,\cdot)ds=
\int_0^t \cos\left ((t-s)\sqrt{A}\right)\partial_s^\ell f(s,\cdot)ds
\end{align*}
and for $m=0$ the claim follows in a straightforward manner.
In order to handle the case $m\geq 1$ it suffices to note the commutator
$[\partial_t,S]=\partial_t$ and to recall the formula
\begin{align*} 
(t\partial_t-\xi\partial_\xi-1)&\int_0^t \frac{\sin((t-s)\xi)}{\xi}\partial_s \mc{F}(f(s,\cdot))(\xi)ds \\
&=\int_0^t \frac{\sin((t-s)\xi)}{\xi}(s\partial_s-\xi\partial_\xi)\partial_s
\mc{F}(f(s,\cdot))(\xi)ds.
\end{align*}
\end{proof}

\bibliographystyle{plain}
\bibliography{gen}

\end{document}